\pdfoutput=1
\documentclass{amsart}

\usepackage[mathscr]{eucal}
\usepackage{amssymb}
\usepackage[usenames,dvipsnames]{xcolor} 
\usepackage[normalem]{ulem}
\usepackage{wasysym}
\usepackage{amsthm}
\usepackage{bbold}
\usepackage{enumitem}
\usepackage{amsmath}
\usepackage{tikz-cd}
\usetikzlibrary{cd}
\usepackage{stmaryrd} 
\usepackage{etoolbox} 
\usepackage{microtype}
\usepackage{adjustbox}
\usepackage{ mathdots }
\usepackage{bbm}
\usepackage{mathtools}
\usetikzlibrary{decorations.pathmorphing} 

\usepackage{tcolorbox}

\usepackage{comment}
\excludecomment{strongfiltrations}

\usepackage[unicode]{hyperref} 

\definecolor{dark-red}{rgb}{0.5,0.15,0.15}
\definecolor{dark-blue}{rgb}{0.15,0.15,0.6}
\definecolor{dark-green}{rgb}{0.15,0.6,0.15}

\hypersetup{
    colorlinks, linkcolor=OliveGreen,
    citecolor=OliveGreen, urlcolor=OliveGreen
}

\usepackage[nameinlink,capitalise,noabbrev]{cleveref}

\numberwithin{equation}{section}
\usepackage[all]{xy}
\xyoption{line}
\usepackage{graphicx}
\usepackage{float}
\usepackage{caption}

\newtheorem{thmx}{Theorem}

\newtheorem{Thm}[equation]{Theorem}
\newtheorem*{Thm*}{Theorem}
\newtheorem*{MainThm*}{Main Theorem}
\newtheorem{Prop}[equation]{Proposition}
\newtheorem{Lem}[equation]{Lemma}
\newtheorem{Cor}[equation]{Corollary}

\newtheorem*{Que*}{Question}

\theoremstyle{remark}
\newtheorem{Def}[equation]{Definition}

\newtheorem{Not}[equation]{Notation}
\newtheorem{Exa}[equation]{Example}

\newtheorem{Cons}[equation]{Construction}

\newtheorem{Hyp}[equation]{Hypothesis}
\newtheorem{Rec}[equation]{Recollection}

\newtheorem{Rem}[equation]{Remark}

\tikzset{
    labelrotatebelow/.style={anchor=north, rotate=90, inner sep=1.0mm}
}
\tikzset{
    labelrotateabove/.style={anchor=south, rotate=90, inner sep=1.0mm}
}
\SelectTips{cm}{10}

\newcommand{\nc}{\newcommand}
\nc{\dmo}{\DeclareMathOperator}
\renewcommand{\emptyset}{\varnothing}

\usepackage{todonotes}

\nc{\overbar}[1]{\mkern 1.5mu\overline{\mkern-1.5mu#1\mkern-1.5mu}\mkern 1.5mu}

\nc{\kappaaux}{g}
\nc{\kappaCh}{{\kappaaux(\cat C_h)}}
\nc{\kappam}{{\kappaaux({\mathfrak m})}}
\nc{\kappaP}{\Gamma_{\cat P}\unit}
\nc{\kappaQ}{{\kappaaux(\cat Q)}}
\nc{\kappaCP}{{\kappaaux_{\cat C}(\cat P)}}
\nc{\kappaDP}{{\kappaaux_{\cat D}(\cat P)}}
\nc{\kappaCQ}{{\kappaaux_{\cat C}(\cat Q)}}
\nc{\kappaDQ}{{\kappaaux_{\cat D}(\cat Q)}}
\nc{\kappaphiB}{{\kappaaux(\phi(\cat B))}}
\nc{\kappaphiQ}{{\kappaaux(\varphi(\cat Q))}}

\dmo{\Sub}{Sub}
\dmo{\Proj}{Proj}
\dmo{\LMod}{LMod}
\dmo{\cell}{cell}
\nc{\Prst}{{\cat P}\mathrm{r^{st}}}
\nc{\Mack}[2]{\mathrm{Mack}_{#1}(#2)}
\dmo{\fin}{{fin}}
\dmo{\Sphere}{\mathbb{S}}
\dmo{\Alg}{Alg}
\dmo{\CAlg}{CAlg}
\nc{\HA}{{\rmH \hspace{-0.2em}\bbA}}
\nc{\HZ}{{\rmH \hspace{-0.2em}\bbZ}}
\nc{\HZbar}{{\rmH \hspace{-0.2em}\underline{\bbZ}}}
\nc{\Fp}{{\bbF_{\hspace{-0.1em}p}}}
\nc{\HFp}{{\rmH \hspace{-0.15em}\bbF_{\hspace{-0.1em}p}}}

\nc{\mathfrakp}{\mathfrak{p}}
\nc{\mathfrakq}{\mathfrak{q}}
\nc{\mathfrakS}{\mathfrak{S}}
\nc{\mathfrakT}{\mathfrak{T}}
\nc{\Z}{\mathbb{Z}}
\nc{\cF}{\mathcal{F}}
\nc{\hspec}[1]{\Spc^\mathrm{h}({#1})}

\dmo{\Id}{Id}
\dmo{\Loc}{Loc}
\dmo{\Spc}{Spc}
\nc{\thickid}{\mathrm{thickid}}
\nc{\thick}[1]{\mathrm{thick}\langle #1 \rangle}
\nc{\thickt}[1]{\mathrm{thick}_\otimes\langle #1 \rangle}
\nc{\loct}[1]{\mathrm{Loc}_\otimes\langle #1 \rangle}
\nc{\loc}[1]{\mathrm{Loc}\langle #1 \rangle}
\dmo{\End}{End}
\dmo{\Mor}{Mor}
\dmo{\Hom}{Hom}
\dmo{\id}{id}
\dmo{\im}{im}
\dmo{\Ker}{Ker}
\dmo{\ind}{ind}
\dmo{\Ind}{Ind}
\dmo{\CoInd}{coind}
\dmo{\res}{res}
\dmo{\infl}{infl}
\dmo{\triv}{triv}
\dmo{\Tel}{Tel} 
\dmo{\grMod}{grMod}%
\dmo{\Mod}{Mod}%
\dmo{\opname}{op}
\dmo{\SH}{\mathcal{S}\mathcal{H}}
\dmo{\smallb}{b}
\dmo{\Spec}{Spec}
\dmo{\supp}{supp}
\dmo{\Supp}{Supp}
\dmo{\cosupp}{cosupp}
\dmo{\Cosupp}{Cosupp}
\dmo{\hsupp}{hsupp}
\dmo{\esupp}{esupp}
\dmo{\qsupp}{qsupp}
\dmo{\eqsupp}{eqsupp}
\dmo{\Pinj}{Pinj}
\dmo{\Inj}{Inj}
\renewcommand{\mod}{\mathrm{mod}}

\nc{\bbL}{\mathbb{L}}
\nc{\bbA}{\mathbb{A}}
\nc{\bbE}{\mathbb{E}}
\nc{\bbN}{\mathbb{N}}
\nc{\bbQ}{\mathbb{Q}}
\nc{\bbZ}{\mathbb{Z}}
\nc{\bbF}{\mathbb{F}}
\nc{\bbS}{\mathbb{S}}
\nc{\cat}[1]{\mathscr{#1}}

\nc{\cA}{\mathcal{A}}
\nc{\cB}{\mathcal{B}}
\nc{\cC}{\mathcal{C}}
\nc{\cD}{\mathcal{D}}
\nc{\cE}{\mathcal{E}}
\nc{\cU}{\mathcal{U}}

\nc{\CB}{\mathsf{CB}}

\nc{\sD}{\mathsf{D}}

\nc{\ihom}{{\underline{\hom}}}
\nc{\iHom}{\mathcal{H}\mathrm{om}}

\nc{\Mid}{\,\big|\,}
\nc{\SET}[2]{\big\{\,#1\Mid#2\,\big\}}
\nc{\unit}{\mathbb{1}}
\nc{\xra}{\xrightarrow}

\dmo{\Sp}{Sp}
\dmo{\Ho}{Ho}
\dmo{\Fin}{Fin}
\dmo{\add}{add}
\dmo{\Fun}{Fun}
\dmo{\Ext}{Ext}

\dmo{\Map}{Map}
\dmo{\Span}{Span}
\dmo{\N}{N}
\dmo{\Cat}{Cat}
\dmo{\colim}{colim}
\dmo{\hocolim}{hocolim}
\dmo{\Ch}{Ch}
\dmo{\Gr}{Gr}
\nc{\CA}{\cat A}
\nc{\CU}{\cat U}

\dmo{\Ab}{Ab}
\dmo{\Set}{Set}
\dmo{\ev}{ev}
\dmo{\Spcl}{Spcl}
\nc{\Funadd}{\Fun_{\add}}
\dmo{\proj}{proj}
\dmo{\cof}{cof}
\dmo{\deftensor}{Def^\otimes}
\dmo{\defcoid}{Def^{coid}}

\dmo{\dual}{dual}

\dmo{\Perf}{Perf}
\dmo{\tel}{tel}

\dmo{\rk}{rk}

\dmo{\FIdl}{\mathrm{Thick}_\otimes^{\mathrm{fg}}}
\dmo{\Idl}{\mathrm{Thick}_\otimes}
\dmo{\PIdl}{Spec}
\dmo{\Fd}{Fd}
\nc{\hatS}{\widehat{\mathcal{S}}}
\nc{\CS}{\mathcal{S}}
\dmo{\LS}{LS}
\nc{\hatLS}{\widehat{\mathrm{LS}}}

\dmo{\glo}{glo}
\dmo{\Pic}{Pic}
\dmo{\gl}{gl}
\nc{\fan}[1]{\mathscr{F}_{#1}}
\dmo{\GL}{GL}
\dmo{\fib}{fib}
\dmo{\Out}{Out}
\dmo{\Fan}{Fan}
\dmo{\cons}{cons}
\dmo{\Cov}{Cov}
\newcommand{\Spgl}[1]{\Sp_{#1\text{-}\gl}}

\newcommand{\Glo}[1]{\mathrm{Glo}_{#1}}
\newcommand{\Orb}[1]{\mathrm{Orb}_{#1}}
\newcommand{\Epi}[1]{\mathrm{Epi}_{#1}}
\newcommand{\PrL}{\mathrm{Pr}^{\mathrm{L}}_{\mathrm{st}}}

\renewcommand{\leq}{\leqslant}
\renewcommand{\geq}{\geqslant}

\nc{\ua}{\mathord{\uparrow}}

\newcommand{\laxlim}    {\operatorname*{laxlim}}
\newcommand{\laxlimdag}    {\operatorname*{laxlim^\dagger}}

\newcommand{\sfD}{\mathsf{D}}
\newcommand{\sfA}{\mathsf{A}}

\DeclareMathOperator{\op}{op}

\newcommand{\A}[1]{\mathsf{A}(#1)} 
\newcommand{\D}[1]{\mathsf{D}(#1)} 
\newcommand{\tCFG}{\widetilde{\cat{FG}}}
\newcommand{\CFG}{\cat{FG}}
\newcommand{\hCG}	{\widehat{\cat{FG}}}
\newcommand{\hCU}	{\widehat{\cat{U}}}
\newcommand{\fab}{\cat{A}} 
\newcommand{\fabp}{\cat{A}(p)} 
\newcommand{\fabpex}[1]{\cat{A}(p)^{\leq #1}} 
\newcommand{\fabprk}[1]{\cat{A}(p)_{\leq #1}} 
\newcommand{\fabrk}[1]{\cat{A}_{\leq #1}} 
\newcommand{\hfabprk}[1]{\widehat{\cat{A}}(p)_{\leq #1}} 

\newcommand{\Cprime}{\cat C^{\mathrm{pr}}}

\newcommand{\bbm}       {\left[\begin{matrix}}
\newcommand{\ebm}       {\end{matrix}\right]}
\newcommand{\bsm}       {\left[\begin{smallmatrix}}
\newcommand{\esm}       {\end{smallmatrix}\right]}

\newcommand{\img}       {\operatorname{image}}

\newcommand{\ip}[1]     {\langle #1\rangle}
\newcommand{\asym}	{\operatorname{Asym}}

\newcommand{\pri}       {\mathfrak{p}}

\newcommand{\VI}{\mathsf{VI}}
\newcommand{\VIMod}{\VI\mathrm{-Mod}}
\DeclareMathOperator{\type}{type}

\makeatletter
\newcommand{\smallbullet}{} 
\DeclareRobustCommand\smallbullet{%
  \mathord{\mathpalette\smallbullet@{0.7}}%
}
\newcommand{\smallbullet@}[2]{%
  \vcenter{\hbox{\scalebox{#2}{$\m@th#1\bullet$}}}%
}
\makeatother
\makeatletter
\newcommand{\largebullet}{} 
\DeclareRobustCommand\largebullet{%
  \mathord{\mathpalette\largebullet@{1.5}}%
}
\newcommand{\largebullet@}[2]{%
  \vcenter{\hbox{\scalebox{#2}{$\m@th#1\bullet$}}}%
}

\newcommand*\emptycirc[1][1ex]{\tikz\draw[thick] (0,0) circle (#1);} 
\newcommand*\halfcirc[1][1ex]{%
  \begin{tikzpicture}
  \draw[fill] (0,0)-- (-90:#1) arc (-90:90:#1) -- cycle;
  \draw[thick] (0,0) circle (#1);
  \end{tikzpicture}}

\makeatother

\newcounter{enum-resume-hack}
\Crefname{Thm}{Theorem}{Theorems}
\Crefname{Prop}{Proposition}{Propositions}
\Crefname{Lem}{Lemma}{Lemmas}
\Crefname{thmx}{Theorem}{Theorems}
\begin{document}

\title[The spectrum for families of bounded rank and VI-modules]{The spectrum of global representations \\ \smaller{} for families of bounded rank and VI-modules}

\author[Barrero]{Miguel Barrero}
\author[Barthel]{Tobias Barthel}
\author[Pol]{Luca Pol}
\author[Strickland]{Neil Strickland}
\author[Williamson]{Jordan Williamson}

\date{\today}

\makeatletter
\patchcmd{\@setaddresses}{\indent}{\noindent}{}{}
\patchcmd{\@setaddresses}{\indent}{\noindent}{}{}
\patchcmd{\@setaddresses}{\indent}{\noindent}{}{}
\patchcmd{\@setaddresses}{\indent}{\noindent}{}{}
\makeatother

\address{Miguel Barrero, Department of Mathematics, University of Aberdeen, Fraser Noble Building, Aberdeen AB24 3UE, UK}
\email{miguel.barrero@abdn.ac.uk}
\urladdr{https://sites.google.com/view/mbarrero}

\address{Tobias Barthel, Max Planck Institute for Mathematics, Vivatsgasse 7, 53111 Bonn, Germany}
\email{tbarthel@mpim-bonn.mpg.de}
\urladdr{https://sites.google.com/view/tobiasbarthel/}

\address{Luca Pol, Fakult{\"a}t f{\"u}r Mathematik, Universit{\"a}t Regensburg, Universit{\"a}tsstraße 31, 93053 Regensburg, Germany}
\email{luca.pol@mathematik.uni-regensburg.de}
\urladdr{https://sites.google.com/view/lucapol/}

\address{Neil P. Strickland,
 School of Mathematical and Physical Sciences,
 Hicks Building, 
 Sheffield S3 7RH, 
 UK
}
\email{N.P.Strickland@sheffield.ac.uk}
\urladdr{https://strickland1.org}

\address{Jordan Williamson, Department of Algebra, Faculty of Mathematics and Physics, Charles University in Prague, Sokolovsk\'{a} 83, 186 75 Praha, Czech Republic}
\email{williamson@karlin.mff.cuni.cz}
\urladdr{https://jordanwilliamson1.github.io/}

\begin{abstract}
A global representation is a compatible collection of representations of the outer automorphism groups of the finite groups belonging to a family $\mathscr{U}$. These arise in classical representation theory, in the study of representation stability, as well as in global homotopy theory. In this paper we begin a systematic study of the derived category $\mathsf{D}(\mathscr{U};k)$ of global representations over fields $k$ of characteristic zero, from the point-of-view of tensor-triangular geometry. We calculate its Balmer spectrum for various infinite families of finite groups including elementary abelian $p$-groups, cyclic groups, and finite abelian $p$-groups of bounded rank. We then deduce that the Balmer spectrum associated to the family of finite abelian $p$-groups has infinite Krull dimension and infinite Cantor--Bendixson rank, illustrating the complex phenomena we encounter. As a concrete application, we provide a complete tt-theoretic classification of finitely generated derived VI-modules. Our proofs rely on subtle information about the growth behaviour of global representations studied in a companion paper, as well as novel methods from non-rigid tt-geometry. 
\end{abstract}



\maketitle

\vspace{-4ex}
\setcounter{tocdepth}{1}
\tableofcontents

\section*{Introduction}\label{sec:intro}
This paper initiates the study of \emph{global representation theory} from the perspective of tensor-triangular geometry, and is the second in a series of papers starting with \cite{BBPSWstructural}, which builds on \cite{PolStrickland2022}.

\subsection*{Overview and main results}\label{ssec:intro_overview}

A global representation over a fixed field $k$ is a compatible collection of representations for the outer automorphism group of each finite group belonging to a family $\cat U$. By viewing $\cat U$ as a category with morphisms given by conjugacy classes of surjective group homomorphisms, we can organize this data into a convenient abelian category of global representations 
\[
\mathsf{A}(\cat U;k)\coloneqq \Fun(\cat U^{\op},\Mod{k})
\]
which simultaneously generalizes:
    \begin{itemize}
        \item \emph{ordinary representation theory}, since any finite group can be realized as the outer automorphism group of a finite group;
        \item \emph{the category of VI-modules} appearing in the representation theory of the finite general linear groups, which can be recovered by taking $\cat U$ to be the family of elementary abelian $p$-groups;
        \item \emph{rational global homotopy theory} for global families $\cat U$.
\end{itemize}
Despite the ubiquity of global representations as witnessed by the above examples, the general structure of $\mathsf{A}(\cat U;k)$ remains mysterious. For example, a full classification of isomorphism types of finitely generated global representations appears to be out of reach even for simple infinite families $\cat U$. Based on the homological algebra established in the first paper \cite{BBPSWstructural} in this series, the focus of this paper is therefore the geometry and global structure of the derived category of $\mathsf{A}(\cat U;k)$: 
    \[
        \sfD(\cat U;k)\coloneqq\sfD(\mathsf{A}(\cat U;k)) \simeq \Fun(\cat U^{\op},\sfD(k)),
    \]
which forms a tensor-triangulated category via the pointwise tensor product of global representations. Our perspective on $\sfD(\cat U;k)$ is geometric in the sense of tensor-triangular geometry \cite{Balmer2005} as controlled by the Balmer spectrum $\Spc(\sfD(\cat U;k)^c)$. Throughout we will work over a fixed field $k$ of characteristic zero and we will usually drop it from the notation.

\subsubsection*{Classification of derived VI-modules}

In order to first illustrate our results in an important special case, consider the abelian category $\VIMod$ of VI-modules, i.e., functors on the category of finite dimensional $\bbF_p$-vector spaces and injective homomorphisms with values in $\Mod{k}$. In other words, VI-modules are compatible systems of $\mathrm{GL}_n(\bbF_p)$-representations; as such, they appear both in the study of representation stability and as admissible representations of the infinite general linear group $\mathrm{GL}_{\infty}(\bbF_p)$. The structure of this category has been studied in detail by Putman--Sam \cite{PutmanSam2017}, Nagpal \cite{Nagpal1,Nagpal2}, and others, but a full classification of isomorphism classes of finitely generated VI-modules seems to be out of reach. 

In this context, our results admit a particularly explicit and clean formulation. To explain this, define the \emph{type} of a VI-module $M$ as 
    \[
        \type(M) = \SET{n \in \bbN}{M(\bbF_p^{\oplus n}) \neq 0}.
    \]
This definition extends naturally to the derived category $\sfD(\VIMod)$, which forms a tensor-triangulated category under pointwise tensor product. Finally, say two (derived) VI-modules $M$ and $N$ are tt-equivalent if they can be built from each other using triangles, retracts, and tensoring; that is, $M$ and $N$ are tt-equivalent if they generate the same thick ideal in $\sfD(\VIMod)$. As an application of our methods, we are able to classify all tt-equivalence classes of (derived) VI-modules:

\begin{thmx}\label{thmx:VImod}
    The assignment $M \mapsto \type(M)$ induces a bijection
        \[
            \begin{tikzcd}[ampersand replacement=\&,column sep=small]
                {\begin{Bmatrix}
                    \text{tt-equivalence classes } \langle M\rangle_{\otimes}\colon \\
                    M \in \sfD(\VIMod) \text{ compact}
                \end{Bmatrix}}
                    \& 
                {\begin{Bmatrix}
                    \text{cofinite subsets of }\\
                    \text{of } \bbN, \text{ or } \emptyset
                \end{Bmatrix}}.
                    \arrow["\sim", from=1-1, to=1-2]
            \end{tikzcd}
        \]
\end{thmx}
The derived category of VI-modules can be realized as $\sfD(\cat E_p)$ for the family $\cat E_p$ of elementary abelian $p$-groups; under this identification, \cref{thmx:VImod} translates into a computation of the spectrum $\Spc(\sfD(\cat E_p)^c)$ which we carry out in \cref{thm-elementary}.

\subsubsection*{The family of abelian $p$-groups of bounded $p$-rank}
As indicated above, classification results for more general families $\cat{U}$ of finite groups rely fundamentally on a tt-geometric approach to $\sfD(\cat U)$  via its spectrum. The Balmer spectrum $\Spc(\sfD(\cat U)^c)$ serves a dual purpose: on the one hand, it provides a classification of the thick ideals of compact derived global representations of $\cat U$. On the other hand, it offers a geometric point of view on the derived category: geometric operations on the topological space $\Spc(\sfD(\cat U)^c)$ are reflected in categorical operations on $\sfD(\cat U)$. For instance, open-closed decompositions of the spectrum manifest themselves as semi-orthogonal decompositions of the category.

An important example is given by the family $\fabprk{r}$ of abelian $p$-groups of $p$-rank at most $r$; in particular, when $r=1$, we get the family of cyclic $p$-groups $\fabprk{1} = \cat C_p$. \Cref{Thm-abelian-p-groups-rank-r} provides a complete description of the spectrum: 

\begin{thmx}\label{thmx:boundedrank}
    Fix a prime number $p$ and a positive integer $r$, and write $\fabprk{r}$ for the family of abelian $p$-groups of $p$-rank at most $r$. Let $\bbN^+$ be the one-point compactification of $\bbN$. The spectrum $\Spc(\sfD(\fabprk{r})^c)$ is Hausdorff and homeomorphic to 
        \[
            \hatS_{\leq r}=\{v=(v_1,\ldots, v_r)\in (\bbN^+)^r \mid \infty \geq v_1 \geq \ldots \geq v_r \geq 0\},
        \]
    where $\hatS_{\leq r}$ inherits the subspace topology from $(\bbN^+)^r$. Furthermore, the Cantor--Bendixson rank of the spectrum is precisely $r$.
\end{thmx}

Here, the Cantor--Bendixson rank is an ordinal-valued measure of the complexity of a topological space and its internal structure; see \cref{rec:CBrank} for a definition. 

\subsubsection*{Towards the family of abelian $p$-groups}

Ultimately, one of our goals is to give a complete description of the spectrum for the family of finite abelian groups. Although this seems currently out of reach, as a ``meta-theorem'' collecting several of the results proven in this paper, the next theorem provides ample evidence of the intricate phenomena that we encounter:

\begin{thmx}\label{thmx:main}
   Fix a prime number $p$ and let $\cat A(p)$ be the family of finite abelian $p$-groups. The Balmer spectrum $\Spc(\sfD(\cat A(p))^c)$ has infinite Krull dimension and infinite Cantor–Bendixson rank.
\end{thmx}

In comparison with the tensor-triangular geometry of representations or rational equivariant spectra for a fixed group, we find the complexity appearing $\Spc(\sfD(\cat A(p))^c)$ rather striking. We offer a moral explanation for this richness further below in the introduction, through the lens of non-rigid tt-geometry. The proof of \cref{thmx:main} relies on a study of different subfamilies of $\cat A(p)$, filtering either by $p$-rank or by $p$-exponent. 
For any $r\geq 0$, we have an embedding on Balmer spectra
    \[
        \Spc(\sfD(\fabprk{r})^c) \to\Spc(\sfD(\fabp)^c).
    \]
In light of \cref{thmx:boundedrank}, we conclude that $\Spc(\sfD(\fabp)^c)$ has infinite Cantor--Bendixson rank, thereby verifying the second part of \cref{thmx:main}. A sketch explanation for the infinite Krull dimension of $\Spc(\sfD(\cat A(p))^c)$ is given after \cref{thmx:familyprimes}. Finally, we remark that we have a conjectural description of the set underlying the spectrum $\Spc(\sfD(\fabp)^c)$; we plan to return to this problem in a future work. With that in mind, in fact, the overall goal of this paper is twofold:
    \begin{enumerate}
        \item study the abstract tt-geometry of $\Spc(\sfD(\cat U)^c)$ in general;
        \item describe the spectrum completely for various important families of finite groups (abelian and non-abelian). 
    \end{enumerate}

\subsection*{Interlude: Global homotopy theory}
The inspiration to our approach to understanding the tensor-triangular geometry of global representations comes from global homotopy theory. Equivariant homotopy theory studies algebraic invariants of topological spaces equipped with an action of a fixed group $G$. Global homotopy theory provides a framework which captures such equivariant structures that exist \emph{globally}, that is in a uniform way for all groups in a collection $\cat U$ of groups. In particular, Schwede~\cite{Schwedebook} introduced a category of \emph{global spectra} $\Spgl{\cat U}$ in which global cohomology theories are representable, in analogy to how the classical stable homotopy category encodes cohomology theories for (non-equivariant) topological spaces. The category $\Spgl{\cat U}$ can be constructed from assembling the categories $\Sp_G$ of $G$-equivariant spectra for all $G \in \cat U$, in a sense made precise in \cite{LNP}.

Discarding higher chromatic structures while retaining much of the group-theoretic information, \emph{rational global spectra} can be thought of as global cohomology theories taking values in $\bbQ$-vector spaces. Our next theorem refines a result of Schwede~\cite[Equation 4.5.36]{Schwedebook} and Wimmer~\cite[Theorem 3.2.20]{Wimmerthesis}, and shows that there is a deep connection between global representation theory and global homotopy theory, see \cref{thm-algebraic-model}; however, for the purposes of the present paper, we emphasize that this connection is only a source of motivation for our approach.

\begin{thmx}\label{thmx:globalalgmodel}
Let $\cat{U}$ be a family of finite groups which is closed under subgroups and quotients. There is a symmetric monoidal equivalence of $\infty$-categories 
\[
\Phi \colon \Spgl{\cat{U}}^\mathbb{Q} \xrightarrow{\sim} \mathsf{D}(\cat{U};\bbQ)
\]
which is compatible with the geometric fixed points functors in the sense that for all $G \in \cat U$ and $X \in \Spgl{\cat U}^\bbQ$, we have $\pi_*(\Phi^G X)\cong H_*(\Phi(X))(G)$.
\end{thmx}

This connection to equivariant homotopy theory suggests an approach to understanding the tensor-triangular geometry of $\Spgl{\cat{U}}$ one group at a time in line with the strategies applied in~\cite{BalmerSanders17,BHNNNS2019,BGH2020,BCHNP}.

\subsection*{Methodology}\label{ssec:intro_methodology}
The conventional approach to understanding the Balmer spectrum, which works in virtually all known examples, is by probing a given rigidly-compactly generated tt-category $\cat T$ through a jointly (nil-)conservative family of tt-functors $f_i^*\colon \cat{T} \to \cat{S}_i$ for $\cat{S}_i$ that are better understood. For example in equivariant stable homotopy theory, one is naturally led to consider the family of geometric fixed points functors, while in algebra one examines the base-change functors to the residue fields. The underlying reason that this works so well is that such a family induces a surjective map on spectra 
    \[
        \xymatrix{\bigsqcup_{i \in I}\Spc(\cat{S}_i^c) \ar@{->>}[r] & \Spc(\cat{T}^c),}
    \]
see \cite{BCHS2024}. This provides a powerful tool for describing all the points of the Balmer spectrum, while the remaining challenge lies in understanding the containment relations between primes and the topology of the spectrum.

Translated to global representation theory via \cref{thmx:globalalgmodel}, we specialize this ansatz by taking the jointly conservative family of evaluation functors 
    \[
        (\ev_G\colon\sfD(\cat{U}) \to \sfD(\{G\}), \quad X \mapsto X(G))_{G \in \pi_0 \cat U}.
    \]
The resulting map on spectra 
    \begin{equation}\tag{$\dagger$}\label{eq:intro_groupprimemap}
        \begin{split}
            \pi_0 \cat U  =  \bigsqcup_{G\in\pi_0\cat U}\Spc(\sfD(\{G\})^c) & \to \Spc(\D{\cat U}^c), \\
            G & \mapsto \mathfrakp_G = \{X \mid H_*(X(G))=0\}
        \end{split}
    \end{equation}
gives rise to a collection of prime ideals in $\D{\cat U}^c$ which we refer to as \emph{group primes}. We show that this map is always injective; based on the above discussion, one might optimistically expect this map to also be surjective. The special case where the family $\cat U$ is essentially finite has already been studied in the literature \cite{Xu2014,Wang2019}. Here, the expectation turns out to be correct: the Balmer spectrum coincides with the finite discrete space $\pi_0\cat U$, see \cref{thm:spcfinite}. Surprisingly, we prove in \cref{cor:spcfinitep_characterization} that for finite $p$-groups, finiteness is in fact also \emph{necessary} for surjectivity:

\begin{thmx}\label{thmx:finitespc}
Fix a prime number $p$ and let $\cat U$ be a family of finite $p$-groups closed under quotients. Then the map \eqref{eq:intro_groupprimemap} is surjective if and only if $\cat{U}$ is essentially finite. In this case the map \eqref{eq:intro_groupprimemap} is a homeomorphism.
\end{thmx}

This immediately implies that the tt-geometry for infinite families must be much richer as the notion of group prime does not even capture the entire underlying set of the Balmer spectrum. The purpose of this paper is then to study the intricate and surprising new phenomena that appear when considering infinite families.

\subsection*{Phenomenology}\label{ssec:intro_phenomenology}

The failure of surjectivity of the map \eqref{eq:intro_groupprimemap} is explained by the observation that the category $\sfD(\cat U)$ is rarely rigidly-compactly generated. In fact, in \cite{BBPSWstructural} we prove that if $1 \in \cat U$ then there is a containment 
    \[
        \sfD^b(k\text{-}\mod)\simeq\sfD(\cat U)^{\dual} \subseteq \sfD(\cat U)^c
    \]
between dualizables and compact objects, and that this is an equality if and only if $\cat{U}$ contains only the trivial group. As the dualizable objects contain no information about the global structure of the derived category, we are led to consider its subcategory of compact objects $\sfD(\cat U)^c$. This is an essentially small tt-category which generates the whole derived category, see \cref{prop-rigidly-comp-gen,prop-compacts-tt-widely-closed}. 

The known surjectivity criteria in tensor-triangular geometry, such as \cite{Balmer18} or the aforementioned \cite{BCHS2024}, rely on rigidity. \Cref{thmx:finitespc} demonstrates that this is not a mere technicality, but absolutely fundamental. In fact, the standard toolbox of tensor-triangular geometry ceases to be available in the non-rigid context. From this abstract perspective, we introduce new methods that we hope will be useful in a more systematic study of \emph{non-rigid tt-geometry}.

In order to illustrate the various phenomena we encounter in the case of essentially infinite collections $\cat U$ of groups, consider the following three one-parameter families and their corresponding spectra $\Spc(\sfD(\cat U)^c)$:

\begin{figure}[h!]
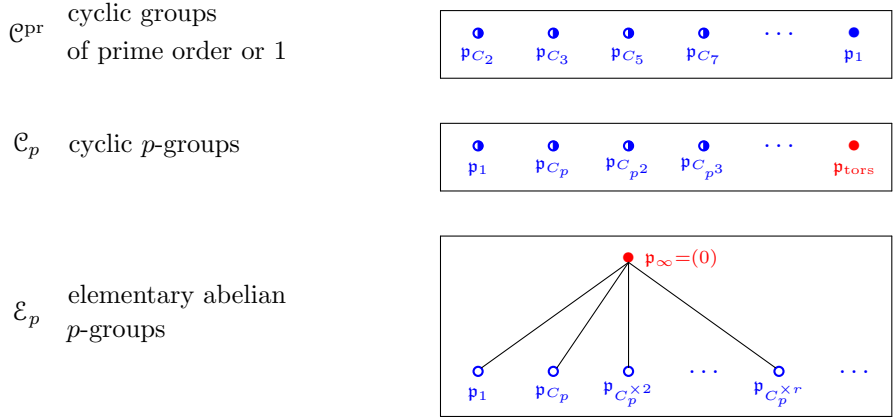

    \centering
\begin{equation*}
\label{eq:Spc(C_p)}%
\xy
(-40,0)*{\Cprime};
(-40,-15)*{\cat{C}_p};
(-40,-37.5)*{\cat{E}_p};
(-20,2.5)*{\text{cyclic groups}\hphantom{hhhhl}};
(-20,-2.5)*{\text{of prime order or 1}};
(-20,-15)*{\text{cyclic $p$-groups}\hphantom{hhh}};
(-20,-35)*{\text{elementary abelian}};
(-20,-40)*{\text{$p$-groups}\hphantom{hhhhhhhh}};
(20,0)*{\color{blue}\halfcirc[0.4ex]};
(30,0)*{\color{blue}\halfcirc[0.4ex]};
(40,0)*{\color{blue}\halfcirc[0.4ex]};
(50,0)*{\color{blue}\halfcirc[0.4ex]};
(60,0)*{\color{blue}\ldots};
(70,0)*{\color{blue}\bullet};
(20,-3)*{\color{blue}\scriptstyle \mathfrakp_{C_2}};
(30,-3)*{\color{blue}\scriptstyle \mathfrakp_{C_3}};
(40,-3)*{\color{blue}\scriptstyle \mathfrakp_{C_5}};
(50,-3)*{\color{blue}\scriptstyle \mathfrakp_{C_7}};
(70,-3)*{\color{blue}\scriptstyle \mathfrakp_{1}};
{\ar@{-} (15,3)*{};(75,3)*{}};
{\ar@{-} (15,3)*{};(15,-6)*{}};
{\ar@{-} (75,3)*{};(75,-6)*{}};
{\ar@{-} (15,-6)*{};(75,-6)*{}};
(20,-15)*{\color{blue}\halfcirc[0.4ex]};
(30,-15)*{\color{blue}\halfcirc[0.4ex]};
(40,-15)*{\color{blue}\halfcirc[0.4ex]};
(50,-15)*{\color{blue}\halfcirc[0.4ex]};
(60,-15)*{\color{blue}\ldots};
(70,-15)*{\color{red}\bullet};
(20,-18)*{\color{blue}\scriptstyle \mathfrakp_{1}};
(30,-18)*{\color{blue}\scriptstyle \mathfrakp_{C_p}};
(40,-18)*{\color{blue}\scriptstyle \mathfrakp_{C_{p^2}}};
(50,-18)*{\color{blue}\scriptstyle \mathfrakp_{C_{p^3}}};
(70,-18)*{\color{red}\scriptstyle \mathfrakp_{\text{tors}}};
{\ar@{-} (15,-12)*{};(75,-12)*{}};
{\ar@{-} (15,-12)*{};(15,-21)*{}};
{\ar@{-} (75,-12)*{};(75,-21)*{}};
{\ar@{-} (15,-21)*{};(75,-21)*{}};
(40,-30)*{\color{red}\bullet};
(20,-45)*{\color{blue}\emptycirc[0.4ex]};
(30,-45)*{\color{blue}\emptycirc[0.4ex]};
(40,-45)*{\color{blue}\emptycirc[0.4ex]};
(50,-45)*{\color{blue}\ldots};
(60,-45)*{\color{blue}\emptycirc[0.4ex]};
(70,-45)*{\color{blue}\ldots};
(47,-30)*{\color{red}\scriptstyle \mathfrakp_{\infty} = (0)};
(20,-48)*{\color{blue}\scriptstyle \mathfrakp_{1}};
(30,-48)*{\color{blue}\scriptstyle \mathfrakp_{C_p}};
(40,-48)*{\color{blue}\scriptstyle \mathfrakp_{C_p^{\times 2}}};
(60,-48)*{\color{blue}\scriptstyle \mathfrakp_{C_p^{\times r}}};
{\ar@{-} (20.5,-44.5)*{};(40,-30.5)*{}};
{\ar@{-} (30.5,-44.5)*{};(40,-30.5)*{}};
{\ar@{-} (40,-44.5)*{};(40,-30.5)*{}};
{\ar@{-} (59.5,-44.5)*{};(40,-30.5)*{}};
{\ar@{-} (15,-27)*{};(75,-27)*{}};
{\ar@{-} (15,-27)*{};(15,-51)*{}};
{\ar@{-} (75,-27)*{};(75,-51)*{}};
{\ar@{-} (15,-51)*{};(75,-51)*{}};
\endxy
\end{equation*}
\caption{Essential phenomena in $\Spc(\sfD(\cat U)^c)$. Blue indicates group primes, while the red bullets correspond to additional primes required to compactify $\pi_0\cat U$.}\label{fig:phenomena}
\end{figure}
\vspace{-1mm}

\subsubsection*{Change of topology} The first phenomenon arises for the family $\Cprime$ of cyclic groups of prime order or the trivial group $1$. In this case, the group prime map induces a continuous bijection
    \[
        \mathfrakp_{(-)}\colon\pi_0\Cprime \to \Spc(\D{\Cprime}^c),
    \]
so every prime is a group prime. (Note that this does not contradict \cref{thmx:finitespc}, as $\Cprime$ is not a family of $p$-groups for fixed $p$.) However, $\mathfrakp_{(-)}$ is not a homeomorphism as the topology changes: the discrete topology on $\pi_0\Cprime$ turns into the one-point compactification of $\Cprime \setminus \{1\}$, with $\mathfrakp_1$ forming the accumulation point. In \cref{fig:phenomena}, the different behaviour of the group primes is indicated by depicting $\mathfrakp_{C_p}$ with semi-filled circle corresponding to closed-open points of the spectrum, while the closed point $\mathfrakp_1$ is shown as a bullet. This phenomenon can be anticipated, as $\pi_0\Cprime$ is discrete and infinite, whereas the spectrum is always quasi-compact. \\

Restricting our attention to collections $\cat U$ of finite $p$-groups closed under quotients for some prime number $p$, we show that the previous phenomenon does not occur: the group prime map \eqref{eq:intro_groupprimemap} is always an open embedding, i.e., $\pi_0\cat U$ forms a discrete subspace of $\Spc(\sfD(\cat U)^c)$, see \cref{prop:pgroup_embedding}. This implies the existence of additional points in $\Spc(\sfD(\cat U)^c)$ whenever $\cat U$ is infinite, which exhibit $\Spc(\sfD(\cat U)^c)$ as a suitable compactification of $\pi_0\cat U$. Here, we distinguish between two different types of compactification, depending on whether the dimension of the space changes or not.

\subsubsection*{Compactification of type A}

Consider the example $\cat U = \cat C_p$ of cyclic $p$-groups. As a special case of \cref{Thm-abelian-p-groups-rank-r}, we show that the group prime map exhibits $\Spc(\sfD(\cat C_p)^c)$ as the one-point compactification of $\pi_0\cat C_p$. In particular, note that the Cantor--Bendixson rank of the space increases from $0$ to $1$. The newly added accumulation point $\mathfrakp_{\mathrm{tors}}$ consists of all global representations that are eventually zero, i.e., vanish when evaluated on all sufficiently large groups. It is our first example of a \emph{profinite group prime}, as explained in more detail below. 

\subsubsection*{Compactification of type B}

The second type of compactification appears for the family $\cat E_p$ of elementary abelian $p$-groups. In this case, a single point $\mathfrakp_{\infty}$ is added to $\pi_0\cat E_p$ to form $\Spc(\sfD(\cat E_p)^c)$, but this time it is in the closure of each of the group primes, $\mathfrakp_{\infty} \in \overline{\{\mathfrakp_{C_p^{\times n}}\}}$ for all $n \geq 0$. Topologically, this means that $\mathfrakp_{\infty}$ is both an accumulation point on the underlying constructible topology and also the only closed point in the spectral topology; as such, $\Spc(\sfD(\cat E_p)^c)$ is homeomorphic to the Hochster-dual of $\Spec(\bbZ)$. In fact, $\mathfrakp_{\infty}$ is simply given by the zero ideal in $\sfD(\cat E_p)^c$; a more general interpretation of these type of points as \emph{family primes} is given in the next subsection.

\subsection*{Family primes and elementary abelian $p$-groups}\label{ssec:familyprimes}

As a consequence of \cref{thmx:finitespc} and as witnessed by the examples above, the group primes do not suffice to capture the entire spectrum of global representations for any infinite family $\cat U$ of finite $p$-groups. We therefore introduce two new types of primes:
    \begin{itemize}
        \item family primes, corresponding to the points in compactifications of type B;
        \item and profinite group primes, corresponding to the points in compactifications of type A.
    \end{itemize}
Our goal is then to prove that these new primes provide a sufficient supply to understand the entire spectrum in numerous examples of families.

Family ideals generalize the construction of group primes by taking any subcollection $\cat V \subseteq\cat U$ and defining  
    \[
        \mathfrakp_{\cat V} =\{ X \in \D{\cat U}^c \mid H_*(X(G))=0, \; \forall G \in \cat V\}.
    \]
This is a thick ideal and our next contribution is to establish a condition under which $\mathfrakp_{\cat V}$ is prime. Building crucially on previous work \cite{BBPSWstructural}, we prove the next result (see \cref{prop:zeroideal,cor:familyprime}) which highlights the special role of multiplicative global families in global representation theory, i.e., families which are closed under finite products, subgroups and quotients.

\begin{thmx}\label{thmx:familyprimes}
    If $\cat V$ is a multiplicative global family in $\cat U$, then $\mathfrakp_{\cat V}$ is prime in $\D{\cat{U}}^c$. In particular, when $\cat U = \cat V$ is itself multiplicative global, then the zero ideal is prime in $\D{\cat U}^c$. 
\end{thmx}

This result also allows us to prove the first part of \cref{thmx:main}. Indeed, there is a strictly ascending filtration of the family $\fabp$ of finite abelian $p$-groups by multiplicative global subfamilies via the $p$-exponent: 
    \begin{equation}\tag{$\ast$}\label{eq:intro_apfiltration}
        \{1\} \subsetneq \cat E_p \subsetneq \fabp^{\leq 2} \subsetneq \ldots \subsetneq \fabp^{\leq l}= \{G \in \fabp \mid p^l G=0\} \subsetneq \ldots \subseteq \fabp.
    \end{equation}
Applying \cref{thmx:familyprimes} to this filtration provides a sequence of prime ideals in $\sfD(\fabp)^c$
    \[
        \mathfrakp_{1} \supsetneq \mathfrakp_{\cat E_p} \supsetneq \mathfrakp_{\leq 2} \supsetneq \ldots \supsetneq \mathfrakp_{\leq l}=\{ X \mid H_*(X(G))=0, \;\; G \in \cat A(p)^{\leq l}\}\supsetneq \ldots \supsetneq (0)
    \]
which turns out to be strictly descending (\cref{cor:infinitedimensional}). In particular, the spectrum $\Spc(\sfD(\fabp)^c)$ has infinite Krull dimension. 

Returning to the example of the family $\cat E_p$ of elementary abelian $p$-groups, we combine group primes and family primes to describe $\Spc(\sfD(\cat E_p)^c)$ and hence obtain a classification of the thick ideals therein, see \cref{thm-elementary}.

\begin{thmx}\label{thmx:elemab}\leavevmode
The prime ideals in $\sfD(\cat E_p)^c$ are the group primes $\mathfrakp_{(\bbZ/p)^{\times n}}$ for $n \in \bbN$ together with the family prime $(0)$. A subset $U \subseteq \Spc(\sfD(\cat E_p)^c)$ is open if and only if it either contains only group primes or if it is equal to the whole set; see \Cref{fig:phenomena}.
\end{thmx}

As part of the proof, we also show that all thick ideals in $\sfD(\cat E_p)^c$ are radical and can be generated by a single object. The equivalence between $\sfD(\cat E_p)$ and $\sfD(\VIMod)$ then lets us deduce the classification of derived VI-modules in \cref{thmx:VImod}.

\subsection*{Profinite group primes and abelian $p$-groups of bounded $p$-rank}\label{ssec:profiniteprimes} 
In order to discuss the next type of prime ideal that we encounter in global representation theory, we need to leave the world of finite groups and enter the world of finitely generated profinite groups. Given a subcategory of finite groups $\cat U$, we can consider its profinite extension $\hCU$ which consists of those finitely generated profinite groups for which a cofinal system of finite quotients lies in $\cat U$. In other words, we can write any $G \in \hCU$ as a filtered limit $G=\lim_{N \in \mathcal{N}(G;\cat U)} G/N$, where $\mathcal{N}(G;\cat U)$ is the collection of the open normal subgroups of $G$ with quotient belonging to $\cat U$. For any $G\in \hCU$, we construct a prime ideal
\[
\mathfrakp_G=\{X \mid \colim_{N \in \mathcal{N}(G;\cat U)}H_*(X(G/N))=0\} \in \Spc(\D{\cat U}^c)
\]
which we call a \emph{profinite group prime}, together with a profinite group prime map 
    \begin{equation}\tag{$\ddagger$}\label{eq:progroupprimemap}
        \mathfrakp_{(-)}\colon \pi_0 \hCU \to \Spc(\D{\cat U}^c)
    \end{equation}
extending the map from \eqref{eq:intro_groupprimemap}. 

If $\cat U$ admits a \emph{profinite reflective filtration} $\{\cat U[n]\}_{\geq n}$, that is, an exhaustive filtration of subfamilies in which the inclusion $\cat U[n] \hookrightarrow \hCU$ admits a left adjoint for all $n \geq 0$, then the map \eqref{eq:progroupprimemap} turns out to be injective, see \cref{prop:profinitemap-inj}. So, for families with a profinite reflective filtration, profinite group primes provide a closer approximation to the spectrum than only considering ordinary group primes. In \cref{cor-std-filt-profinite} we prove that many kinds of reflective filtrations are profinite.  For instance this holds if each category $\cat{U}[n]$ is essentially finite. For the family of finite abelian $p$-groups, the filtration $\{\cat A(p)^{\leq l}\}_{l\geq 0}$ in \eqref{eq:intro_apfiltration} is profinite reflective. 

Our next goal is then to find conditions on $\cat U$ which guarantee that the profinite group prime map is also surjective. The example of $\cat E_p$ implies that this is not always the case, so we need to focus on a special class of families. Given a positive integer $r\geq 0$, we say that a family of finite groups $\cat U$ is $r$-\emph{submultiplicative} if all groups in $\cat U$ are $r$-generated (that is, generated by at most $r$ elements), $1 \in \cat U$, and if $G \leq G_0 \times G_1$ with $G_0,G_1 \in \cat U$ and $G$ is $r$-generated, then $G \in \cat U$. Examples of this type include:
\begin{enumerate}
    \item $r$-generated finite groups;
    \item $r$-generated abelian groups, which includes the family of cyclic groups as the case $r=1$;
    \item $r$-generated solvable groups;
    \item $r$-generated nilpotent groups of class at most $c$ (for any
   given $c$);
    \item $r$-generated $p$-groups (for any given prime number $p$);
    \item $r$-generated groups $G$ such that all simple composition
   factors of $G$ have order at most $m$ (for any given $m$). 
\end{enumerate}

A special feature of $r$-submultiplicative families $\cat U$ is that they admit profinite reflective filtrations $\{\cat U[n]\}$ in which $\cat U[n]$ is essentially finite for all $n$. We are then able to show that the reflections induce a bijection 
\[
\pi_0\hCU\simeq\lim_n \pi_0 \cat U[n]
\]
and so the source of \eqref{eq:progroupprimemap} carries a natural profinite topology. With this topology in hand, we are able to prove (\cref{Thm-profinite-rgenerated}):

\begin{thmx}\label{thmx:rsubmultiplicative}
    Let $\cat U$ be an $r$-submultiplicative family. Then the profinite group map 
    \[
        \mathfrakp_{-}\colon \pi_0 \hCU \xrightarrow{\sim} \Spc(\D{\cat{U}}^c) 
    \]
    is a homeomorphism. Furthermore, in this case all thick ideals in $\D{\cat{U}}^c$ are radical. 
\end{thmx}

This result applies in particular to $\cat{U} = \fabprk{r}$, the family of finite abelian $p$-groups of $p$-rank at most $r$ (the case $r=2$ is depicted in \cref{fig:phenomena2}), and $\cat{U} = \fabrk{r}$, the family of finite abelian groups generated by at most $r$ elements. Unpacking the construction of $\hCU$ in this case then results in \cref{thmx:boundedrank}, giving a concrete geometric picture of the spectrum. Finally, we show the resulting spectrum for $\fabrk{r}$ is just the product of the spectra for $\fabprk{r}$, see \cref{cor-abelian-groups-rank-r}.

\begin{figure}[h!]
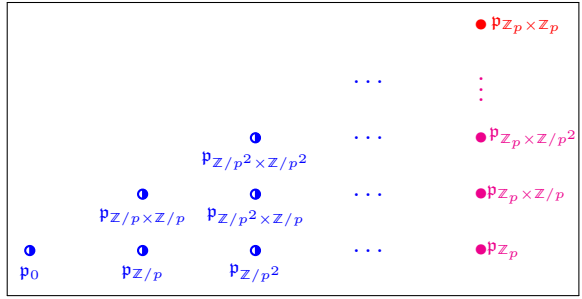

    \centering
        \begin{equation*}
        \xy
            (-25,20)*{\fabprk{2}: \text{ family of}\hphantom{hhh}};
            (-25,15)*{\text{finite abelian $p$-groups}};
            (-25,10)*{\text{of $p$-rank $\leq 2$}\hphantom{hhhhhhh}};
            (0,0)*{\color{blue}\halfcirc[0.4ex]};
            (15,0)*{\color{blue}\halfcirc[0.4ex]};
            (30,0)*{\color{blue}\halfcirc[0.4ex]};
            (45,0)*{\color{blue}\ldots};
            (60,0)*{\color{magenta}\bullet};
            (15,7.5)*{\color{blue}\halfcirc[0.4ex]};
            (30,7.5)*{\color{blue}\halfcirc[0.4ex]};
            (45,7.5)*{\color{blue}\ldots};
            (60,7.5)*{\color{magenta}\bullet};
            (30,15)*{\color{blue}\halfcirc[0.4ex]};
            (45,15)*{\color{blue}\ldots};
            (60,15)*{\color{magenta}\bullet};
            (45,22.5)*{\color{blue}\ldots};
            (60,22.5)*{\color{magenta}\vdots};
            (60,30)*{\color{red}\bullet};
            (0,-3)*{\color{blue}\scriptstyle \mathfrakp_{0}};
            (15,-3)*{\color{blue}\scriptstyle \mathfrakp_{\Z/p}};
            (30,-3)*{\color{blue}\scriptstyle \mathfrakp_{\Z/{p^2}}};
            (63,0)*{\color{magenta}\scriptstyle \mathfrakp_{\Z_p}};
            (15,4.5)*{\color{blue}\scriptstyle \mathfrakp_{\Z/p\times \Z/p}};
            (30,4.5)*{\color{blue}\scriptstyle \mathfrakp_{\Z/p^2\times \Z/p}};
            (66,7.5)*{\color{magenta}\scriptstyle \mathfrakp_{\Z_p \times \Z/p}};
            (30,12)*{\color{blue}\scriptstyle \mathfrakp_{\Z/{p^2}\times \Z/{p^2}}};
            (67,15)*{\color{magenta}\scriptstyle \mathfrakp_{\Z_p \times\Z/{p^2}}};
            (66,30)*{\color{red}\scriptstyle \mathfrakp_{\Z_p \times \Z_p}};
            {\ar@{-} (-3,33)*{};(73,33)*{}};
            {\ar@{-} (-3,33)*{};(-3,-6)*{}};
            {\ar@{-} (73,33)*{};(73,-6)*{}};
            {\ar@{-} (-3,-6)*{};(73,-6)*{}};
        \endxy
        \end{equation*}
    \caption{An illustration of $\Spc(\sfD(\fabprk{2})^c)$. The group primes are in blue and precisely give the isolated points of the spectrum. The remaining primes are profinite, with those in magenta forming accumulation points in the horizontal directions, which in turn converge to the point in red at the top.}
    \label{fig:phenomena2}
\end{figure}

\subsection*{Notation and terminology}\label{ssec:notation}

For the benefit of the reader, we collect some notation here which we use frequently throughout the paper, together with references for further details.

\begin{itemize}
    \item $\cat G$ (resp.~$\cat G_p$) denotes the category of finite groups (resp.~finite $p$-groups) and conjugacy classes of epimorphisms, see \cref{not-families}. Furthermore, we have the full subcategories $\cat E_p$ of elementary abelian $p$-groups, $\cat C_p$ of cyclic $p$-groups, $\cat C$ of all cyclic groups, $\Cprime$ of cyclic groups of prime order together with the trivial group, and $\cat A(p)$ of finite abelian $p$-groups.
    \item We write $\tCFG$ for the category of finitely generated groups and surjective homomorphisms, and $\CFG$ for the quotient category in which we identify conjugate homomorphisms, see \cref{def-fg}. Finally, we write $\hCG$ for the category of finitely topologically generated profinite groups with maps the conjugacy classes of surjective homomorphisms, see \cref{defn-hCG}. 
    \item We write $\cat{U}$ and $\cat{V}$ for replete full subcategories of $\cat{G}$ which are widely closed, see \cref{hyp}. Given $\cat{U}$, we write $\cat{U}\langle r\rangle$ for the full subcategory of $\cat{U}$ consisting of the $r$-generated groups, see \cref{defn-r-generated}, and $\widehat{\cat{U}}$ for the profinite extension of $\cat{U}$, see \cref{defn-hCU}.
    \item We write $G \gg H$ if there exists an epimorphism $G \twoheadrightarrow H$. For $G \in \cat{U}$, $\cat{U}_{\ll G}$ denotes the downwards closure of $G$ under this relation, and $\cat{U}_{\not \ll G}$ is the complement of this subfamily, see \cref{def:downclosure}.
    \item $\mathsf{A}(\cat{U};k)$ is the abelian category of functors $\cat{U}^{\op} \to \Mod{k}$ for a fixed field $k$ of characteristic 0, see \cref{def-AU-and-GrAU}. We suppress the field from the notation unless it is relevant.
    \item $\D{\cat{U}}$ is the derived category of $\A{\cat{U}}$, see \cref{def-derived-cat}, and $\D{\cat{U}}^c$ denotes the full subcategory of compact objects.
    \item Given a triangulated category $\cat T$ with arbitrary coproducts, we write $\cat T^c$ for the full subcategory of compact objects. 
    \item Given a collection of objects $\cat E \subseteq \cat T$, we write $\loc{\cat{E}}$ and $\loct{\cat E}$ for the localizing subcategory and localizing ideal generated by $\cat{E}$. Likewise, we write $\thick{\cat{E}}$ and $\thickt{\cat{E}}$ for the thick subcategory and thick ideal generated by $\cat{E}$.
\end{itemize}

\subsection*{Acknowledgements}\label{ssec:acknowledgements}

MB is supported by the EPSRC grant EP/X038424/1 “Classifying spaces, proper actions and stable homotopy theory”. TB is supported by the European Research Council (ERC) under Horizon Europe (grant No.~101042990). LP is supported by the SFB 1085 Higher Invariants in Regensburg. JW is supported by the project PRIMUS/23/SCI/006 from Charles University and the Charles University Research Centre
program No.~UNCE/24/SCI/022.

MB, TB, and LP are grateful to the Max Planck Institute for Mathematics in Bonn for its hospitality and financial support. MB, TB, LP, and JW would also like to thank the Fondation des Treilles for its support and hospitality, where work on this paper was undertaken. The authors would also like to thank the Hausdorff Research Institute for Mathematics for its hospitality and support during the trimester program ‘Spectral Methods in Algebra, Geometry, and Topology’, funded by the Deutsche Forschungsgemeinschaft under Germany’s Excellence Strategy – EXC-2047/1 – 390685813. JW would like to thank the Isaac Newton Institute for Mathematical
Sciences for the support and hospitality during the programme `Topology, representation theory and higher 
structures' when work on this paper was undertaken. This work was supported by: EPSRC Grant Number EP/R014604/1. The authors would like to thank the Isaac Newton Institute for Mathematical Sciences, Cambridge, for support and hospitality during the programme `Equivariant homotopy theory in context', where work on this paper was undertaken. This work was supported by EPSRC grant EP/Z000580/1.

\newpage
\part{Categories of global representations}
In this part we introduce the abelian category of global representations and its derived category, and list various useful features and results that they enjoy. Finally we connect this to rational global homotopy theory by showing that the derived category of rational global representations is symmetric monoidally equivalent to the $\infty$-category of rational global spectra as introduced by Schwede \cite{Schwedebook}.

 \section{The abelian category of global representations} 
 In this section we introduce the abelian category of global representations and recall some key properties from \cite{PolStrickland2022}. Let us start by setting some notation and terminology. 
 
\begin{Not}\label{not-families}
We denote by $\cat{G}$ the category of finite groups and conjugacy classes of surjective group homomorphisms. We will denote by $\cat U$ a full subcategory of $\cat G$ which we will always assume to be closed under isomorphism (i.e., replete), even if not explicitly stated.  We furthermore write $\pi_0\cat U$ for the set of isomorphism classes of objects of $\cat U$. Given $H,K \in\cat U$, we will often write $H \gg K$ to mean that $\Hom_{\cat U}(H,K)\not =\emptyset$ so there exists an epimorphism $H \twoheadrightarrow K$. We will denote the \emph{upwards closure} of a subset $S \subseteq \pi_0\cat U$ by
        \[
            \ua(S) \coloneqq \{G \in \pi_0\cat{U} \mid \exists H \in S\colon G \gg H\}.
        \]
    Finally, we say that a subgroup $H$ of $G \times K$ is \emph{wide} if both projections $G \leftarrow H\to K$ are surjective homomorphisms.
\end{Not} 

We will often assume that $\cat U$ satisfies some of the following additional conditions. 

\begin{Def}\label{def-type-families}
Let $\cat U$ be a subcategory of $\cat G$ as above. We say that:
    \begin{enumerate}
        \item $\cat U$ is \emph{closed downwards}: if $G \in \cat U$ and there exists a surjective group homomorphism $G \twoheadrightarrow H$, then $H \in \cat U$ too; 
        \item $\cat U$ is \emph{closed upwards}: if $G \in \cat U$ and there exists a surjective group homomorphism $H\twoheadrightarrow G$, then $H \in \cat U$ too;
        \item $\cat U$ is a \emph{global family}: if it is closed downwards and it is closed under subgroups;
        \item $\cat U$ is \emph{widely closed}: if whenever $G \twoheadleftarrow H \twoheadrightarrow K$ are surjective homomorphisms with $G, H, K \in \cat U$, the image of the combined morphism $H \to G \times K$ (which is a wide subgroup) is also in $\cat U$;
        \item $\cat U$ is \emph{essentially finite}: if it contains only finitely many isomorphism classes of objects; 
        \item $\cat U$ is \emph{unital}: if it contains the trivial group $1$;
        \item $\cat U$ is \emph{multiplicative}: if $G,H \in\cat U$ then also $G\times H \in \cat U$.
    \end{enumerate}
\end{Def}

\begin{Exa}\label{eg-implies-widely-closed}
    As discussed in \cite[Remark 3.2]{PolStrickland2022}, conditions (a) and (b)  both imply (d). In particular any global family is widely closed. The same reference also shows that any groupoid is widely closed. 
\end{Exa}

\begin{Rem}
    In \cref{def-type-families} we introduced the conditions of being closed upwards and of being closed downwards for a collection of finite groups $\cat U$. Later, we will also need a relative version of these conditions which depends on an inclusion $i\colon \cat U \to \cat V$. For instance we will say that $\cat U$ is closed upwards in $\cat V$, if $G \in \cat U$ and there exists a surjective group homomorphism $H \twoheadrightarrow G$ with $H \in \cat V$, then $H \in \cat U$. We recover the absolute conditions by letting $\cat V$ be the family of all finite groups.
\end{Rem}

Throughout the whole paper we will always work under the following assumption unless otherwise stated.
\begin{Hyp}\label{hyp}
    We assume that any subcategory of $\cat G$ is widely closed, in addition to being full and replete.
\end{Hyp}
Recall from \cref{eg-implies-widely-closed} that virtually all examples of subcategories satisfy the above hypothesis. 
Next, we introduce the abelian category of global representations. 

\begin{Def}\label{def-AU-and-GrAU}
    Let $\cat U\subseteq \cat G$ be a subcategory and let $k$ be a field of characteristic zero. The abelian category of $\cat U$-global representations is the functor category 
    \[
    \sfA(\cat U; k)\coloneqq\Fun(\cat U^{\op}, \Mod{k}).
    \]
    We will often drop the field from the notation and simply write $\A{\cat U}$.
\end{Def}

\begin{Exa}
    Let $\{G\}\subseteq \cat G$ denote the replete full subcategory spanned by a single finite group $G$. In this case $\A{\{G\}}\simeq\Mod{k[\Out(G)]}$. 
\end{Exa}

We now introduce a collection of projective generators in $\A{\cat U}$ which we will show below give compact generators for the derived category as well.

\begin{Def}\label{defn-eGV}
    Let $\cat U \subseteq \cat G$ be a subcategory. For any $G \in \cat U$, there is an evaluation functor 
    \[
    \ev_G \colon \A{\cat U} \to \A{\{G\}}\quad X \mapsto X(G),
    \]
    which by \cite[Lemma 2.9]{PolStrickland2022} admits a left adjoint 
    \[
    e_{G, \bullet} \colon \A{\{G\}} \to \A{\cat{U}} \quad V \mapsto e_{G,V}.
    \]
    For brevity, we set $e_G \coloneqq e_{G,k[\Out(G)]}$ so that by construction 
    \[\Hom_{\A{\cat{U}}}(e_G, X) \cong X(G).\]
    Since $\ev_G$ is exact and also admits a right adjoint, together with the fact that these evaluation functors are jointly conservative, we deduce that $(e_G)_{G \in \cat U}$ provides a set of finitely presented projective generators for $\A{\cat{U}}$.   
\end{Def}

\begin{Rem}\label{rem-formula-eGV}
    More concretely, we have the formula 
    \[
    e_{G,V} \cong V \otimes_{k[\Out(G)]} k[\Hom_{\cat{U}}(-,G)]\in\A{\cat{U}}.
    \]
\end{Rem}

\begin{Exa}\label{ex-unit}
    If $\cat{U}$ is unital, we observe that $e_{1}$ is the constant diagram with value $k$ and all maps the identity.
\end{Exa}

Let us record some general features of the abelian category $\A{\cat U}$:
\begin{enumerate}
    \item It has all small limits and colimits, and these are calculated pointwise.
    \item It is Grothendieck with generators given by $e_G$ for all $G \in \cat U$. This means that filtered colimits are exact and that any $X \in \A{\cat U}$ admits an epimorphism $P \to X$ where $P$ is a direct sum of generators. If the direct sum in the definition of $P$ is finite, then $X$ is said to be \emph{finitely generated}. 
    \item By general Grothendieck abelian category theory, we know that $\A{\cat U}$ has enough injective objects.
    \item It is a symmetric monoidal category with pointwise tensor product, so $(X\otimes Y )(G)=X(G) \otimes Y(G)$ for all $X,Y\in \A{\cat U}$ and $G\in\cat{U}$. The tensor unit $\unit$ is the constant diagram with value $k$ and all maps the identity. In particular, if $\cat{U}$ is unital, then $\unit \cong e_1$ by \cref{ex-unit}.
    \item The symmetric monoidal structure is closed with internal hom object given by 
    \[
    \iHom(X,Y)(G)\cong\Hom_{\A{\cat U}}(e_G \otimes X, Y)
    \]
    for all $X,Y\in \A{\cat U}$ and $G\in\cat{U}$. 
    \item Any object in $\A{\cat U}$ is flat.
\end{enumerate}

In the next result we collect some properties of projective objects in $\A{\cat U}$.

\begin{Prop}\label{prop-proj-objects}
    Let $\cat U$ be a subcategory of $\cat G$. 
    \begin{enumerate}
        \item Any projective object of $\A{\cat U}$ decomposes as a direct sum of objects of the form $e_{G,V}$ for $G\in\cat U$ and $V$ an $\Out(G)$-representation.
        \item Products of projective objects in $\A{\cat U}$ are again projective.
        \item The full subcategory of projective objects in $\A{\cat U}$ is closed under tensor products. Moreover, the tensor product of finitely generated projective objects is again finitely generated projective.
        \item If $\cat U$ is a multiplicative global family, then the full subcategory of projective objects in $\A{\cat U}$ is closed under the internal hom functor. Moreover, the internal hom between finitely generated projective objects is again finitely generated projective.
        \item The tensor unit $\unit\in\A{\cat{U}}$ is a projective object if and only if for each $G\in\cat{U}$, there is a unique (up to isomorphism) normal subgroup $N\lhd G$ such that $G/N$ is a minimal object of $\cat{U}$ under the relation $\gg$. It is furthermore finitely generated if and only if there are only finitely many isomorphism classes of minimal objects in $\cat{U}$. 
    \end{enumerate}
\end{Prop}
\begin{proof}
    For (a)--(d) see \cite[Corollary 8.5, Propositions 8.6 and 8.7]{PolStrickland2022}. We note that the proof of \cite[Proposition 8.7]{PolStrickland2022} in fact proves that the subcategory of finitely generated projective objects is closed under tensor products and internal hom functors under the  assumptions of part (c) and (d), respectively. We emphasize that (c) holds since we are under \cref{hyp}. Part (e) is proved in \cite[Proposition 3.21]{BBPSWstructural}.
\end{proof}

In the next construction we record how these abelian categories are related to one another via change of families functors. 

\begin{Cons}\label{con-functors}
    Given a functor $f \colon \cat U \to \cat V$, we have adjunctions
    \[
    \begin{tikzcd}[column sep=large, row sep=large]
        \A{\cat{U}} \arrow[r,"f_!", yshift=2mm] \arrow[r, yshift=-2mm, "f_*"'] & \A{\cat{V}} \arrow[l, "f^*" description] 
    \end{tikzcd}
    \]
    with left adjoints displayed on the top. By definition $f^*(X)(H)=X(f(H))$, $f_!$ is given by the left Kan extension along $f$, and $f_*$ is given by the right Kan extension along $f$. We note that $f^*$ is always symmetric monoidal and exact. If $f$ is fully faithful, then the theory of Kan extensions tells us that $f_!$ and $f_*$ are also fully faithful, and that if $f$ admits a left adjoint $q$, then $f_! \simeq q^*$, see \cite[Lemma 5.3]{PolStrickland2022} for more details. 
\end{Cons}

\begin{Rem}\label{rem-restriction-gen}
    A priori the objects $e_{G,V}$ from \cref{defn-eGV} depend on the ambient category $\cat U$, so we could indicate this by writing $e_{G,V}^{\cat U}$.
    However by \cite[Lemma 5.3(j)]{PolStrickland2022}, given an inclusion $i\colon \cat U \to \cat V$, we have canonical isomorphisms
    \[
    i_! e_{G,V}^{\cat U}\cong e_{G,V}^{\cat V} \quad \mathrm{and} \quad i^*e_{G,V}^{\cat V}\cong e_{G,V}^{\cat U}
    \]
    for all $G \in \cat{U}$, so often there is no harm in dropping the superscript. 
\end{Rem}

\section{The derived category of global representations}
In this section we recall some background about the derived category of global representations from~\cite{BBPSWstructural}, introduce some special objects in the derived category, and finally discuss the formalism of changing family in the derived category.  For instance we show that under mild assumptions on the family $\cat U$, the derived category $\D{\cat U}$ naturally fits into two recollements, see \cref{thm:recollement}. Let us begin by introducing the derived category.  
\begin{Def}\label{def-derived-cat}
    Let $\cat{U}\subseteq \cat G$ be a subcategory. The derived category of global representation is 
    \[
    \sfD(\cat U)\coloneqq \sfD(\A{\cat U}).
    \]
    If we want to emphasize the field we are working over, we will write $\sfD(\cat U;k)$.
\end{Def}
The derived category is canonically a triangulated category.
\begin{Prop}\label{prop-rigidly-comp-gen}
    The derived category $\D{\cat{U}}$ is compactly generated by 
    \[
    \{e_G \mid G \in \cat U\}.
    \]
    Moreover $\D{\cat{U}}$ is rigidly-compactly generated if and only if $\cat U$ is a finite groupoid. 
\end{Prop}
\begin{proof}
The first claim was proved in \cite{BBPSWstructural} but it also follows from \cref{lem-prop-e_g} below. The second claim is proved in \cite[Theorem 8.7]{BBPSWstructural}.
\end{proof}

It follows from the previous result that the full subcategory of compact objects $\D{\cat{U}}^c$ is an essentially small triangulated category. The next result gives another characterisation of the compact objects.

\begin{Prop}\label{prop-compact-perfect}
    The following are equivalent for an object $X \in \D{\cat{U}}$:
    \begin{enumerate}
        \item $X$ is compact;
        \item $X \in \thick{e_G \mid G \in \cat{U}}$;
        \item $X$ is isomorphic (in $\D{\cat{U}}$) to a perfect complex, that is, to a bounded complex of finitely generated projective objects.
    \end{enumerate}
\end{Prop}
\begin{proof}
    This is proved (along with some other characterisations) as~\cite[Theorem 7.3]{BBPSWstructural}.
\end{proof}

We now discuss the symmetric monoidal structure on the derived category and on its subcategory of compact objects. 
The derived tensor product equips $\D{\cat{U}}$ with the structure of a tensor-triangulated category. In fact since any object in the abelian category is flat, the derived tensor product extends the tensor product of the abelian category in the sense that there is a commutative diagram 
\[
\begin{tikzcd}
\A{\cat U} \times \A{\cat U} \arrow[r,"-\otimes -"] \arrow[d, hook] & \A{\cat U} \arrow[d, hook] \\
\sfD(\cat U) \times \sfD(\cat U) \arrow[r, "-\otimes-"]& \sfD(\cat U),
\end{tikzcd}
\]
where the vertical arrows denote the canonical inclusions. 
We next show that our standing assumptions (\cref{hyp}) ensure that the subcategory of compact objects $\D{\cat U}^c$ is closed under the tensor product of $\D{\cat{U}}$.

\begin{Prop}\label{prop-compacts-tt-widely-closed}
     The full subcategory of compact objects $\D{\cat{U}}^c$ is closed under tensor products. If in addition $\unit \in \A{\cat U}$ is finitely generated projective, then $\D{\cat U}^c$ is symmetric monoidal subcategory of $\D{\cat{U}}$. 
\end{Prop}
\begin{proof}
    Since compact objects and perfect objects agree by \cref{prop-compact-perfect}, to check closure under tensor products it suffices to show that the tensor product of two finitely generated projective objects is again a finitely generated projective object. This follows from \cref{prop-proj-objects}(c). If $\unit \in \A{\cat{U}}$ is finitely generated projective, then it is a perfect complex, and hence compact.
\end{proof}

\begin{Rem}\label{rem-unit-fg-projective}
We observe that the condition that $\unit \in\A{\cat U}$ is a finitely generated projective is satisfied in many cases of interest. For example it is satisfied if $\cat U$ is unital by \cref{ex-unit}, or if $\cat U$ is a finite groupoid by \cite[Examples 3.22(b)]{BBPSWstructural}. Moreover by \cref{prop-proj-objects}(e), the condition that $\unit \in\A{\cat U}$ is a finitely generated projective descends to any downwards closed subcategory of $\cat U$.
\end{Rem}

In the next construction we extend \cref{defn-eGV} to the derived level.
\begin{Cons}
    Recall from \cref{defn-eGV} that for any $G \in \cat U$, the evaluation functor 
    \[
    \ev_G \colon \A{\cat U} \to \A{\{G\}}, \quad X \mapsto X(G)
    \]
    admits a left adjoint $e_{G, \bullet}$. We observe that the additive functor $e_{G,\bullet}$ is exact as by Maschke's theorem every short exact sequence in $\A{\{G\}}$ is split. It then follows that the above adjunction descends to an adjunction
    \[
    \begin{tikzcd}
        \sfD(\cat U) \arrow[r, shift right, "\ev_G"'] & \arrow[l, shift right, "e_{G, \bullet}"']\sfD(\{G\})
    \end{tikzcd}
    \]
    at the level of derived categories without the need to derive the functors. 
\end{Cons}

Next, we record some properties of the functor that we have just defined.

\begin{Lem}\label{lem-prop-e_g}
    Let $G \in \cat U$, $X \in \sfD(\cat U)$ and $V\in \D{\{G\}}$. 
    \begin{enumerate}
    \item $e_{G,V}\simeq e_{G, H_*(V)}$ as objects of $\D{\cat U}$;
    \item The functor $H_*(\ev_G)$ induces an equivalence 
    \[
    \Hom_{\D{\cat U}}(e_{G,V}, X)\cong \Hom_{\Gr(\A{\{G\}})}(H_*(V), H_*(X(G)))
    \]
    where $\Gr$ denotes the category of $\Z$-graded objects.
    \end{enumerate}
\end{Lem}
\begin{proof}
Recall from \cite[Lemma 2.16]{BBPSWstructural} that since $\A{\{G\}}$ is semisimple, there is a chain homotopy equivalence $V \simeq H_*(V)$ in $\D{\{G\}}$. Then part (a) follows from the fact that $e_{G,\bullet}$ is an additive functor. By adjunction part (b) is equivalent to proving that 
\[
\Hom_{\D{\{G\}}}(V, X(G)) \cong \Hom_{\Gr(\A{\{G\}})}(H_*(V), H_*(X(G)))
\]
which holds by \cite[Corollary 2.18]{BBPSWstructural}.
\end{proof}

\begin{Lem}\label{lem-eGU-eGV}
 Let $\cat U$ be a subcategory and $G \in \cat U$. For $k[\Out(G)]$-modules $U$ and $V$, there are natural maps 
 \[ 
  e_{G,U\otimes V} \xrightarrow{i} e_{G,U}\otimes e_{G,V} \xrightarrow{p}   e_{G,U\otimes V}
 \]
 whose composite is the identity.
\end{Lem}
\begin{proof}
   The map $i$ in the lemma is the map adjoint to the identity map of $U \otimes V$. Explicitly, using the formula in \cref{rem-formula-eGV}, we have 
    \[ 
 i(K)([\alpha]\otimes(u\otimes v)) = ([\alpha]\otimes u)\otimes([\alpha]\otimes v). 
 \]
    for all $K \in \cat U$, $\alpha \in \Hom_{\cat U}(K,G)$, $u\in U$ and $v\in V$. In the opposite direction, we define a map $p(K)\colon  e_{G,U}(K)\otimes e_{G,V}(K) \to e_{G, U \otimes V}(K)$ by the formula
    \[ 
    p(K)([\alpha]\otimes u\otimes [\beta]\otimes v) = 
     \begin{cases}
      [\alpha]\otimes(u\otimes \theta \cdot v) & \text{ if } \beta=\theta\alpha
       \text{ for some } \theta\in\Out(G) \\
      0 & \text{ otherwise}
     \end{cases}
 \]
 for all $\alpha,\beta\in \Hom_{\cat U}(K,G)$, $u\in U$ and $v \in V$. Unravelling the definitions, one verifies that for all $\phi, \psi \in \Out(G)$ we have
 \[ p(K)([\phi^{-1}\alpha]\otimes\phi \cdot u \otimes[\psi^{-1}\beta]\otimes\psi \cdot v) = 
     p(K)([\alpha]\otimes u\otimes [\beta]\otimes v)
 \]
 so the map $p(K)$ is well-defined. As it is clearly natural in $K$, it defines a map $p$ as in the lemma. It is also easy to see that $p\circ i=\mathrm{id}$ concluding the proof.
\end{proof}

\begin{Lem}\label{cor-eG-eGV}
 Let $\cat U\subseteq\cat G$ be a subcategory and $G \in \cat U$. If $C\in \D{\{G\}}^c$ is nonzero, then 
 \[
 \thickt{e_{G,C}}=\thickt{e_G} \subseteq \D{\cat U}^c
 \]
\end{Lem}
\begin{proof}
 By \cref{lem-prop-e_g}(a) we can reduce to the case where $C=V$ is a nonzero, finite dimensional $\Out(G)$-representation.

 We first claim that $e_G=e_{G,k[\Out(G)]}$ is a retract of
 $e_G\otimes e_{G,V}$. To see this, we note that for any finite group $\Gamma$ and $k[\Gamma]$-module $V$ of $k$-dimension $d$, we have $k[\Gamma]\otimes V\cong k[\Gamma]^d$ as $k[\Gamma]$-modules.  In particular, if $V\neq 0$ then $k[\Gamma]$ is a retract of $k[\Gamma]\otimes V$.  The claim then follows from this by taking $\Gamma=\Out(G)$ and using \cref{lem-eGU-eGV}.

 With this claim in hand we can prove the lemma. Write $V$ as $S_1\oplus\dotsb\oplus S_m$, where each $S_i$
 is indecomposable.  This means that $e_{G,V}=\bigoplus_ie_{G,S_i}$
 and $e_{G,S_i}$ is a retract of $e_G$ so $e_{G,V}\in\thickt{e_G}$.
 On the other hand, as $e_G$ is a retract of $e_G\otimes e_{G,V}$ by the previous paragraph, we see
 that $e_G\in\thickt{e_{G,V}}$.  It follows that
 $\thickt{e_{G,V}}=\thickt{e_G}$ as claimed.
\end{proof}

We next describe some objects which play the role of characteristic functions.
\begin{Def}\label{def-chi-objects}
    Let $\cat U\subseteq \cat G$ be a subcategory and $G \in \cat{U}$.
    \begin{enumerate}
    \item If $V$ is an $\Out(G)$-representation, we let $\chi_{G,V}$ denote the object of $\A{\cat{U}}$ given by the formula
    \[
    \chi_{G,V}(H)\coloneqq\begin{cases}
        e_{G,V}(H)\cong V \otimes_{k[\Out(G)]} k[\Hom_{\cat{U}}(H,G)] & \mathrm{if}\; H \cong G\\
        0 & \mathrm{if}\; H \not \cong G
    \end{cases}
    \]
    for all $H \in \cat U$, together with the following structure maps: for any surjective group homomorphism $\alpha \colon H \to K$, the map $\alpha^* \colon \chi_{G,V}(K) \to \chi_{G,V}(H)$ is the same as for $e_{G,V}$ if $H\cong K\cong G$ (in which case $\alpha$ is necessarily an isomorphism), and zero in all other cases.
    \item If $C \in \D{\{G\}}$, we let $\chi_{G, C}$ denote the object 
    \[
    \chi_{G,C}\coloneqq e_{G,C} \otimes \chi_{G, k} \in \D{\cat U}.
    \]
    \end{enumerate}
\end{Def}

Let us record some of the key properties of these objects.

\begin{Prop}\label{prop-chi-VandQ}
    Let $\cat U\subseteq \cat G$ be a subcategory and $G \in \cat{U}$. 
    \begin{enumerate}
    \item If $C\in \D{\{G\}}$, then $\chi_{G,C}\simeq \chi_{G,H_*(C)}\in \D{\cat U}$.
    \item For all nonzero $C\in \D{\{G\}}$, we have
    \[
    \thickt{\chi_{G,k}}=\thickt{\chi_{G,C}} \subseteq \D{\cat U}.
    \]
    \end{enumerate}
\end{Prop}

\begin{proof}
     Part (a) follows from the fact the tensor product is exact and \cref{lem-prop-e_g}. Consider a nonzero finite dimensional $\Out(G)$-representation $V$, and note that by \cref{cor-eG-eGV} we have an equality of thick ideals
     \[
    \thickt{e_{G,k}}=\thickt{e_{G,V}} \subseteq \D{\cat U}.
    \]
    Tensoring the above equality with $\chi_{G,k}$ gives the claim in (b) for $C$ a representation, and this is enough by part (a).
\end{proof}

We finish this section by discussing how the derived categories for different families relate to one another.

\begin{Cons}\label{cons-restriction-derived}
    Let $f\colon \cat{U} \to \cat{V}$ be a functor. By the universal property of the derived category, the adjunctions
    \[
    \begin{tikzcd}[column sep=large, row sep=large]
        \A{\cat{U}} \arrow[r,"f_!", yshift=2mm] \arrow[r, yshift=-2mm, "f_*"'] & \A{\cat{V}} \arrow[l, "f^*" description] 
    \end{tikzcd}
    \]
from \cref{con-functors}, descend to adjunctions 
    \[
    \begin{tikzcd}[column sep=large, row sep=large]
        \D{\cat{U}} \arrow[r,"f_!", yshift=2mm] \arrow[r, yshift=-2mm, "f_*"'] & \sD(\cat{V}), \arrow[l, "f^*" description] 
    \end{tikzcd}
    \]
    where we have not distinguished between a functor and its derived functor. We observe that the functor $f^*$ is exact so it does not require deriving, but at this level of generality the functor $f_!$ and $f_*$ are only right exact and left exact, respectively, so they require deriving. The theory of Kan extensions tells us that if $f$ is fully faithful, then the functors $f_!$ and $f_*$ are also fully faithful.
\end{Cons}

\begin{Prop}\label{prop-gen-fun-preserve-compacts}
Let $f\colon \cat{U} \to \cat{V}$ be a functor. 
\begin{enumerate}
\item The functor $f^*$ is symmetric monoidal and hence $f_!$ (resp., $f_*$) admits a canonical oplax (resp., lax) symmetric monoidal structure.
\item The functor $f_!\colon \D{\cat{U}} \to \sD(\cat{V})$ preserves compact objects.  
\end{enumerate}
\end{Prop}
\begin{proof}
    Part (a) follows from the fact that $f^* \colon \A{\cat U} \to \A{\cat V}$ is symmetric monoidal and exact. For (b), the functor $f_!$ preserves compacts, since its right adjoint $f^*$ preserves colimits.
\end{proof}

\begin{Prop}\label{prop-incl-preserve-compacts}
Let $i\colon \cat{U} \to \cat{V}$ be an inclusion. 
\begin{enumerate}
\item If $\cat{U}$ is closed upwards in $\cat V$, then $i_! \colon \D{\cat U}\to \D{\cat V}$ is extension by zero, and hence $i_!$ preserves tensors (but not necessarily the unit object).
\item If $\cat{U}$ is closed downwards in $\cat V$, then $i_* \colon \D{\cat U}\to \D{\cat V}$ is extension by zero, and hence preserves tensors (but not necessarily the unit object).
\item If $\cat U$ is a multiplicative global family, then $i_!$ is strong symmetric monoidal.
\item If $\cat{U}$ is closed downwards in $\cat V$, then $i^*\colon \sD(\cat{V}) \to \D{\cat{U}}$ preserves compact objects.
\end{enumerate}
\end{Prop}
\begin{proof}
 Parts (a) and (b) follow from \cite[Lemma 5.3 (f) and (g)]{PolStrickland2022} which show that under the above assumptions the abelian functors $i_!$ and $i_*$ are given by extension by zero and so they are exact and they preserve tensors. Now for (c), suppose that $\cat U$ is a multiplicative global family and consider the oplax structure map $i_!(X \otimes Y) \to i_!(X) \otimes i_!(Y)$; we want to show that this map is an isomorphism. By a localizing subcategory argument we can assume that $X=e_G$ and $Y=e_H$ for some $G,H \in \cat U$. The oplax structure map is an isomorphism in this case by \cite[Lemma 5.3(h)]{PolStrickland2022}, and hence is always an isomorphism. The same reference also shows that $i_!$ preserves the tensor unit, so that $i_!$ is strong monoidal. For (d), suppose $\cat{U}$ is closed downwards. In this case, $i^*\colon \A{\cat{V}} \to \A{\cat{U}}$ preserves finitely generated projectives by \cite[Lemma 11.2]{PolStrickland2022}. Therefore by \cref{prop-compact-perfect} we see that $i^*\colon \sD(\cat{V}) \to \D{\cat{U}}$ preserves compacts.
\end{proof}

\begin{Rem}
    If $\cat{U}$ is not closed downwards, then $i^*\colon \sD(\cat{V}) \to \D{\cat{U}}$ need not preserve compacts. A counterexample may be found following the same line of argument as in~\cite[Remark 11.3]{PolStrickland2022}.
\end{Rem}

Using the change of families functors we can explain how to construct the characteristic function objects from the compact generators. To this end recall our notation for the upwards closure of a subset from \cref{not-families}.

\begin{Lem}\label{lem-chi-eGs}
     For any $G \in \cat U$ and $V\in \D{\{G\}}$, there is a triangle in $\D{\cat U}$
        \[
            F \to e_{G,V} \xrightarrow{f} \chi_{G,V}
        \]
        where $f$ is adjoint to the identity map and $F \in \loc{e_{H} \mid H \in \ua(\{G\}) \setminus \{G\}}$.
\end{Lem}
\begin{proof}
  Write $i \colon  \ua(\{G\}) \setminus \{G\} \to \cat U$ for the inclusion. We claim that the counit map $i_!i^*F \to F$ is an equivalence. This follows from the fact that $i_!$ is extension by zero by \cref{prop-incl-preserve-compacts}(a) and that 
  $F(H)\simeq 0$ for all $H \not \in \ua(\{G\}) \setminus \{G\}$. By \cref{prop-rigidly-comp-gen}, we know that $i^*F \in \loc{e_H \mid H\in \ua(\{G\}) \setminus \{G\} }$ and so 
  \[
    F\simeq i_!i^* F \in \loc{e_H \mid H \in \ua(\{G\}) \setminus \{G\} }
  \]
  using \cref{rem-restriction-gen}.
\end{proof}
As another application of the change of family functors, we show that the derived category fits into a (symmetric) recollement. 

\begin{Thm}\label{thm:recollement}
    Let $i\colon \cat{U} \to \cat{V}$ be an inclusion with complement $j\colon \cat V \,\backslash\, \cat{U} \to \cat{V}$. Assume that $\cat U$ is closed downwards in $\cat V$, so that $ \cat V \,\backslash\, \cat{U}$ is closed upwards in $\cat V$. Then there are recollements of the form:
    \[
    \begin{tikzcd}[column sep=1.5cm]
        \D{\cat U} \arrow[r, "i_*" description] & \D{\cat V} \arrow[r, "j^*" description] \arrow[l, yshift=-2.5mm, "i^{\dagger}"] \arrow[l, yshift=2.5mm, "i^*"'] & \D {\cat V \,\backslash\, \cat{U}}\arrow[l, yshift=2.5mm, "j_!"'] \arrow[l, yshift=-2.5mm, "j_*"]
    \end{tikzcd}
    \]
    and
    \[
     \begin{tikzcd}[column sep=1.5cm]
        \D{\cat V \, \backslash \, \cat U} \arrow[r, "j_!" description] & \D{\cat V} \arrow[r, "i^*" description] \arrow[l, yshift=-2.5mm, "j^*"] \arrow[l, yshift=2.5mm, "j^!"'] & \D{\cat{U}}\arrow[l, yshift=2.5mm, "i_!"'] \arrow[l, yshift=-2.5mm, "i_*"].
    \end{tikzcd}
    \]
\end{Thm}

\begin{proof}
    Let us first prove the existence of the first recollement. According to \cite[Definition A.8.1]{HA} we need to verify:
    \begin{enumerate}
        \item the subcategories $\D{\cat U}$ and $\D{\cat V \, \backslash\, \cat U}$ are closed under equivalences;
        \item the functors $i_*$ and $j_*$ admit left adjoints $i^*$ and $j^*$;
        \item the functors $i^*$ and $j^*$ are exact;
        \item $j^* i_*\simeq 0$;
        \item the functors $i^*$ and $j^*$ are jointly conservative. 
    \end{enumerate}
    The existence of the other adjoints will follow formally from these points, see \cite[Remark A.8.19]{HA}. Parts (a), (b) and (c) are clear, and (d) follows from the fact that $i_*$ is the extension by zero functor, see \cref{prop-incl-preserve-compacts}(b). Therefore it remains to verify part (e). Let $X \in \D{\cat V}$ be such that $j^*X \simeq i^*X\simeq 0$; we claim that $X \simeq 0$. We note that $j_*j^* X \simeq i_*i^*X \simeq 0$. If $G \in \cat V \setminus \cat U$, then $j_!j^* e_G\cong e_G$ (by \cref{rem-restriction-gen}) so 
    \begin{align*}
    \Map_{\D{\cat V}}(e_G, X) & \simeq\Map_{\D{\cat V}}(j_!j^* e_G, X) \\
                            & \simeq \Map_{\D{\cat V}}(e_G, j_*j^*X)\simeq 0.
    \end{align*}
    Similarly, if $G \in \cat U$, then $i_!i^* e_G\cong e_G$ (again by \cref{rem-restriction-gen}) so 
    \begin{align*}
    \Map_{\D{\cat V}}(e_G, X) & \simeq \Map_{\D{\cat V}}(i_!i^* e_G, X) \\
                            & \simeq \Map_{\D{\cat V}}(e_G, i_*i^*X)\simeq 0.
    \end{align*}
    Since $\{e_G\}$ form a set of generators, by \cref{prop-rigidly-comp-gen}, we conclude that $X\simeq 0$. This concludes the proof of the first recollement. From the axioms (a)--(e) one deduces that 
    \begin{equation}\label{kernel}
    \ker(j^*)\simeq \D{\cat U}
    \end{equation}
    see \cite[Remark A.8.5]{HA}. 
    
    For the second recollement we instead apply \cite[Proposition A.8.20]{HA} so we are reduced to checking that:
    \begin{enumerate}
        \item[(f)] The functor $j_!$ admits left and right adjoints;
        \item[(g)] $\D{\cat U}\simeq\{Y \in \D{\cat V} \mid \Map_{\D{\cat V}}(j_! X, Y)\simeq 0, \; \forall X\in \D{\cat V \, \backslash\, \cat U}\}$. 
    \end{enumerate}
    The only nontrivial part of (f) follows from the adjoint functor theorem using that $j_!$ is extension by zero and so it preserves all limits. Using the standard compact generators for $\D{\cat V \, \backslash\, \cat U}$ and \cref{lem-prop-e_g} we can rewrite part (g) as saying that $\D{\cat U}$ is equivalent to 
    \[
     \{Y \in \D{\cat V} \mid  Y(G)\simeq 0, \; \forall G \in \cat V \, \backslash\, \cat U\}=\ker{j^*}
    \]
    which holds by \eqref{kernel}.
\end{proof}

\begin{Rem}\label{rem-fibre-seq-recollement}
    In the situation of \cref{thm:recollement}, the (co)unit maps of the adjunctions define fibre sequences of functors 
    \begin{align*}
    j_!j^*  \to \mathrm{id} \to i_*i^* \\
    i_*i^\dagger \to \mathrm{id} \to j_*j^*\\
    i_!i^* \to \mathrm{id} \to j_!j^! \\
    j_!j^* \to \mathrm{id} \to i_* i^*
    \end{align*}
    which one can verify directly from the axioms (a)--(e). We leave the details to the interested reader. 
\end{Rem}

\begin{Cor}\label{cor:verdierquotient}
    Let $i\colon \cat U \to \cat V$ be the inclusion of a downwards closed subcategory with complement $j\colon \cat V \, \backslash\, \cat U \to \cat V$. There is a diagram of Verdier sequences
    \[
    \begin{tikzcd}
        &  \D{\cat V \, \backslash\, \cat U}\arrow[d, tail, "j_!"]& \\
       \D{\cat U}\arrow[r, tail, "i_*"] & \D{\cat V}\arrow[d,two heads, "i^*"] \arrow[r,two heads, "j^*"]&\D{\cat V \, \backslash\, \cat U} \\
        &  \D{\cat U}. & 
    \end{tikzcd}
    \]
 Passing to compact objects we obtain a non-unital symmetric monoidal equivalence 
    \[
     \left(\frac{\D{\cat V}^c}{\D{\cat V \, \backslash\, \cat U}^c}\right)^{\natural}\xrightarrow{\sim} \D{\cat U}^c,
    \]
    where $(-)^{\natural}$ denotes idempotent completion (or  \emph{Karoubi envelope}). If $\unit \in \D{\cat V}^c$, the equivalence is also unital. In particular, on spectra $i^*$ induces an embedding of topological spaces, 
        \[
            \Spc(\D{\cat U}^c) \hookrightarrow \Spc(\D{\cat V}^c).
        \]
\end{Cor}
\begin{proof}
    The existence of the Verdier sequences follows formally from \cref{thm:recollement} but we can also give a direct proof. Let us argue for the vertical sequence, the argument for the other being similar. We verify that the functor $i^*$ has the universal property of the Verdier quotient of $\D{\cat V}$ by $\im(j_!)$. Let $F \colon \D{ \cat V} \to \cat T$ be a functor such that $Fj_!\simeq 0$. 
    Using the fibre sequence from \cref{rem-fibre-seq-recollement} we obtain that $F \simeq F i_*i^*$ showing that $F$ factors (up to equivalence) through $i^*$, hence proving the claim. Since $j_!$ and $i^*$ preserve compact objects (see \cref{prop-gen-fun-preserve-compacts} and \cref{prop-incl-preserve-compacts}(d)), we obtain functors on the subcategories of compact objects
    \[
    \begin{tikzcd}
    {\D{\cat V \setminus \cat U}^c }\arrow[r,"(j_!)^c"] & {\D{\cat V}^c} \arrow[r,"(i^*)^c"] & {\D{\cat U}}^c.
    \end{tikzcd}
    \]
    The claim on compact objects follows from the Neeman--Thomason localization theorem \cite[Theorem 2.1]{Neeman92}. We note that with no additional assumptions on $\cat V$, the categories $\D{\cat V}^c$ and $\D{\cat U}^c$ are non-unital symmetric monoidal by \cref{prop-compacts-tt-widely-closed}. Since the functor $(i^*)^c$  preserves tensors, the resulting Verdier quotient equivalence is only non-unital symmetric monoidal. However the additional assumption on $\cat V$ ensures that $\D{\cat V}^c$ and $\D{\cat U}^c$ are unital, and so the functor $(i^*)^c$ is symmetric monoidal. The final statement on Balmer spectra is then a consequence of \cite[Proposition 3.11]{Balmer2005}.
\end{proof}

\begin{Rem}\label{rem-other-Verdier-quotient}
    In the situation of \cref{cor:verdierquotient}, we saw that the functor $j^*$ induces an equivalence of tt-categories
    \[
    \frac{\D{\cat V}}{\D{\cat U}}\xrightarrow{\sim} \D{\cat V \, \backslash\, \cat U}.
    \] 
    However, in this case, the functor $j^*$ does not preserve compact objects in general, so we cannot conclude that there is a similar description for the subcategory of compact objects. Nonetheless we can use the symmetry of the recollement in \cref{thm:recollement} to see that there is a Verdier sequence
    \[
    \begin{tikzcd}
    \D{\cat U}\arrow[r,"i_!"] & \D{\cat V} \arrow[r,"j^!"] & \D{\cat V \, \backslash\, \cat U}
    \end{tikzcd}
    \]
    which now does restrict to compact objects, giving a non-unital symmetric monoidal equivalence
    \[
    \left(\frac{\D{\cat V}^c}{\D{\cat U}^c}\right)^{\natural}\xrightarrow{\sim} \D{\cat V \, \backslash\, \cat U}^c.
    \]
\end{Rem}

\section{Connection to rational global spectra}

In this section we show that the $\infty$-category of rational global spectra, as recalled in \cref{appendix}, is symmetric monoidally equivalent to the derived category of global representations, extending a result of Wimmer \cite{Wimmerthesis}. This section is not essential for the remainder of the paper, though it offers additional motivation for studying global representation theory, as discussed in the introduction. 

Let us start by recalling the main results of the appendix. 
\begin{enumerate}
\item Given a global family of compact Lie groups $\cat F$, we have introduced an $\infty$-category of $\cat F$-global spectra $\Spgl{\cat F}$ by making use of the notion of partially lax limit, see \cref{def-global-spectra}.
\item For any $G \in \cat F$, there is a symmetric monoidal, colimit preserving restriction functor $\res_G \colon \Spgl{\cat F}\to \Sp_G$ into the $\infty$-category of genuine $G$-spectra. The restriction functors admit left adjoints $L_G\colon \Sp_G \to \Spgl{\cat F}$. There is also a geometric fixed points functor $\Phi^G \colon \Spgl{\cat F}\to \Sp$ for any $G \in\cat F$, and the functors $\{\Phi^G\}_{G\in\cat F}$ are jointly conservative.
\item The $\infty$-category of rational global spectra $\Spgl{\cat F}^{\bbQ}$ is the Bousfield localization of $\Spgl{\cat F}$ at the class of rational global equivalences, i.e., at those maps $f$ such that $\res_G(f)$ is rational equivalence for all $G \in \cat F$, see \cref{prop-Qglsp-bousfield}.
\item The $\infty$-category of rational global spectra is compactly generated by the set
\[
\{L_G^\bbQ \unit \coloneqq L_\bbQ L_G \unit \mid G \in\cat F\}
\]
of rationalizations of $L_G\unit$, see \cref{cor-smashing+compactly-gen}.
\end{enumerate}

Next we discuss an $\infty$-categorical model for the derived category $\sfD(\cat U; \bbQ)$.

\begin{Rem}\label{rem-model-derived-cat}
    One model for the $\infty$-category $\sfD(\cat U;\bbQ)$ is the $\infty$-categorical localization of the category of chain complexes $\Ch(\A{\cat U;\bbQ})$ at the quasi-isomorphisms. On the other hand, we observe that 
    \[
    \Ch(\A{\cat U;\bbQ})=\Fun(\cat U^{\op}, \Ch(\bbQ))
    \]
    so we can also describe $\sfD(\cat U;\bbQ)$ as the $\infty$-categorical localization of $\Fun(\cat U^{\op}, \Ch(\bbQ))$ at the pointwise quasi-isomorphisms. In other words, we have an equivalence of $\infty$-categories 
    \begin{equation}\label{sym-mon-model}
    \sfD(\cat U;\bbQ)\simeq \Fun(\cat U^{\op}, \sD(\bbQ)).
    \end{equation}
    We observe that the category $\Ch(\A{\cat U};\bbQ)$ admits a symmetric monoidal structure by pointwise tensor product. It follows from \cite[Lemma 3.18]{BBPSWstructural} that the $\infty$-categorical localization is compatible with the symmetric monoidal structure so we can lift \eqref{sym-mon-model} to a symmetric monoidal equivalence, compare \cite[Proposition 2.2.1.9]{HA}. 
\end{Rem}

We now state the main result of this section.

\begin{Thm}\label{thm-algebraic-model}
    Let $\cat{F}$ be a global family of finite groups. There is a symmetric monoidal equivalence of $\infty$-categories 
\[
\Phi \colon \Spgl{\cat{F}}^\mathbb{Q} \xrightarrow{\sim} \mathsf{D}(\cat{F};\bbQ)
\]
such that $H_*(\Phi X)(G)\cong \pi_* (\Phi^G X)$ for all $X \in \Spgl{\cat F}^\bbQ$ and $G\in \cat F$.
\end{Thm}

\begin{Rem}
    This extends work of Wimmer~\cite[Theorem 3.2.20]{Wimmerthesis} in two directions:
\begin{enumerate}
\item We lift the tensor-triangulated equivalence of homotopy categories to a symmetric monoidal equivalence of $\infty$-categories.
\item We prove the existence of a symmetric monoidal algebraic model for any global family of finite groups, and not just for multiplicative global families. 
\end{enumerate}
\end{Rem}

Constructing the functor witnessing the equivalence will require substantial work. In the next two constructions we will use the notion of partially lax limits without further comment. The reader not familiar with this construction is advised to read \cref{appendix} before continuing.

\begin{Cons}\label{cons-Phi-nat-trans}
 Let $\cat F$ be a global family of compact Lie groups and recall the global orbit category $\Glo{\cat F}$ from \cref{def-global-orbits}. We let $\Epi{\cat F}$ denote the wide subcategory of $\Glo{\cat F}$ spanned by the surjective group homomorphisms. We will construct a natural transformation
 \[
 \Phi \colon \Sp_\bullet \Rightarrow \Delta(\Sp) \in \Fun( \Epi{\cat F}^{\op}, \CAlg(\PrL))
 \]
 between the functor $\Sp_\bullet$ of \cref{cons-functor-Spbullet} (whose partially lax limit gives the $\infty$-category of global spectra) and the constant functor with value $\Sp$. The value of this natural transformation at $G\in \Epi{\cat F}$ will be given by the $G$-geometric fixed points functor $\Phi^G \colon \Sp_G \to \Sp$, which was described in \cref{cons-geometric-fixed-points} using the $G$-space $\widetilde{E}\mathcal{P}_G$. We will proceed in steps. 
 \begin{enumerate}
 \item Consider the full subcategory of $\mathcal{P}_G$-torsion objects
 \[
    \Sp_G^{\mathcal{P}_G\text{-}\mathrm{tors}}\coloneqq\loct{(E\mathcal{P}_G)_+}=\loct{G/H_+ \mid H<G}. 
 \] 
 For a \emph{surjective} continuous group homomorphism $\alpha\colon H \to G$, we have
 \[
 \alpha^*( E\mathcal{P}_G)_+\simeq (E \alpha^*\mathcal{P}_G)_+\in \Sp_G^{\mathcal{P}_H\text{-}\mathrm{tors}}
 \]
 so we obtain a natural transformation $\Sp_\bullet^{\mathcal{P}_\bullet\text{-}\mathrm{tors}} \Rightarrow \Sp_\bullet$ in $\Fun(\Epi{\cat F}^{\op},\PrL)$. 
 \item We define the category of $\mathcal{P}_G$-local objects $\Sp_G^{\mathcal{P}_G\text{-}\mathrm{loc}}$ as the Verdier quotient of $\Sp_G$ by the subcategory of $\mathcal{P}_G$-torsion objects. The associated localization functor $\Sp_G \to \Sp_G^{\mathcal{P}_G\text{-}\mathrm{loc}}$ is given by the functor $\widetilde{E}\mathcal{P}_G \otimes -$. As Verdier quotients of presentable categories are calculated as cofibres, we can apply the functor (see \cite[Remark 1.1.1.7]{HA})
 \begin{align*}
  \hspace{1cm}\cof \colon &\Fun(\Delta^1, \Fun(\Epi{\cat F}^{\op},\PrL))& \to&  &\Fun(\Delta^1, \Fun(\Epi{\cat F}^{\op},\PrL))\\
 &(\Sp_\bullet^{\mathcal{P}_\bullet\text{-}\mathrm{tors}} \Rightarrow \Sp_\bullet)& \mapsto &  &(\Sp_\bullet \Rightarrow \Sp_\bullet^{\mathcal{P}_\bullet\text{-}\mathrm{loc}})
 \end{align*}
 to the natural transformation constructed in (a) and obtain a natural transformation $\Sp_\bullet \Rightarrow \Sp_\bullet^{\mathcal{P}_\bullet\text{-}\mathrm{loc}}$ in $\Fun(\Epi{\cat F}^{\op},\PrL)$. 
 \item In fact, since $\Sp_\bullet \in\Fun(\Epi{\cat F}^{\op}, \CAlg(\PrL))$ and the $\mathcal{P}_\bullet$-torsion objects form an ideal, the  natural transformation $\Sp_\bullet \Rightarrow\Sp_\bullet^{\mathcal{P}_\bullet\text{-}\mathrm{loc}}$ refines to a natural transformation in $\Fun(\Epi{\cat F}^{\op},\CAlg(\PrL))$, compare \cite[Proposition 2.2.1.9]{HA}. 
 \item Now recall that for any $G\in\cat F$, there is a symmetric monoidal equivalence $(-)^G\colon \Sp_G^{\mathcal{P}_G\text{-}\mathrm{loc}} \xrightarrow{\sim} \Sp$, see \cite[VI.5.3]{MM}. Thus for any surjective continuous group homomorphism $\alpha \colon H \to G$, we get a commutative square
 \[
 \begin{tikzcd}
     \Sp_G^{\mathcal{P}_G\text{-}\mathrm{loc}} \arrow[d,"\alpha^*(-)\otimes \widetilde{E}\mathcal{P}_H"'] \arrow[r, "\sim"] & \Sp \arrow[d,"F_\alpha"]\\
     \Sp_H^{\mathcal{P}_H\text{-}\mathrm{loc}} \arrow[r,"\sim"] & \Sp.
 \end{tikzcd}
 \]
 The commutativity of the above diagram forces $F_\alpha$ to be colimit preserving and symmetric monoidal, and so it must be homotopic to the identity functor by the universal property of $\Sp$, see \cite[Corollary 4.8.2.19]{HA}. 
 \item Putting all this together we obtain a natural transformation $\Sp_\bullet \Rightarrow \Delta(\Sp)$ in  $\Fun( \Epi{\cat F}^{\op}, \CAlg(\PrL))$ whose value at $G\in \cat F$ is given by 
    \[  
        (- \otimes \widetilde{E}\mathcal{P}_G)^G
    \]
 which is a model for $\Phi^G$.
 \end{enumerate}
\end{Cons}

\begin{Cons}\label{glo-geo-fix-points}
We claim that the geometric fixed points functor refines to a symmetric monoidal colimit preserving functor
 \[
 \Phi \colon \Spgl{\cat F} \to \Fun(\Epi{\cat F}^{\op}, \Sp), \quad (X \mapsto (G \mapsto \Phi^G X)),
 \]
 where we equip the target with the pointwise symmetric monoidal structure.
 We construct this functor in steps:
 \begin{itemize}
     \item[(1)] The inclusion functor $\iota \colon \Epi{\cat F} \to \Glo{\cat F}$ induces a symmetric monoidal colimit preserving forgetful functor
     \[
     \Spgl{\cat F}\subseteq \laxlim_{\Glo{\cat F}^{\op}}\Sp_\bullet \to \laxlim_{\Epi{\cat F}^{\op}}\Sp_\bullet
     \]
     see \cref{rec-laxlim-sym-mon}.
     \item[(2)] Applying the lax limit functor to the natural transformation from \cref{cons-Phi-nat-trans}, we obtain a colimit preserving symmetric monoidal functor 
     \[
     \laxlim_{\Epi{\cat F}^{\op}} \Sp_\bullet \to \laxlim_{\Epi{\cat F}^{\op}} \Delta(\Sp)
     \]
     whose value at $G$ is given by the $G$-geometric fixed points.
     \item[(3)] It remains to prove that there is a symmetric monoidal equivalence
     \[
     \laxlim_{\Epi{\cat F}^{\op}} \Delta(\Sp)\simeq \Fun(\Epi{\cat F}^{\op}, \Sp),
     \]
     where we equip the functor category with the pointwise symmetric monoidal structure. The fact that the lax limit of the constant functor can be identified with the functor category follows easily from the definitions, see for instance \cref{ex-constant-laxlim}. Thus, the nontrivial part of the claim is to construct a symmetric monoidal functor realizing this equivalence. By the properties of the lax limit (see \cref{rec-laxlim-sym-mon}) we have to construct an object in  
     \[
     \laxlim_{\Epi{\cat F}^{\op}} \Fun^{\mathrm{L},\otimes}(\Fun(\Epi{\cat F}^{\op}, \Sp)^\otimes, \Delta(\Sp^\otimes))
     \]
     which, by \cref{ex-constant-laxlim} again, can be identified with 
     \[
     \Fun(\Epi{\cat F}^{\op},\Fun^{\mathrm{L},\otimes}(\Fun(\Epi{\cat F}^{\op}, \Sp)^{\otimes}, \Sp^{\otimes}) ).
     \]
     Unravelling the definition of the pointwise symmetric monoidal structure, we can construct an object in 
      \[
     \Fun(\Epi{\cat F}^{\op},\Fun(\Fun(\Epi{\cat F}^{\op}, \Sp^{\otimes})\times_{\Fun(\Epi{\cat F}^{\op}, \Fin_{\ast})} \Fin_\ast, \Sp^{\otimes}) )
     \]
     and then verify that it is symmetric monoidal and colimit preserving. This latter category receives a natural functor from 
      \[
     \Fun(\Epi{\cat F}^{\op},\Fun(\Fun(\Epi{\cat F}^{\op}, \Sp^{\otimes}), \Sp^{\otimes}) )
     \]
     induced by the canonical projection. We can rewrite this last category as 
     \[
     \Fun(\Epi{\cat F}^{\op}\times \Fun(\Epi{\cat F}^{\op}, \Sp^{\otimes}), \Sp^{\otimes}) \ni \ev
     \]
     and observe that the evaluation functor (which is adjoint to the identity functor) lives in here. Tracing our steps back, we find that the corresponding functor in fact lands in the subcategory of symmetric monoidal and colimit preserving functors and hence gives the claimed functor.
 \end{itemize}
 The required functor is obtained by combining the previous steps.
 \end{Cons}
We now check that our refined geometric fixed points functor $\Phi \colon \Spgl{\cat F} \to \Fun(\Epi{\cat F}^{\op}, \Sp)$ behaves well with respect to rational $\cat F$-global equivalences. Recall that a rational equivalence in $\Sp$ is a map of spectra which induces an equivalence on rational homotopy groups. 

\begin{Lem}\label{lem-rat-glo-geo-fix}
     The functor $\Phi$ from \cref{glo-geo-fix-points} induces a symmetric monoidal functor on rationalizations:
     \[
     \Phi \colon  \Spgl{\cat F}^\bbQ \to \Fun(\Epi{\cat F}^{\op}, \Sp^\bbQ).
     \]
\end{Lem}
\begin{proof}
    Recall from \cref{lem-geom-fix-rational-eq} that a map of global spectra $f\colon X \to Y$ is a rational $\cat F$-global equivalence if and only if $\Phi^G f$ is a rational equivalence for all $G \in\cat F$.
     It then follows that the functor of \cref{glo-geo-fix-points} descends to the rationalization, and it remains symmetric monoidal as the localization functor $L_\bbQ$ is compatible with the symmetric monoidal structure, compare \cite[Proposition 2.2.1.9]{HA}. 
\end{proof}

We note that $\Ho(\Epi{\cat F})$ is the category whose objects are the groups in $\cat F$, and whose morphisms are conjugacy classes of surjective group homomorphisms, compare \cref{rem-homotopy-category}. This is the category that we previously denoted by $\cat F$, see \cref{not-families}. Write $\pi_{\cat F}\colon \Epi{\cat F}\to \Ho(\Epi{\cat F})=\cat F$ for the canonical projection functor. The next result shows that if we are only interested in global families of finite groups we can replace the indexing category $\Epi{\cat F}$ by its homotopy category without losing information.

\begin{Prop}\label{prop-hoEpi}
     Let $\cat F$ be a global family of finite groups. Then the functor $\pi^*_{\cat F}$ induces a symmetric monoidal equivalence 
     \[
     \Fun(\cat F^{\op}, \Sp^{\bbQ}) \xrightarrow{\sim} \Fun(\Epi{\cat F}^{\op},\Sp^\bbQ).
     \]
\end{Prop}
\begin{proof}
     As both the source and the target have the pointwise symmetric monoidal structure, it is clear that the functor $\pi^*_{\cat{F}}$ is symmetric monoidal. Thus, we only need to check that it is an equivalence.  We observe that  $\unit_{\bbQ}[\Map_{\cat F}(-,G)]$ corepresents the functor which evaluates at $G$, so we deduce that $\Fun(\cat F^{\op}, \Sp^{\bbQ})$ is compactly generated by 
      \[
     \{\unit_{\bbQ}[\Map_{\cat F}(-,G)] \mid G \in \cat F\}.
     \]
     Similarly, $\Fun(\Epi{\cat F}^{\op},\Sp^\bbQ)$ is compactly generated by 
     \[
     \{\unit_{\bbQ}[\Map_{\Epi{\cat F}}(-,G)] \mid  G \in \cat F\}.
     \]
     Next, we claim that for all $G \in \cat F$, the natural map
     \begin{equation}\label{eq-gen-pif}
     \unit_{\bbQ}[\Map_{\Epi{\cat F}}(-,G)]  \xrightarrow{\sim} \pi_{\cat F}^* \unit_{\bbQ}[\Map_{\cat F}(-,G)] 
     \end{equation}
     is an equivalence. This follows from the formula in \cref{rem-homotopy-category} (discarding the components indexed by non-surjective group homomorphisms) and the fact that $BK$ is rationally contractible for any finite group $K$. 
     
     Using these results we show that $\pi_{\cat F}^*$ is an equivalence. We begin by showing that the functor is fully faithful, i.e., that the canonical map on mapping spectra
     \[
     \mathrm{map}_{\Fun(\cat F^{\op}, \Sp^\bbQ)}(X,Y) \to \mathrm{map}_{\Fun(\Epi{\cat F}^{\op}, \Sp^\bbQ)}(\pi^*_{\cat F}X, \pi^*_{\cat F}Y)
     \]
     is an equivalence for all $X$ and $Y$. By a localizing subcategory argument, we can reduce to the case where $X$ is a compact generator, say corresponding to $G \in \cat F$. Using a localizing subcategory argument in the other variable (which uses compactness of $X$), we can further reduce to $Y$ being also a compact generator, say corresponding to $H \in \cat F$. In this case the source of the above map takes the form 
     \[
     \mathrm{map}(\unit_{\bbQ}[\Map_{\cat F}(-,G)],\unit_{\bbQ}[\Map_{\cat F}(-,H)]) \simeq \unit_{\bbQ}[\Map_{\cat F}(G,H)]
     \]
     while the target has the form
     \[
      \mathrm{map}(\unit_{\bbQ}[\Map_{\Epi{\cat F}}(-,G)], \unit_{\bbQ}[\Map_{\Epi{\cat F}}(-,H)])\simeq \unit_{\bbQ}[\Map_{\Epi{\cat F}}(G,H)]
     \]
     using \eqref{eq-gen-pif}.  Thus the above map is an equivalence by the formula in \cref{rem-homotopy-category} once again.      
     Since $\pi_{\cat F}^*$ is fully faithful and clearly preserves colimits, the essential image is a localizing subcategory, which contains the generators by \eqref{eq-gen-pif}. It then follows that the functor is essentially surjective and so an equivalence.      
\end{proof}

 Given finite groups $H,G\in\cat F$, we write $\mathrm{Rep}(H,G)$ for the set of conjugacy classes of group homomorphisms from $H$ to $G$, and set
\[
b_G(-)\coloneqq\bbQ[\mathrm{Rep}(-,G)]\in \A{\cat F;\bbQ},
\]
which we can view as an object of $\D{\cat{F};\bbQ}$. In the next proposition we record three main ingredients from \cite{Wimmerthesis}. Recall the canonical compact generators $L_G^\bbQ\unit$ of $\Spgl{\cat F}^\bbQ$ from \cref{cor-smashing+compactly-gen}.

 \begin{Prop}\label{prop-Wimmer-work}
     Let $\cat F$ be a global family of finite groups. 
     \begin{enumerate}
        \item The $\infty$-category $\sfD(\cat F;\bbQ)$ is compactly generated by the set $\{b_G\mid G \in \cat F\}$.
        \item For any $G,H \in\cat F$ and $k\in\bbZ$, we have
         \[
         \pi_k(\Phi^H L_G^\bbQ \unit)\cong 
         \begin{cases}
             \bbQ[\mathrm{Rep}(H,G)] & \mathrm{if}\; k=0 \\
             0                       & \mathrm{if}\; k \not =0.
         \end{cases}
         \]
         Furthermore, the functor $H \mapsto \pi_0(\Phi^H L_G^\bbQ \unit)$ (with functoriality induced by \cref{lem-rat-glo-geo-fix}) defines an object of $\A{\cat F;\bbQ}$ naturally isomorphic to $b_G$. 
        \item The object $b_G$ corepresents the functor 
         \[
         C\mapsto  \prod_{(H)\leq G} H_*(C)(H)^{W_G(H)}
         \]
         in $\D{\cat{F};\bbQ}$.
     \end{enumerate}
 \end{Prop}
 \begin{proof}
    Part (a) is \cite[Corollary 3.2.4]{Wimmerthesis} and part (c) is \cite[Lemma 3.2.3]{Wimmerthesis}.
    Part (b) combines \cite[Corollary 2.3.21 and Lemma 3.2.12]{Wimmerthesis}.
 \end{proof}

 \begin{Rem}
    The derived category $\sfD(\cat F;\bbQ)$ admits a set of ``canonical'' generators $\{e_G \mid G\in \cat F\}$, see \cref{prop-rigidly-comp-gen}. We warn the reader that these are different from the generators $\{b_G \mid G \in\cat F\}$ defined above. For the proof of \cref{thm-algebraic-model} it will convenient to work with the latter rather than the former as we will see. 
\end{Rem}

We are finally ready to prove the main theorem of this section.

\begin{proof}[Proof of \cref{thm-algebraic-model}]
    We first construct the functor. By \cref{lem-rat-glo-geo-fix} there is a symmetric monoidal functor 
    \[
     \Phi \colon \Spgl{\cat F}^{\bbQ} \to\Fun(\Epi{\cat F}^{\op}, \Sp^\bbQ).
    \]
    Using \cref{prop-hoEpi}, the symmetric monoidal equivalence $\Sp^\bbQ \simeq \D{\bbQ}$, and \cref{rem-model-derived-cat}, we also have symmetric monoidal equivalences
    \[
    \Fun(\Epi{\cat F}^{\op}, \Sp^{\bbQ}) \simeq \Fun({\cat F}^{\op}, \Sp^{\bbQ})\simeq \Fun({\cat F}^{\op}, \D{\bbQ}) \simeq \sfD(\cat F; \bbQ).
    \]
     Putting all this together, we see that geometric fixed points induces a symmetric monoidal functor 
    \[
    \Phi \colon \Spgl{\cat F}^{\bbQ} \to  \sfD(\cat F;\bbQ)
    \]
    such that $H_*(\Phi X)(G)\cong \pi_*(\Phi^G X)$ for all $G \in\cat F$ and $X\in \Spgl{\cat F}^\bbQ$. We observe that $\Phi$ preserves colimits as the geometric fixed points functor preserves colimits and colimits in $\sfD(\cat F;\bbQ)$ are computed pointwise. Moreover, by construction of the functor $\Phi$ and  \cref{prop-Wimmer-work}(b), we have
    \[
    H_*(\Phi(L_G^\bbQ\unit)) \cong H_0(\Phi(L_G^\bbQ\unit))\cong b_G.
    \]
    By taking cycle representatives we can lift this isomorphism to a quasi-isomorphism 
    \begin{equation}\label{bg}
    \Phi(L_G^\bbQ\unit) \simeq b_G.
    \end{equation}
    
    We are now ready to show that the functor $\Phi$ is fully faithful. By a localizing subcategory argument, it suffices to prove that for all $G\in\cat F$ and $X \in \Spgl{\cat F}^\bbQ$, the map
    \[
    \pi_*\Map_{\Spgl{\cat F}^\bbQ}(L_G^\bbQ \unit, X) \xrightarrow{\Phi} \pi_*\Map_{\sfD(\cat F)}(b_G, \Phi X)
    \]
    is an equivalence. By a simple adjunction argument for the left hand side and \cref{prop-Wimmer-work}(c) for the right hand side, we may rewrite the above map as 
    \[
    \pi_*^G(X) \to \prod_{(H)\leq G} H_*(\Phi X)(H)^{W_G(H)}\cong
    \prod_{(H)\leq G} \pi_*(\Phi^H X)^{W_G(H)}. 
    \]
    This map is an equivalence by \cite[Corollary 3.4.28]{Schwedebook} which applies since we are working rationally.

    As $\Phi$ is fully faithful and preserves colimits, the essential image is a localizing subcategory. Since the objects $\{b_G \mid G \in\cat F\}$ generate the derived category by \cref{prop-Wimmer-work}(a) and are in the essential image of $\Phi$ by \eqref{bg}, we conclude that $\Phi$ is essentially surjective. 
\end{proof}

\newpage
\part{The tt-geometry of global representations}
In this part we initiate the study of the tensor triangular geometry of global representations. In particular, we introduce the homological support for objects in the derived category, and define two types of Balmer primes: group primes and family primes. We then calculate the Balmer spectrum for any essentially finite subcategory $\cat U$ and for the category of elementary abelian $p$-groups. Finally, we give a simple criterion to verify that a group prime is isolated in the Balmer spectrum. Throughout, we assume some familiarity with the basic notions of tt-geometry, as developed by Balmer in \cite{Balmer2005}. 

We remind the reader that \cref{hyp} is always in place without further comments. However, in order to ensure that $\D{\cat U}^c$ is symmetric monoidal, and not only non-unital symmetric monoidal, we will occasionally need to refer to the following hypothesis which extends \cref{hyp}.
\setcounter{section}{4}
\setcounter{equation}{-1}
\begin{Hyp}\label{hyp2}
   Let $\cat{U}\subseteq \cat G$ be a subcategory which is full, replete, and widely closed, and for which the tensor unit of $\D{\cat{U}}$ is compact. 
\end{Hyp}
We note that the additional assumption that the unit of $\D{\cat{U}}$ is compact ensures that the full subcategory $\D{\cat{U}}^c$ is symmetric monoidal, so that we may use standard constructions and results from tt-geometry. This hypothesis is satisfied if the tensor unit of $\A{\cat{U}}$ is finitely generated projective by \cref{prop-compact-perfect}, for instance if $\cat{U}$ is unital. Note that if $\cat{U}$ satisfies \cref{hyp2} then so does every downwards closed subcategory of $\cat U$ by \cref{prop-incl-preserve-compacts}(d). 

\setcounter{section}{3}
\section{Homological support and primes}\label{sec:hsupp}
The conventional approach to understanding the spectrum in the rigid context, which works in virtually all known examples, is by probing the given rigid tt-category through a jointly nil-conservative family of tt-functors. The underlying reason that this works so well is that a family of geometric tt-functors $f_i^*\colon \cat{T} \to \cat{S}_i$ between rigidly-compactly generated tt-categories which is jointly nil-conservative, i.e., which detects (weak) rings in $\cat{T}$, induces a surjective map on spectra 
    \[
        \xymatrix{\bigsqcup_{i \in I}\Spc(\cat{S}_i^c) \ar@{->>}[r] & \Spc(\cat{T}^c),}
    \]
see \cite{BCHS2024}. For a subcategory $\cat{U}$ satisfying \cref{hyp2}, we make the same ansatz for the non-rigid tt-category $\D{\cat{U}}$, taking the family of evaluation functors 
    \begin{equation}\label{eq:evaluationfamily}
        (\D{\cat{U}} \to \D{\{G\}})_{G \in \pi_0\cat{U}} 
    \end{equation}
indexed on the isomorphism classes of groups in $\cat{U}$. This family is jointly conservative, so in particular detects weak rings. However, as we will see below, the induced map on spectra captures only a proper subset of the spectrum of $\D{\cat{U}}$ in general, in sharp contrast to the rigid case. This is a source of the rich structure of the spectrum of $\D{\cat{U}}^c$, and necessitates the development of new techniques to compute the spectrum.

To begin with, recall that Balmer \cite{Balmer2005} assigns to any essentially small tt-category $\cat{K}$ its spectrum $\Spc(\cat{K})$ along with a support function $\supp$. The pair $(\Spc(\cat{K}),\supp)$ is universal among support theories that parametrize radical thick ideals in $\cat{K}$ in terms of Thomason subsets of $\Spc(\cat{K})$. When all thick ideals are radical, $\supp$ thus induces a bijection
    \begin{equation}\label{eq:ttclassification}
        \supp\colon \big\{\text{thick ideals in } \cat{K}\big\} \xrightarrow{\cong} \big\{\text{Thomason subsets in } \Spc(\cat{K})\big\}.
    \end{equation}
In non-rigid situations, there can exist non-radical thick ideals. We therefore introduce the following terminology:

\begin{Def}\label{def:standardtt}
    An essentially small tt-category $\cat{K}$ is \emph{standard} if all of its thick ideals are radical. When that is the case, then the corresponding compactly generated tt-category $\cat{T} = \Ind\cat{K}$ is also called \emph{standard}.
\end{Def}

\begin{Prop}\label{prop:standardtt}
    Let $\cat{K}$ be an essentially small tt-category. Then $\cat{K}$ is standard in each of the following cases:
    \begin{enumerate}
        \item $\cat{K}$ is rigid; 
        \item $\cat{K} = \prod_{i \in I}\cat{K}_i$ is a finite product of standard tt-categories $\cat{K}_i$; 
        \item $\cat{K} = \colim_{i \in I}\cat{K}_i$ is a filtered colimit of standard tt-categories $\cat{K}_i$ along tt-functors $f_i^*$. 
    \end{enumerate}    
\end{Prop}

\begin{proof}
    The key observation, due to Balmer \cite[Proposition 4.4]{Balmer2005}, is that $\cat{K}$ is standard if and only if $a \in \thickt{a \otimes a}$ for all $a \in \cat{K}$. 

    The case $(a)$ then follows from the existence of duals, since any dualizable $a$ is a retract of $a \otimes a \otimes \mathbb{D}a$, see \cite[Remark 4.3]{Balmer2005}. For $(b)$, each of the projections $a_i$ of $a$ in $\cat{K}_i$ satisfies $a_i \in \thickt{a_i \otimes a_i}$, hence so does $a$ itself by working coordinatewise. Finally, $(c)$ holds because for any object $a$ in $\cat{K}$ there exists some $i \in I$ and $a_i \in \cat{K}_i$ such that $a \cong f_i^*(a_i)$, so 
        \[
            a \cong f_i^*(a_i) \in f_i^*\thickt{a_i} \subseteq \thickt{f_i^*(a_i \otimes a_i)} = \thickt{a \otimes a},
        \]
    as claimed.
\end{proof}

With this preparation at hand, we now return to $\D{\cat{U}}$. Based on the jointly conservative family of evaluation functors \eqref{eq:evaluationfamily}, we come to the central definition of this section. 

\begin{Def}\label{def:hsupp} 
Let $X \in \D{\cat{U}}$. The \emph{homological support} of $X$ is 
        \[
            \hsupp(X)\coloneqq \{G \in \pi_0\cat{U} \mid X(G) \not \simeq 0\},
        \]
    viewed as a subset of the discrete space $\pi_0\cat{U}$. Note that this is well-defined, because $X(G) \not \simeq 0$ if and only if $X(H) \not \simeq 0$ whenever $G \cong H$. 
\end{Def}

\begin{Rem}\label{rem:hsupp}
    The condition $X(G) \not \simeq 0$ in \cref{def:hsupp} is equivalent to $H_*(X)(G)\not =0$, which inspired the terminology used here. In \cite{Balmer20_bigsupport}, Balmer constructs a notion of homological support $\Supp^h$ for objects in \emph{rigidly}-compactly generated tt-categories, based on the homological spectrum. It is possible to extend $\Supp^h$ to the non-rigid context, and it is work in progress to compare this extension to the homological support introduced in \cref{def:hsupp}.
\end{Rem}

Recall our notation for the upwards closure of a set from \cref{not-families}.

\begin{Exa}\label{ex:hsuppeg}
    For any $G \in \pi_0\cat{U}$, we have $\hsupp(e_G) = \ua(\hsupp(e_G)) = \ua(\{G\})$.
\end{Exa}
Let us record the key properties that the homological support satisfies. 

\begin{Lem}\label{lem:hsupp}
   The homological support satisfies the following properties:
        \begin{enumerate}
            \item $\hsupp(0) = \varnothing$ and $\hsupp(\unit) = \pi_0\cat{U}$;
            \item $\hsupp(\bigoplus_{i \in I} X_i) = \bigcup_{i \in I} \hsupp(X_i)$ for any collection $\{X_i\}_{I} \subseteq \D{\cat{U}}$;
            \item $\hsupp(\Sigma X) = \hsupp(X)$ for all $X \in \D{\cat{U}}$;
            \item $\hsupp(Y) \subseteq \hsupp(X) \cup \hsupp(Z)$ for any triangle $X \to Y \to Z$ in $\D{\cat{U}}$;
            \item $\hsupp(X \otimes Y) = \hsupp(X) \cap \hsupp(Y)$ for all $X,Y \in \D{\cat{U}}$.
        \end{enumerate}
    In particular, if $X \in \thickt{Y}$ in $\D{\cat{U}}^c$ then $\hsupp(X)\subseteq \hsupp(Y)$. Moreover, if $f\colon \cat{U} \to  \cat{V}$ is any functor, then we have
        \begin{equation}\label{eq:hsupprestriction}
            \hsupp(f^*X) = (\pi_0f)^{-1}\hsupp(X).
        \end{equation}
    Finally, if $\cat{U}$ is closed downwards (resp.~upwards) in $\cat V$ with inclusion $j$, then
        \begin{equation}\label{eq:hsupppushforward}
            \hsupp(j_*X) = j\hsupp(X) \;\;(\text{resp.~} \hsupp(j_!X) = j\hsupp(X)).
        \end{equation}
\end{Lem}
\begin{proof}
    Parts (a)--(e) are easy verifications, which we leave to the reader, and \eqref{eq:hsupprestriction} uses that 
        \[
            (f^*X)(G) \cong X(\pi_0f(G))
        \]
    for any $G  \in \cat{U}$. The final claim uses that $j_*$ (resp.~$j_!$) is extension by zero if $\cat{U}$ is closed downwards (resp.~upwards), see \cref{prop-incl-preserve-compacts}.
\end{proof}

Now suppose that $\cat U$ satisfies \cref{hyp2} to ensure that $\D{\cat U}^c$ is a tt-category. Restricting to the compacts $\D{\cat{U}}^c$, the pair $(\pi_0\cat{U},\hsupp)$ thus forms a support datum in the sense of \cite[Definition 3.1]{Balmer2005}. Balmer's \cite[Theorem 3.2]{Balmer2005} therefore supplies a continuous map 
    \begin{equation}\label{eq:groupprimemap}
        \mathfrakp_{-} = \mathfrakp_{-}^{\cat{U}} \colon \pi_0\cat{U} \to \Spc(\D{\cat{U}}^c)
    \end{equation}
such that $\hsupp(X) = \mathfrakp^{-1}(\supp(X))$ for all $X \in \D{\cat{U}}^c$. Unwinding the construction, we arrive at: 

\begin{Def}\label{def:groupprimes}
    The map $\mathfrakp_{-}$ sends a group $G \in \cat U$ to the tt-prime 
    \[
        \mathfrakp_G =\{X \in \D{\cat{U}}^c \mid G \notin \hsupp(X)\} = \{X \in \D{\cat{U}}^c \mid X(G) \simeq 0\}.
    \]
    We refer to these ideals as \emph{group primes}. 
\end{Def}

The formation of the map $\mathfrakp_{-}$ is compatible with functors of collections of finite groups, as expressed in the next lemma.

\begin{Lem}\label{lem:groupprimes_naturality}
    Let $\cat V$ be a subcategory satisfing \cref{hyp2}, let $i\colon \cat{U} \to \cat{V}$ be the inclusion of a down-closed subcategory, and write $\varphi$ for the map induced on spectra by the pullback functor $i^*\colon \D{\cat{V}}^c \to \D{\cat{U}}^c$ which is well-defined by \cref{prop-incl-preserve-compacts}(d). Then there is a commutative diagram of topological spaces
        \[
            \xymatrix{\pi_0\cat{U} \ar[r]^-{\mathfrakp_{-}^{\cat{U}}} \ar[d]_{\pi_0i} & \Spc(\D{\cat{U}}^{c}) \ar[d]^{\varphi} \\
            \pi_0\cat{V} \ar[r]_-{\mathfrakp_{-}^{\cat{V}}} & \Spc(\D{\cat{V}}^{c}).}
        \]
\end{Lem}
\begin{proof}
    Given $G\in \cat{U}$, we compute directly from the definitions:
        \begin{align*}
            \varphi\mathfrakp_{G}^{\cat{U}} & = \{X \in \D{\cat{V}}^{c} \mid i^*(X) \in \mathfrakp_{G}^{\cat{U}}\} \\
            & = \{X \in \D{\cat{V}}^{c} \mid 0 \simeq i^*(X)(G) = X((\pi_0i)(G))\}\\
            & = \mathfrakp_{(\pi_0i)(G)}^{\cat{V}}. \qedhere
        \end{align*}
\end{proof}

\begin{Prop}\label{prop:groupprimes}
    Let $\cat U$ be a subcategory satisfying \cref{hyp2}. Given two groups $G, H \in \cat{U}$, there is an inclusion $\mathfrakp_H \subseteq \mathfrakp_G$ if and only if $G$ is isomorphic to $H$. In particular, the map $\mathfrakp_{-}\colon \pi_0\cat{U} \to \Spc(\D{\cat{U}}^c)$ is injective. 
\end{Prop}
\begin{proof}
    If $G \cong H$, then $\mathfrakp_G = \mathfrakp_H$. For the other direction, suppose that we have $\mathfrakp_H \subseteq \mathfrakp_G$. Then since $e_G \notin \mathfrakp_G$ we have that $e_G \notin \mathfrakp_H$ so $e_G(H) \neq 0$, which means that there is an epimorphism $H \to G$. Let $X$ be the cofibre of the canonical map $e_{H, k} \to \unit$ so that $X(H)\simeq 0$. Then $X \in \mathfrakp_H \subseteq \mathfrakp_G$, which means that $X(G)\simeq 0$ and therefore $e_{H, k}(G) \neq 0$, which means that there is an epimorphism $G \to H$ and $G \cong H$.
\end{proof}

Generalizing the construction of group primes in \cref{def:groupprimes}, consider $\cat U \subseteq \cat V$ and define a thick ideal 
    \[
        \mathfrakp_{\cat U} \coloneqq \lbrace X \in \D{\cat{V}}^c \mid X(G)\simeq 0 \;\; \forall G \in \cat U \rbrace.
    \] 
Note that these may, in general, fail to be prime. 
\begin{Def}\label{def:familyprime}
    If $\mathfrakp_{\cat U}$ is a prime tt-ideal of $\D{\cat{V}}^c$, then we refer to $\mathfrakp_{\cat U}$ as a \emph{family prime}.
\end{Def}

Our next goal is to exhibit conditions under which $\mathfrakp_{\cat U}$ is prime. This is based crucially on the next result, taken from the companion paper \cite[Theorem 7.7 and Corollary 7.11]{BBPSWstructural}, whose proof relies on the fine structure of $\D{\cat{U}}$ studied there.

\begin{Thm}\label{thm:nontorsion}
    Let $\cat U$ be a multiplicative global family and let $X \in \D{\cat{U}}^c$ be nonzero. Then $H_*(X)$ has a torsion-free element: there exists $G \in\cat U$ and $x \in H_*(X)(G)$ such that for any epimorphism $\alpha\colon G' \to G$ in $\cat{U}$, we have $\alpha^* x \not = 0$. Moreover, if $x \in H_*(X)(G)$ is any such torsion-free element, then $e_G \in \thickt{X}$.
\end{Thm}

\begin{Prop}\label{prop:zeroideal}
    Let $\cat U$ be a multiplicative global family. Then the zero ideal is prime in $\D{\cat{U}}^c$. 
\end{Prop}
\begin{proof}
We need to show that if $X$ and $Y$ are nonzero then so is $X \otimes Y$. By \cref{thm:nontorsion}, we know that there exists $G,K \in \cat U$ and torsion-free elements $x_G \in H_i(X)(G)$ and $y_K \in H_j(Y)(K)$ for some integers $i,j$. Let $p_G \colon G \times K \to G$ be the projection onto $G$, and similarly for $p_K$. Put $x=H_i( X(p_G))(x_G)\not = 0$ and $y = H_j(Y(p_K))(y_K)\not = 0$. Thus the image of the element 
    \[
        x \otimes y \in H_i(X)(G\times K)\otimes H_j(Y)(G\times K) \hookrightarrow  H_{i+j}(X \otimes Y)(G \times K)
    \]
is nonzero, so $X \otimes Y \not \simeq 0$. 
\end{proof}

\begin{Cor}\label{cor:familyprime}
    Let $\cat U \subseteq \cat V \subseteq \cat G$ be subcategories with $\cat U $ a multiplicative global family and $\cat V$ satisfying \cref{hyp2}. Then $\mathfrakp_{\cat U}$ is prime in $\D{\cat{V}}^c$.
\end{Cor}
\begin{proof}
By \cref{prop-gen-fun-preserve-compacts} and \cref{prop-incl-preserve-compacts} there is an exact strong monoidal functor $i^* \colon \D{\cat{V}}^c \to \sD(\cat U)^c$ which therefore induces a map $\mathrm{Spc}(i^*)\colon \mathrm{Spc}(\sD(\cat U)^c) \to \mathrm{Spc}(\D{\cat{V}}^c)$ on spectra. Since the zero ideal $0$ is prime in $\sD(\cat U)^c$ by \cref{prop:zeroideal},
    \[
        \mathrm{Spc}(i^*)(0)= \{X \in \D{\cat{V}}^c \mid i^*(X) \simeq 0 \}=\mathfrakp_{\cat U}
    \]
is a prime ideal in $\D{\cat{V}}^c$, as required.
\end{proof}

As a further consequence of the previous result, we deduce the following.

\begin{Cor}\label{cor:infinitedimensional}
    Let $\cat V \subseteq \cat G$ be any family that contains an infinite nested sequence of multiplicative global families, with strict inclusions. Then the spectrum $\Spc(\D{\cat V}^c)$ has infinite Krull dimension. In particular, if $\cat V$ contains the family $\fabp$ of all finite abelian $p$-groups for a prime number $p$, then its associated Balmer spectrum is infinite dimensional.
\end{Cor}
\begin{proof}
     Assume that we have an infinite sequence of strict inclusions
    \[\cat{V}_0 \subset \cat{V}_1 \subset \dots \subset \cat{V}_n \subset \dots \subset \cat V\]
    of multiplicative global families. Each of these families $\cat{V}_n$ gives a family prime $\mathfrakp_{\cat{V}_n}$ by \cref{cor:familyprime}. By construction, these family primes form a nested sequence, since $\cat{V}_{n-1} \subset \cat{V}_n$ means that $\mathfrakp_{\cat{V}_n} \subseteq \mathfrakp_{\cat{V}_{n-1}}$. Take some group $G \in \cat{V}_n$ that does not belong to $\cat{V}_{n-1}$. Then $e_G$ does not belong to $\mathfrakp_{\cat{V}_n}$, but belongs to $\mathfrakp_{\cat{V}_{n-1}}$, since having $e_G(H)\neq 0$ for some $H \in \cat V_{n-1}$ would mean that there exists an epimorphism $H\to G$, which would contradict the fact that $\cat{V}_{n-1}$ is closed under quotients. Therefore these family primes form a strict descending chain of prime ideals and hence $\mathrm{Spc}(\D{\cat{V}}^c)$ is infinite dimensional.

    For the final claim, let $\cat V$ be a family that contains $\fabp$, and consider the subcategory $\fabpex{n} = \{G \in \fabp \mid p^nG = 0\}$ of groups of $p$-exponent less than or equal to $n$. The nested sequence of multiplicative global families
    \[
    \fabpex{1} \subset \fabpex{2} \subset \cdots  \subset \fabpex{n} \subset\ldots \subset  \fabp\]
    shows that the Balmer spectrum of $\D{\fabp}^c$ (and hence of $\D{\cat{V}}^c$) is infinite dimensional. 
\end{proof}

We conclude this section with three examples that showcase the type of phenomena that we encounter in $\Spc(\D{\cat{U}}^c)$; their proofs and generalizations are given later in this paper.

\begin{Exa}\label{exa:cyclicprimes}
    Let $\Cprime$ be the family of cyclic groups of prime order together with the trivial group. The map \eqref{eq:groupprimemap} gives a bijection
        \[
            \pi_0\Cprime \xrightarrow{\cong} \Spc(\D{\Cprime}^c),
        \]
    so every tt-prime is a group prime in this case. The topology on the spectrum is the cofinal topology, with unique accumulation point given by the group prime $\mathfrakp_{1}$ of the trivial group. See \cref{ex-cyclic-groups-prime-order} for more details. The space can thus be drawn as follows:
    \begin{figure}[ht!]
    \[
    \vcenter{\xy
    (0,0)*{\halfcirc[0.5ex]_{\mathfrakp_{C_2}}};
    (15,0)*{\halfcirc[0.5ex]_{\mathfrakp_{C_3}}};
    (30,0)*{\halfcirc[0.5ex]_{\mathfrakp_{C_{5}}}};
    (45,0)*{\ldots};
    (60,0)*{\bullet_{\mathfrakp_{1}}};
    {\ar@{-} (-5,5)*{};(65,5)*{}};
    {\ar@{-} (-5,5)*{};(-5,-5)*{}};
    {\ar@{-} (-5,-5)*{};(65,-5)*{}};
    {\ar@{-} (65,5)*{};(65,-5)*{}};
    \endxy}
    \]
    \caption{An illustration of $\Spc(\D{\Cprime}^c)$. The half-filled circles are clopen, while the bullet represents the accummulation point.}
    \end{figure}
\end{Exa}

\begin{Exa}\label{exa:cyclicp}
    Consider the family $\cat{C}_p$ of cyclic $p$-groups. In this case, \eqref{eq:groupprimemap} induces a map from the one-point compactification of $\pi_0\cat{C}_p$ to the spectrum,
        \[
            (\pi_0\cat{C}_p)^+ \xrightarrow{\cong} \Spc(\D{\cat{C}_p}^c),
        \]
    which turns out to be a homeomorphism. The compactifying point in $(\pi_0\cat{C}_p)^+$ corresponds to the tt-prime 
        \[
            \mathfrakp_{\mathrm{tors}} = \{X \in \D{\cat{C}_p}^c \mid X(C_{p^n}) = 0\;  \forall n\gg 0 \}.
        \]
    See \cref{ex:cyclic-p-groups} for more details. The next picture illustrates this spectrum:
    \begin{figure}[ht!]
    \[
    \vcenter{\xy
    (0,0)*{\halfcirc[0.5ex]_{\mathfrakp_{1}}};
    (15,0)*{\halfcirc[0.5ex]_{\mathfrakp_{C_p}}};
    (30,0)*{\halfcirc[0.5ex]_{\mathfrakp_{C_{p^2}}}};
    (45,0)*{\ldots};
    (60,0)*{\bullet_{\mathfrakp_{\mathrm{tors}}}};
    {\ar@{-} (-5,5)*{};(65,5)*{}};
    {\ar@{-} (-5,5)*{};(-5,-5)*{}};
    {\ar@{-} (-5,-5)*{};(65,-5)*{}};
    {\ar@{-} (65,5)*{};(65,-5)*{}};
    \endxy}
    \]
    \caption{An illustration of $\Spc(\D{\cat{C}_p}^c)$. The half-filled circles are clopen, while the bullet represents the accummulation point.}
    \end{figure}
\end{Exa}

\begin{Exa}\label{exa:elemabp}
    For the family $\cat{E}_p$ of elementary abelian $p$-groups, the spectrum $\Spc(\D{\cat{E}_p}^c)$ turns out to be homeomorphic to the Hochster dual of $\Spec(\Z)$. The group primes form a discrete subspace via \eqref{eq:groupprimemap}, while the extra point corresponds to the zero prime ideal. It is contained in every other prime, and this completely describes the spectrum; see \cref{thm-elementary} for details. 
    \begin{figure}[H]
        \[
        \begin{tikzpicture}[mybox/.style={draw, inner sep=5pt}]
            \node[mybox] (box){
            \begin{tikzcd}
                &&  \bullet_{(0)} &&& \\
                \emptycirc[0.5ex]_{\mathfrakp_{0}} \arrow[urr, rightsquigarrow] & \emptycirc[0.5ex]_{\mathfrakp_{(\Z/p)^{\times 1}}} \arrow[ur, rightsquigarrow]  & \emptycirc[0.5ex]_{\mathfrakp_{(\Z/p)^{\times 2}}} 
                \arrow[u, rightsquigarrow] & \ldots & \emptycirc[0.5ex]_{\mathfrakp_{(\Z/p)^{\times n}}} \arrow[ull, rightsquigarrow]& \ldots 
            \end{tikzcd}
            };
        \end{tikzpicture}
        \]
        \caption{An illustration of $\Spc(\D{\cat{E}_p}^c)$. The open circles represent isolated points, while the bullet corresponds to the unique closed point.}
    \end{figure}
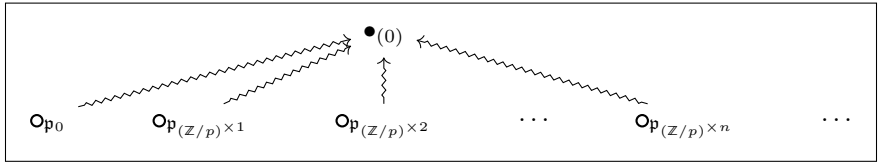
\end{Exa}

\section{The spectrum for essentially finite families}\label{sec:spcfinite}

The first goal of this section is to determine the Balmer spectrum of $\D{\cat{U}}^c$ for essentially finite collections $\cat{U}$ (see \cref{not-families}) satisfying \cref{hyp2}. We show that it is a discrete space consisting only of group primes. Based on this computation, we then establish a number of consequences for more general collections. A special feature of essentially finite families is that the ``characteristic function'' objects of \cref{def-chi-objects} are compact in the derived category:

\begin{Lem}\label{lem:chi-compact}
    Let $\cat U$ be essentially finite. Consider $G \in \cat U$ and $C \in \D{\{G\}}^c$. Then the object $\chi_{G,C}$ is compact in $\D{\cat U}$.
\end{Lem}
\begin{proof}
     Recall from \cref{prop-chi-VandQ}(a) that $\chi_{G,C}\simeq\chi_{G, H_*(C)}$. We note that $H_*(C)$ is finite dimensional as $C$ is compact. Then compactness of $\chi_{G,H_*(C)}$ follows from \cite[Proposition 7.4]{BBPSWstructural}.
\end{proof}

\begin{Prop}\label{prop:spcfinite}
    Let $\cat U$ be essentially finite and consider $X \in \D{\cat{U}}^c$. Then 
        \[
            \thickt{X}=\thickt{\chi_{G,k} \mid G \in \hsupp(X)} \subseteq \D{\cat U}^c.
        \]
        Thus for $Y \in \D{\cat U}^c$, we have $\hsupp(X) \subseteq \hsupp(Y)$ if and only if $X \in \thickt{Y}.$
\end{Prop}
\begin{proof}
    The ``$\supseteq$'' containment holds without the compactness assumption on $X$: if $G \in \hsupp(X)$, then we can use \cref{lem-prop-e_g}(b) to construct a map $e_{G, X(G)} \to X$ which is a quasi-isomorphism at $G$. Tensoring with $\chi_{G,k}$, we obtain a quasi-isomorphism 
    \[
    \chi_{G, X(G)}\coloneqq\chi_{G,k} \otimes e_{G, X(G)} \xrightarrow{\sim} \chi_{G,k} \otimes X,
    \]
    using \cref{def-chi-objects}. Therefore, by \Cref{prop-chi-VandQ}, we have
        \[
            \chi_{G, k} \in \thickt{\chi_{G, X(G)}} \subseteq \thickt{X}.
        \]
    For the other containment, using that $\pi_0\cat U$ is finite, we prove by induction on $n=|\hsupp(X)|$ that $X \in \thickt{\chi_{G,k}\mid G \in \hsupp(X)}$. If $n=0$, then $X\simeq 0$ and the claim is clear. Now let $n\geq 1$ and suppose that the claim holds for all objects $Z \in \D{\cat{U}}^c$ with $|\hsupp(Z)|<n$. Let $\{G_1, \ldots, G_n\}$ be a complete collection of representatives for the isomorphism classes of groups in $\hsupp(X)$ and choose $G_n$ maximal among the $G_i$s with respect to $\gg$. Set $V_n=X(G_n)$. Then we claim that there is a map $f_n\colon \chi_{G_n,V_n}\to X$ which is a quasi-isomorphism at $G_n$. To see this, we start with a map $g_n \colon e_{G_n,V_n}\to X$ which is a quasi-isomorphism at $G_n$, see \cref{lem-prop-e_g}(b). By \cref{lem-chi-eGs} there is a triangle
        \[
            F \to e_{G_n,V_n} \to \chi_{G_n,V_n}
        \]
    with $F \in \loc{e_{H}\mid H \in \ua(\{G_n\}) \setminus \{G_n\}}$. By our choice of $G_n$, we have $\hsupp(X)\cap (\ua(\{G_n\}) \setminus \{G_n\})=\emptyset$. It follows that any map from $F$ to $X$ must be zero, so $g_n$ factors through the desired map $f_n\colon \chi_{G_n,V_n} \to X$. Now let $C=\mathrm{cone}(f_n)$, so that 
    $X \in \thickt{\chi_{G_n,V_n}, C}$. Note that $\hsupp(C)=\hsupp(X)\backslash\{G_n\}$ so by the induction hypothesis 
        \[
            X \in \thickt{\chi_{G_n,V_n}, C}\subseteq\thickt{\chi_{G_n,V_n}, \chi_{G_i,k} \mid 1\leq i < n}.
        \]
    To conclude the proof we apply \cref{prop-chi-VandQ}. The final statement is an immediate consequence of the first claim.
\end{proof}

\begin{Thm}\label{thm:spcfinite}
    Let $\cat U$ be essentially finite and satisfying \cref{hyp2}. Then $\D{\cat{U}}^c$ is standard (in the sense of \cref{def:standardtt}) and the map \eqref{eq:groupprimemap} induces a homeomorphism $\pi_0\cat{U} \cong \Spc(\D{\cat{U}}^c)$. Consequently, $\Spc(\D{\cat{U}}^c)$ is a finite discrete space consisting only of the group primes. 
\end{Thm}
\begin{proof}
    Any thick ideal in $\D{\cat{U}}^c$ is necessarily finitely generated by \Cref{prop:spcfinite} and the fact that the family is essentially finite. The same proposition shows that for any thick ideals $\cat{J}_1,\cat{J}_2$, we have
        \[
            \cat{J}_1 \subseteq \cat{J}_2 \iff \hsupp(\cat{J}_1) \subseteq \hsupp(\cat{J}_2).
        \]
    In particular, $\cat{J}_1$ coincides with its radical as they have the same homological support, so $\D{\cat{U}}^c$ is standard. Moreover, it follows that the pair $(\pi_0\cat{U},\hsupp)$ satisfies the assumptions of \cite[Theorem 16]{BalmerICM}, so that the map $\mathfrakp_{-}$ is a homeomorphism. 
\end{proof}

\begin{Rem}\label{rem:spcfinite}
    More generally, in \cite{Xu2014}, Xu computes the spectrum of the derived category of representations of an essentially finite EI category with coefficients in an arbitrary field $k$; see also \cite{Wang2019} for an alternative proof and \cite{LiuSierra2013} for an analogous result for derived representations of a quiver. Note that their bounded derived categories agree with our categories by \cite[Proposition 7.4]{BBPSWstructural}. As already observed there, these computations in particular demonstrate the failure of the (categorical) structure presheaf on the spectrum to satisfy the sheaf condition. 
\end{Rem}

As a first application of \cref{thm:spcfinite}, we extract a structural property detected by the homological support for arbitrary collections of finite groups (\cref{prop:hsupptt+}). We will begin with two auxiliary lemmas.  

\begin{Lem}\label{lem:upwardhsupp}
    Let $i \colon \cat U \to \cat V$ be the inclusion of a down-closed subcategory and $j \colon \cat V \, \backslash \, \cat{U} \to \cat V$ the inclusion of its complement, which is closed upwards. For $X \in \D{\cat{V}}$, we have:
        \begin{enumerate}
            \item if $\hsupp(X) \subseteq \cat V \, \backslash \, \cat{U}$, then the counit map $j_!j^*X \xrightarrow{\simeq} X$ is an equivalence;
            \item $X \in \loct{e_{G}\mid G \in \ua(\hsupp(X))}$. If $X$ is compact, we may  take the thick ideal instead of the localizing tensor-ideal. 
        \end{enumerate}
\end{Lem}
\begin{proof}
  From \cref{rem-fibre-seq-recollement}, we obtain a cofibre sequence
        \[
            j_!j^*X \to X \to i_*i^*X.
        \]
    By assumption on the homological support and \eqref{eq:hsupprestriction}, we deduce that $i^*X \simeq 0$, hence the counit map $j_!j^*X \to X$ is an equivalence. 

    Now the first claim in (b) follows from (a) applied to $\cat V \, \backslash \, \cat{U} = \ua(\hsupp(X))$, using the fact that $j_!$ preserves tensors (\cref{prop-incl-preserve-compacts}(a)):
        \begin{align*}
            X \simeq j_!j^*X & \in j_!\loct{e_{G}\mid G \in \ua(\hsupp(X))} \\
                & \subseteq \loct{j_!e_{G}\mid G \in \ua(\hsupp(X))} \\
                & = \loct{e_{G}\mid G \in \ua(\hsupp(X))}.
        \end{align*}
    Finally, the claim about compact $X$ is a formal consequence of this and the Neeman--Thomason localization theorem, since the $e_G$s are compact, see for example \cite[Theorem 7.1]{Greenlees2019}.
\end{proof}

\begin{Lem}\label{lem:downwardhsupp}
    Let $\cat{V}\subseteq \cat G$ be a subcategory. Write $i\colon \cat{U} \to \cat{V}$ for the inclusion of some down-closed subcategory. For $Y,Z \in \D{\cat{U}}$, the cofibre of the oplax monoidal structure on $i_!$ satisfies 
        \[
            \mathrm{cof}(i_!(Y \otimes Z) \to i_!(Y) \otimes i_!(Z)) \in \loct{e_G\mid G \in \cat V \, \backslash \, \cat{U}}.
        \]
    In particular, for $X \in \D{\cat{U}}^c$, we have 
        \[
            i_!\thickt{X} \subseteq \thickt{\{i_!X\}\cup\{e_G\mid G\in \cat V \, \backslash \, \cat{U}\}}.
        \]
\end{Lem}
\begin{proof}
    Let $j\colon  \cat V \, \backslash \, \cat{U} \to \cat{V}$ denote the inclusion of the complement. For any $Y,Z \in \D{\cat{U}}$, the oplax monoidal structure on $i_!$ gives a cofibre sequence
        \[
            i_!(Y \otimes Z) \to i_!(Y) \otimes i_!(Z) \to C
        \]
    in $\D{\cat{V}}$. Since $i^*$ is strong monoidal and the unit transformation $\mathrm{id} \to i^*i_!$ is an equivalence, we deduce that $i^*(C) \simeq 0$. Substituting this into the cofibre sequence 
        \[
            j_!j^*C \to C \to i_*i^*C
        \]
    from \cref{rem-fibre-seq-recollement} yields $C \simeq j_!j^*(C) \in \loct{e_G\mid G \in \cat V \, \backslash \, \cat{U}}$, as in \cref{lem:upwardhsupp}(b). In particular, 
    \[
    i_!(X\otimes Z) \in \thick{i_!(X) \otimes i_!(Z), C} \subseteq \thickt{\{i_!X\}\cup\{e_G\mid G\in \cat V \, \backslash \, \cat{U}\}}.
    \]
    The claim about thick ideals follows since $i_!$ is exact.
\end{proof}

This lemma motivates the following notation. Recall that for a subcategory $\cat V\subseteq \cat G$, we write $\cat V_{\leq n}$ (respectively $\cat V_{>n}$) for the subcategory of $\cat V$ of groups of cardinality less than or equal to $n$ (respectively greater than $n$).

\begin{Def}\label{def:ttideal+}
   For $X \in \D{\cat{V}}^c$ and any $n \geq 0$, we write 
   \[
   \ua(\hsupp(X))_{>n} = \{G \in \pi_0\cat{V}_{>n} \mid \exists H \in \hsupp(X)\colon G \gg H\}.
   \]
   We then define an auxiliary thick ideal
        \[
            \mathrm{thick}_{\otimes,n}^{+}\langle X\rangle \coloneqq  
            \thickt{\{X\}\cup\{e_G\mid G\in\ua(\hsupp(X))_{>n}\}}.
        \]
    For varying $n$ these ideals form a nested sequence; set 
    \[
    \mathrm{thick}_{\otimes}^{+}\langle X\rangle \coloneqq \bigcap_{n>0}\mathrm{thick}_{\otimes,n}^{+}\langle X\rangle.
    \]
\end{Def}

\begin{Prop}\label{prop:hsupptt+}
    Let $\cat V\subseteq \cat G$ be a subcategory and consider $X,Y\in\D{\cat{V}}^c$ with $\hsupp(X)\subseteq \hsupp(Y)$.  Then $X\in\mathrm{thick}_{\otimes}^{+}\langle Y \rangle$. 
\end{Prop}
\begin{proof}
    We first prove the claim in the special case $\ua{(\hsupp(Y))} = \cat{V}$. To this end, pick a finite set $\cat{S} \subseteq \cat{V}$ such that $X$ and $Y$ are contained in $\thick{e_{G}\mid G \in \cat{S}}$. Let $n \geq \max\{|G| \mid G \in \cat{S}\}$ and write $i \colon \cat{V}_{\leq n} \to \cat{V}$.  This guarantees that the corresponding counit maps are equivalences: $i_!i^*(X) \simeq X$ and  $i_!i^*(Y) \simeq Y$, as can be checked on the generators $e_G$ for $G \in \cat{S}$. By \eqref{eq:hsupprestriction}, we have $\hsupp(i^*X) \subseteq \hsupp(i^*Y)$, so \cref{prop:spcfinite} implies $i^*X \in \thickt{i^*Y}$. Applying $i_!$ and appealing to \cref{lem:downwardhsupp}, we get
        \begin{align*}
            X \simeq i_!i^*(X) \in i_!\thickt{i^*Y} & \subseteq \thickt{\{i_!i^*Y\}\cup\{e_G\mid G\in \pi_0\cat{V}_{>n}\}} \\
            & = \thickt{\{Y\}\cup\{e_G\mid G\in\ua(\hsupp(Y))_{>n}\}}.
        \end{align*}
    Since $n$ may be taken arbitrarily large, we conclude that $X \in \mathrm{thick}_{\otimes}^{+}\langle Y \rangle$ when $\ua{(\hsupp(Y))} = \cat{V}$.

    The general case can then be deduced as follows: Write $j\colon \ua{(\hsupp(Y))} \to \cat{V}$ for the inclusion. The counit map $j_!j^* \to \mathrm{id}$ is an equivalence (\cref{lem:upwardhsupp}(a)) on $X$ and $Y$. In addition, $j_!$ preserves tensors (\cref{prop-incl-preserve-compacts}(a)), $e_G$s (\cref{rem-restriction-gen}), and homological support \eqref{eq:hsupppushforward}, so
        \[
            j_!\mathrm{thick}_{\otimes,n}^{+}\langle j^*X\rangle \subseteq \mathrm{thick}_{\otimes,n}^{+}\langle j_!j^*X\rangle = \mathrm{thick}_{\otimes,n}^{+}\langle X\rangle
        \]
        for any $n \geq 0$. By \eqref{eq:hsupprestriction}, $\hsupp(j^*X) \subseteq \hsupp(j^*Y)$. Using the claim for $\cat{V} = \hsupp(j^*Y)$ as just established, we thus get
        \[
            X \simeq j_!j^*(X) \in j_!\mathrm{thick}_{\otimes,n}^{+}\langle j^*Y\rangle \subseteq \mathrm{thick}_{\otimes,n}^{+}\langle Y\rangle.
        \]
    Taking the intersection for varying $n$ then gives the desired claim. 
\end{proof}

\section{Embedding group primes}\label{sec:ttembedding}
In this section we give a criterion for verifying if a group prime is isolated in the Balmer spectrum. We then deduce that the group prime map \eqref{eq:groupprimemap} is an open embedding for any downwards closed subcategory of finite $p$-groups. 

\begin{Def}\label{def:downclosure}
    For $G \in \cat{U}$, we define the \emph{down-closure} of $G$ in $\cat U$ by
    \[
    \cat U_{\ll G} \coloneqq \{H \in \cat{U}\mid G \gg H\}
    \]
    and denote its complement by $\cat U_{\not \ll G} \coloneqq \cat U \, \backslash\, \cat U_{\ll G}.$
\end{Def}

The next lemma highlights the relevance of the previous subcategories.

\begin{Lem}\label{lem:lambdaG}
    For any $K \in \cat{U}$, we have
        \[  
            e_K \in \mathfrakp_G \iff K \notin \; \cat U_{\ll G} \iff K \in \; \cat U_{\not \ll G}.
        \]
\end{Lem}
\begin{proof}
    By definition, $e_K\in \mathfrakp_G$ if and only if $e_K(G) = 0$, which holds if and only if $G \not \gg K$. This is equivalent to $K \notin \cat U_{\ll G}$ and to $K \in \cat U_{\not \ll G}$.
\end{proof}
As a consequence we describe the derived category $\D{\{G\}}$ as a Verdier quotient. 

\begin{Prop}\label{prop:groupprimeVerdier}
    For any $G \in \cat{U}$, the evaluation functor
    $\ev_G \colon \D{\cat U} \to \D{\{G\}}$
    induces a symmetric monoidal equivalence
    \[
            \D{\cat{U}}/\loc{\mathfrakp_G} \simeq \D{\{G\}}.
    \]
    After passing to compact objects, there is a non-unital symmetric monoidal equivalence
        \[
            (\D{\cat{U}}^{c}/\mathfrakp_G)^{\natural} \simeq \D{\{G\}}^{c}
        \]
        which is symmetric monoidal if $\cat U$ satisfies \cref{hyp2}.
\end{Prop}
\begin{proof}
      Let us prove the first claim. To this end, let $i \colon \cat U_{\ll G} \to \cat{U}$ and $j \colon \cat U_{\not \ll G} \to \cat U$ denote the inclusions. \Cref{lem:lambdaG} shows that there is a containment  
     \[
     \mathrm{im}(j_!)=\loc{e_K\mid K \in \cat U_{\not \ll G}} \subseteq \loc{\mathfrakp_G}.
     \]
    It follows that the Verdier quotient $\D{\cat U}/\loc{\mathfrakp_G}$ can be calculated in two steps: by first nullifying $\mathrm{im}(j_!)$, and then by nullifying the image of $\loc{\mathfrakp_G}$ in the quotient $\D{\cat U}/\mathrm{im}(j_!)$. Now \cref{cor:verdierquotient} tells us that the nullification of $\mathrm{im}(j_!)$ is equivalent to $\D{\cat U_{\ll G}}$ via the functor $i^*$. Therefore we have a factorization as in the following commutative square 
        \[
            \xymatrix{
            \loc{\mathfrakp_G} \ar[r]\ar[d]^{i^*} &\D{\cat{U}}\ar[r] \ar[d]^{i^*} & \D{\cat{U}}/\loc{\mathfrakp_G} \ar@{=}[d] \\
            \loc{\mathfrakp_G} \ar[r] &\D{\cat{U}_{\ll G}} \ar[r] & \D{\cat{U}_{\ll G}}/\loc{\mathfrakp_G} \ar[r] & \D{\{G\}}.
            }
        \]
   This reduces the claim to $\cat{U}_{\ll G}$ which is essentially finite. In other words, we need to prove that the right most horizontal arrow in the above diagram is an equivalence. Write $\cat{U}_{< G}\coloneqq\cat{U}_{\ll G}\, \backslash \,\{G\}$, and let $k\colon \cat{U}_{< G} \to \cat{U}_{\ll G}$ and $h \colon \{G\}\to \cat U_{\ll G}$ be the inclusions. In light of \cref{cor:verdierquotient}, the functor $h^*$ induces an equivalence 
   \[
   \frac{\D{\cat U_{\ll G}}}{\im(k_*)}\xrightarrow{\sim} \D{\{G\}}
   \]
   so it suffices to verify the equality
   \[
   \loc{\mathfrakp_G}=\mathrm{im}(k_*).
   \]
    Since $k_*$ is given by extension by zero and $\cat{U}_{\ll G}$ is essentially finite so the objects $\chi_{H,k}$s are compact (\cref{lem:chi-compact}), we have 
        \[
            \mathrm{im}(k_*) = \loc{\chi_{H,k}\mid H \in \cat U_{<G}} = \loc{\mathrm{im}(k_*)^c}.
        \]
    The claim then follows because both $\mathfrakp_G$ and $\im(k_*)^c$ agree with 
    \[
    \thickt{\chi_{H,k}\mid H \in \cat U_{<G}}\subseteq \D{\cat U_{\ll G}}^c
    \]
    by \cref{prop:spcfinite}.
    The second claim follows from the first one together with the Neeman--Thomason localization theorem \cite[Theorem 2.1]{Neeman92}. Recall that in general $\D{\cat U}^c$ is only non-unital symmetric monoidal, however under \cref{hyp2} it also admits a unit object.
\end{proof}   

This result leads to a strengthening of \cref{prop:groupprimes}:

\begin{Cor}\label{cor:groupprimes_generic}
    Let $\cat U$ be a subcategory satisfying \cref{hyp2}. For any $G \in \cat{U}$, the group prime $\mathfrakp_G$ is a maximal prime in $\Spc(\D{\cat{U}}^{c})$, that is, if $\mathfrakp \in \Spc(\D{\cat{U}}^{c})$ is such that $\mathfrakp_G \subseteq \mathfrakp$, then $\mathfrakp_G = \mathfrakp$.
\end{Cor}
\begin{proof}
    The Verdier localization $\D{\cat{U}}^{c}\to (\D{\cat{U}}^{c}/\mathfrakp_G)^{\natural}$ induces a bijection between the spectrum of the quotient and those primes $\mathfrakp$ in $\D{\cat{U}}^{c}$ containing $\mathfrak{p}_G$, see \cite[Proposition 3.11]{Balmer2005}. Using \cref{prop:groupprimeVerdier}, we thus get a bijection
        \[
            \Spc(\D{\{G\}}^{c}) \cong \Spc((\D{\cat{U}}^{c}/\mathfrakp_G)^{\natural}) \cong \{\mathfrakp \in \Spc(\D{\cat{U}}^c) \mid \mathfrakp_G \subseteq \mathfrakp\}.
        \]
    The claim follows, because $\Spc(\D{\{G\}}^{c})$ is a point by \cref{thm:spcfinite}.  
\end{proof}
\begin{proof}[Alternative proof]
    More concretely, we could argue as follows.  Consider a compact object $X\notin\mathfrak{p}_G$, and an arbitrary $Y \in \D{\cat{U}}^c$; it suffices to show that $Y$ lies in the thick ideal $\cat{J}$ generated by $\mathfrak{p}_G\cup\{X\}$. By \cref{lem-prop-e_g}(b) there is a canonical map $e_{G,X(G)}\to X$, and we let $X'$ be the cofibre.  By construction we have $X'(G)\simeq 0$ so $X'\in\mathfrak{p}_G$ and $e_{G,X(G)}\in\cat{J}$. 
 We are assuming that $X\notin\mathfrak{p}_G$ so $X(G)\not\simeq 0$ so $e_{G,X(G)}$ generates the same thick ideal as $e_G$ by \cref{cor-eG-eGV}.  We therefore have $e_G\in\cat{J}$, and so $e_{G,V}\in\cat{J}$ for any $V \in \D{\{G\}}^c$.  We have another cofibre sequence $e_{G,Y(G)}\to Y\to Y'$ with $e_{G,Y(G)}\in\cat{J}$ and $Y'\in\mathfrak{p}_G\subseteq\cat{J}$, so $Y\in\cat{J}$ as required.
\end{proof}

\begin{Rem}
    Although all group primes are maximal primes, the converse does not hold in general. Indeed, the spectrum $\Spc(\D{\cat{C}_p}^c)$ of \cref{exa:cyclicp} is Hausdorff, hence all of its points are maximal primes, but $\mathfrakp_{\infty}$ is not a group prime.  
\end{Rem}

Recall that a point $x\in X$ in a topological space is said to be \emph{isolated} if the set $\{x\}$ is open in $X$; when $X$ is non-Hausdorff, this does not imply that $X$ is the coproduct of $\{x\}$ and $X\setminus\{x\}$. We now state our criterion for verifying whether a group prime is isolated in the Balmer spectrum of $\D{\cat{U}}^{c}$. 

\begin{Prop}\label{lem:isolation_criterion}
    Let $\cat U$ be a subcategory satisfying \cref{hyp2} and let $G \in \cat{U}$. If $\pi_0 \cat U_{\not \ll G}$ has finitely many minimal elements under $\gg$, then $\mathfrakp_G \in \Spc(\D{\cat{U}}^{c})$ is isolated. 
\end{Prop}
\begin{proof}
    Let $\{G_1,\ldots, G_n\}$ be a complete list of the minimal elements in $\pi_0 \cat U_{\not \ll G}$ and consider the set $\cat S=\{e_{G_1},\ldots, e_{G_n}\}$. 
    We first claim that 
    \[
    \D{\cat U_{\not \ll G}}^c=\thickt{\cat S}.
    \]
   In fact it will be enough to show that $e_K \in \thickt{\cat S}$ for all $K \in \cat U_{\not \ll G}$. By definition of minimal elements, given $K\in \cat U_{\not \ll G}$ there exists $G_i$ and an epimorphism $K \gg G_i$. Then we know by \cite[Proposition 4.11]{PolStrickland2022} that there exists some representation $W$ so that $e_{K,W}$ is a retract of $e_{K}\otimes e_{G_i}$ (this is because the graph subgroup of the epimorphism $K \to G_i$ is a wide subgroup of $G_i\times K$). Now by \cref{cor-eG-eGV} we know that $e_{K}$ and $e_{K,W}$ generate the same thick ideal so we conclude that $e_K \in \thickt{e_{G_i}} \subseteq \thickt{\cat S}$, verifying our claim. 

    Next, let $i\colon \cat U_{\ll G} \to \cat{U}$ denote the inclusion and write $j\colon\cat U_{\not \ll G} \to \cat{U}$ for the inclusion of the complement. By the previous paragraph we can assume that $\D{\cat U_{\not \ll G}}$ is generated as a thick ideal by a finite set $\cat S$. By \cref{cor:verdierquotient}, there is a Verdier localization
        \[
            i^*\colon \D{\cat{U}}^{c} \to \D{\cat U_{\ll G}}^{c} 
        \]
    with kernel $\thickt{j_!\cat{S}}$. Note that the finiteness of $\cat{S}$ guarantees that the Thomason subset $\supp(j_!\cat{S}) \subseteq \Spc(\D{\cat{U}}^{c})$ is in fact closed. Therefore, $i^*$ induces an \emph{open} embedding $\Spc(\D{\cat U_{\ll G}}^{c}) \hookrightarrow \Spc(\D{\cat{U}}^c)$ on spectra. By \cref{thm:spcfinite}, the space $\Spc(\D{\cat U_{\ll G}}^{c})$ is discrete; in particular, $\mathfrakp_G$ viewed as a point in this space is isolated. It thus follows that $\mathfrakp_G$ is also isolated in $\Spc(\D{\cat{U}}^{c})$, as desired.
\end{proof}

\begin{Rem}\label{rem:isolation_criterion}
    We can make the previous proposition more explicit. To this end, consider an essentially small tt-category $\cat{K}$ , and let $\mathfrakp$ be a point of $\Spc(\cat{K})$. Then we claim that $\mathfrakp$ is isolated if and only if the following conditions hold:
    \begin{enumerate} 
        \item it is a maximal prime ideal in $\Spc(\cat{K})$; and
        \item there is an object $X\in\cat{K}$ with $\mathfrak{p}=\sqrt{\thickt{X}}$.
    \end{enumerate}
Since prime ideals are radical, the second condition is implied by the existence of some object $X$ with $\mathfrak{p}=\thickt{X}$. To prove the claim, we first note that \cite[Proposition 2.14]{Balmer2005} shows that $\mathfrakp$ is isolated if and only if there exists some $X \in \cat{K}$ such that $\{\mathfrakp\}=\Spc(\cat{K})\setminus\supp(\sqrt{\thickt{X}})$. Note that the latter set is always closed upwards under inclusions of prime ideals. Since the complement of $\supp(\mathfrakp)$ is  precisely $\SET{\mathfrakq}{\mathfrakq \supseteq \mathfrakp}$ (see \cite[Lemma 4.6]{Balmer2005}), we thus see that $\mathfrakp$ is isolated if and only if
    \[
        \{\mathfrakp\} = \Spc(\cat{K})\setminus\supp(\sqrt{\thickt{X}}) = \Spc(\cat{K})\setminus\supp(\mathfrakp) = \SET{\mathfrakq}{\mathfrakq \supseteq \mathfrakp}.
    \]
The middle and composite displayed equality translate into Conditions (b) and (a) above, respectively, thereby finishing the proof of the claim. 

Under the assumptions of \cref{lem:isolation_criterion}, we can then specify an explicit generator for the group prime $\mathfrakp_G$: Suppose we have a finite collection of elements $H_1,\dotsc,H_r\in\cat{U}_{\not\leq G}$ such that every $H\in\cat{U}_{\not\leq G}$ admits an epimorphism to some $H_i$.  Let $P_G$ be the cofibre of the evident augmentation $e_{G,k}\to\unit$, then
    \[
        \mathfrak{p}_G=\thickt{P_G \oplus \bigoplus_ie_{H_i}}.
    \]
This gives an alternative approach to \cref{lem:isolation_criterion}.
\end{Rem}

\begin{Rem}
In \cref{exa:cyclicprimes} we saw that the group prime $\mathfrakp_{C_p} \in \Spc(\Cprime)$ is isolated. However $\Cprime_{\not \ll C_p}=\{C_q \mid q \not =p \; \mathrm{prime}\}$ has infinitely many minimal elements. This shows that our criterion only gives a sufficient condition for isolation.   
\end{Rem}

\begin{Cor}\label{prop:pgroup_embedding}
    Let $p$ be a prime number and let $\cat U$ be a downwards closed subcategory of finite $p$-groups. Then the map $\mathfrakp_{-}\colon \pi_0\cat{U} \to \Spc(\D{\cat{U}}^c)$ is an open embedding.
\end{Cor}
\begin{proof}
    We will verify the criterion of \cref{lem:isolation_criterion}. To this end, fix some $G\in \cat U$ and let $K$ be a minimal element in $\pi_0 \cat U_{\not \ll G}$, and note that $K$ must be nontrivial. Standard $p$-group theory tells us that the centre $Z(K)$ is nontrivial and that there is a subgroup $C \leq Z(K)$ of order $p$. As $C \cong C_p$ is central, it is normal and so we have an extension $ C_p \hookrightarrow K \twoheadrightarrow Q$. By minimality of $K$, we must have $Q \in \cat U_{\ll G}$ since $\cat U$ is closed downwards. Therefore to prove that there are only finitely many minimal elements $K$, it suffices to show that:
   \begin{enumerate}
       \item there are finitely many groups $Q\in \pi_0\cat U_{\ll G}$;
       \item for a fixed $Q$ there are only finitely many groups $K \in \pi_0\cat{G}$ which fit into an extension of the type $C_p \hookrightarrow K \twoheadrightarrow Q$.
   \end{enumerate}
   Part (a) is clear since any group in $\cat U_{\ll G}$ must have cardinality less than or equal to $|G|$. For part (b) note that if we consider extensions where the kernel is $C_p$, then they are
 classified by the finite group $H^2(Q;C_p)$.
\end{proof}

\begin{Rem}
    It is conceivable that the hypothesis that $\cat{U}$ is downwards closed in the previous corollary could be relaxed. However, it cannot be removed completely, as the conclusion of the corollary does not hold for some subcategories $\cat U$ of finite $p$-groups. Some examples are given by extraspecial $p$-groups. We demonstrate this in the following explicit example. 
    
    Let $p=2$, write $Q_8$ for the quaternion group, and let $C$ denote the centre of $Q_8$, which has order $2$. Then set $E_1=Q_8$, $E_2=Q_8 \times_C Q_8$ amalgamated with respect to the diagonal copy of $C$, and $E_n=E_{n-1} \times_C Q_8$ for $n \geq 3$.
    For each $n$, the $2$-group $E_n$ is extraspecial, meaning that the centre $C$ of $E_n$ has order two and the quotient $E_n/C$ is an elementary abelian $2$-group. Let $c \in C$ denote the non-identity element of the centre. If  $a \notin C$, then either $a^2=c$, or we can conjugate $a$ with an element in the first copy of $Q_8$ and multiply with $a$ to obtain $c$. This shows that all non-trivial normal subgroups $\{1\} \neq N \lhd E_n$ contain $C$. Therefore there exist no epimorphisms from $E_n$ to $E_m$ for $n \neq m$, because any non-trivial quotient of $E_n$ is a quotient of the elementary abelian group $E_n/C$, and $E_m$ is non-abelian.

    If we let $\cat{U}$ denote the family formed by the $E_n$'s and the trivial group $1$, then $\cat{U} \setminus \{1\}$ is a groupoid, and the subcategory $\cat{U} \subset \cat{G}$ has the same structure as the category of cyclic groups of prime order and the trivial group of \cref{exa:cyclicprimes}. The resulting spectrum $\Spc(\D{\cat{U}}^{c})$ is homeomorphic to the one shown in \cref{exa:cyclicprimes}, and in this case, the group prime map $\mathfrakp_{-}\colon \pi_0\cat{U} \to \Spc(\D{\cat{U}}^c)$ is neither an embedding nor open.
\end{Rem}

As one application of \cref{prop:pgroup_embedding}, we obtain a sharper version of \cref{thm:spcfinite} when restricting attention to downwards closed families of finite $p$-groups. Without any such constraint, however, \Cref{exa:cyclicprimes} demonstrates that the conclusion of the next corollary fails. 

\begin{Cor}\label{cor:spcfinitep_characterization}
    For a downwards closed family of finite $p$-groups $\cat{U}$, the following conditions are equivalent:
        \begin{enumerate}
            \item $\mathfrakp_{-}\colon \pi_0\cat{U} \to \Spc(\D{\cat{U}}^{c})$ is a surjection;
            \item $\mathfrakp_{-}\colon \pi_0\cat{U} \to \Spc(\D{\cat{U}}^{c})$ is a homeomorphism;
            \item $\cat{U}$ is essentially finite.
        \end{enumerate}
\end{Cor}
\begin{proof}
    By \cref{prop:pgroup_embedding}, $\mathfrakp_{-}$ is an embedding. A surjective embedding is a homeomorphism, so $(a) \Rightarrow (b)$. Condition $(b)$ implies that $\Spc(\D{\cat{U}}^{c})$ is discrete. Since it is always spectral and thus in particular quasi-compact, it must be finite, which gives $(c)$. Finally, the implication $(c) \Rightarrow (a)$ is proven in \cref{thm:spcfinite}. 
\end{proof}

\begin{Rem}
    This result amply demonstrates the failure of the surjectivity criterion of \cite{BCHS2024} in the context of non-rigidly compactly generated tt-categories. Let $\cat{U}$ be a downwards closed family of finite $p$-groups. While the evaluation functors
        \[
            (\mathrm{ev}_{G}\colon \D{\cat{U}} \to \D{\{G\}})_{G \in \pi_0\cat{U}}
        \]
    are jointly conservative, the induced maps on spectra are not jointly surjective, unless $\cat{U}$ is essentially finite. 
\end{Rem}

We conclude this section with a computation of the Balmer spectrum for global representations indexed on the multiplicative global family of elementary abelian $p$-groups $\cat E_p$, for some prime number $p$, and explain the interpretation in terms of derived VI-modules.

\begin{Thm}\label{thm-elementary}\leavevmode
The category $\sD(\cat E_p)^c$ is standard (in the sense of \cref{def:standardtt}) and its spectrum is given as follows:
\begin{enumerate}
    \item The only prime ideals in $\sD(\cat E_p)^c$ are the group primes $\mathfrakp_{(\bbZ/p)^{\times n}}$ for $n \in \bbN$ and the zero ideal $(0)$ which we denote by $\mathfrakp_\infty$. Therefore we may identify the underlying set of the spectrum with $\pi_0\cat E_p \cup \{\mathfrakp_\infty\}$.
    \item A subset $U$ is open in the Balmer topology on $\Spc(\D{\cat{E}}^c)$ if and only if it is contained in $\pi_0\cat E_p$, or is equal to the whole set. Moreover, $U$ is a quasi-compact open if and only if it is a finite subset of $\pi_0\cat E_p$ or equal to the whole set.
 \end{enumerate}
 In other words, the Balmer spectrum of $\D{\cat{E}}^c$ can be depicted in the following way:
	
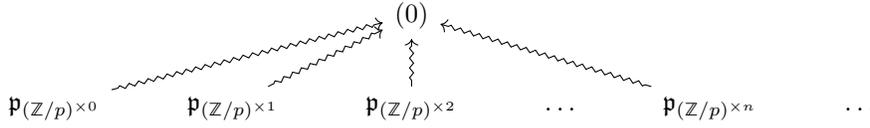
\[
	\begin{tikzcd}
         &            &  (0)           &        &               &       \\
	\mathfrakp_{(\bbZ/p)^{\times 0}} \arrow[urr, rightsquigarrow] & \mathfrakp_{(\bbZ/p)^{\times 1}} \arrow[ur, rightsquigarrow]  & \mathfrakp_{(\bbZ/p)^{\times 2}} 
	\arrow[u, rightsquigarrow] & \ldots & \mathfrakp_{(\bbZ/p)^{\times n}} \arrow[ull, rightsquigarrow]& \ldots 
	\end{tikzcd}
	\]
    Thus, $\Spc(\D{\cat{E}}^c)$ is homeomorphic to the Hochster dual of $\Spec(\bbZ)$.
\end{Thm}
\begin{proof}
    Throughout this proof we will adopt the following notation: Any group in $\cat E_p$ is isomorphic to $(\bbZ/p)^{\times n}$ for some $n$. In what follows we will abbreviate $(\bbZ/p)^{\times n}$ to $n$, and write $e_{n,V}$ instead of $e_{(\bbZ/p)^{\times n}, V}$. Likewise, we denote the corresponding group prime by $\mathfrakp_n$.

    We begin by establishing two auxiliary properties of $\sD(\cat E_p)^c$, which both rely on \cref{thm:nontorsion}:
        \begin{enumerate}
            \item[(1)] $\hsupp(X) \subseteq \hsupp(Y) \implies \thickt{X} \subseteq \thickt{Y}$ for all $X,Y \in \sD(\cat E_p)^c$;
            \item[(2)]
            $\hsupp(X) \subseteq \pi_0\cat E_p$ is cofinite for all $X \in \sD(\cat E_p)^c$ nontrivial.
        \end{enumerate}
    Note that, in particular, (1) shows that $\D{\cat{E}}^c$ is standard: $\hsupp(X) = \hsupp(X \otimes X)$ for any $X \in \D{\cat{E}}^c$, so $X \in \thickt{X \otimes X}$ by (1). This means we can apply Balmer's criterion \cite[Proposition 4.4]{Balmer2005}.
        
    Let us prove Statement (1). When $n\geq m$, \cite[Remark 4.12]{PolStrickland2022} shows that $e_n$ is a retract of $e_n\otimes e_m$, so $e_n\in\thickt{e_m}$. \Cref{thm:nontorsion} tells us that there exists $n_0$ with $e_{n_0}\in\thickt{Y}$ and so $e_n\in\thickt{Y}$ for all $n\geq n_0$. From this it is clear that when $n \geq n_0$, we have $\mathrm{thick}_{\otimes,n}^{+}\langle Y\rangle = \thickt{Y}$. Therefore $\mathrm{thick}_{\otimes}^{+}\langle Y \rangle = \thickt{Y}$, so we conclude by \cref{prop:hsupptt+}. To see Statement (2), \Cref{thm:nontorsion} again shows that there exists $n_0$ with $e_{n_0}\in\thickt{X}$, so $\hsupp(e_{n_0}) \subseteq \hsupp(X)$. The former subset of $\pi_0\cat E_p$ is cofinite, hence so is the latter.

    Now consider some nonzero prime ideal $\mathfrakq \in \Spc(\sD(\cat E_p)^c)$. We claim that there exists $s \in \bbN$ such that $\mathfrakq = \mathfrakp_s$. Since the zero ideal is prime by \cref{prop:zeroideal}, this determines the underlying set of the spectrum. In order to establish the claim, take some nonzero $X \in \mathfrakq$. By (2), its homological support is cofinite, say $\hsupp(X) = \pi_0\cat E_p \setminus S$ for some finite $S$. Note that $S$ must be non-empty, for otherwise $\hsupp(X) = \hsupp(\unit)$ so that $\unit \in \mathfrakq$ by (1), a contradiction. 

    Put $m_{s}\coloneqq \cof(e_{s, k} \to \unit)$. One calculates that $\hsupp(m_{s})=\pi_0 \cat E_p \setminus \{s\}$. Therefore, 
        \[
            \hsupp(X) = \bigcap_{s\in S}\hsupp(m_s) = \hsupp(\bigotimes_{s \in S}m_s).
        \]
    By (1), this gives $\bigotimes_{s \in S}m_s \in \thickt{X} \subseteq \mathfrakq$, so there exists $s \in S$ with $m_s \in \mathfrakq$ because $\mathfrakq$ is prime. Any $Y \in \mathfrakp_{s}$ satisfies $\hsupp(Y) \subseteq \pi_0 \cat E_p \setminus \{s\} = \hsupp(m_s)$, which implies $Y \in \thickt{m_s}$, again by (1). Combining these statements, we obtain containments
        \[
            \mathfrakp_{s} \subseteq \thickt{m_s} \subseteq \mathfrakq.
        \]
    But group primes are maximal primes by \cref{cor:groupprimes_generic}, so that these containments are in fact equalities, thereby establishing the claim. 

    It remains to determine the topology on $\Spc(\sD(\cat E_p)^c)$. Note that the group primes are isolated (\Cref{prop:pgroup_embedding}) but not closed (as each of them contains $\mathfrakp_\infty$), while the zero prime ideal $\mathfrakp_\infty$ is closed but not isolated (as otherwise all points would be isolated, contradicting quasi-compactness). Now consider an open subset $U$ in $\Spc(\sD(\cat E_p)^c)$. On the one hand, if $U$ does not contain $\mathfrakp_\infty$, then there are no constraints on $U$. On the other hand, if $U$ contains $\mathfrakp_\infty$, then its closed complement must be empty, because every closed subset contains $\mathfrakp_\infty$; hence $U = \Spc(\sD(\cat E_p)^c)$. This also determines the quasi-compact open subsets, thus finishing the proof. 
\end{proof}

\begin{Rem}\label{rem:fingenideals}
    In the course of the proof of \cref{thm-elementary}, we have seen that every prime ideal in $\sD(\cat E_p)^c$ is generated by a single object; in fact, this holds for all thick ideals. An explanation for this phenomenon is given in \cite{ttCohen_pre}. Indeed, for an essentially small tt-category $\cat K$, the following conditions are equivalent:
        \begin{enumerate}
            \item every prime ideal in $\cat K$ is generated by a single object;
            \item every radical ideal in $\cat K$ is generated by a single object;
            \item the Hochster dual of $\Spc(\cat K)$ is Noetherian.
        \end{enumerate}
\end{Rem}

\begin{Rec}\label{rec:VImod}
    Let $\VI$ denote the category of finite-dimensional $\bbF_p$-vector spaces and injective homomorphisms. The symmetric monoidal abelian category of VI-modules is defined as 
        \[
            \VIMod \coloneqq \Fun(\VI,\Mod{k}),
        \]
    equipped with pointwise additive and tensor structure. This category features prominently in representation stability, see for example \cite{PutmanSam2017}. Pontryagin duality furnishes an equivalence of categories $\VI \simeq \cat{E}^{\op}$ and hence a tt-equivalence
        \[
            \sfD(\VIMod) \simeq \sfD(\cat{E}). 
        \]
    Keeping in mind \cref{rem:fingenideals}, we can thus translate \cref{thm-elementary} into a classification of derived VI-modules:
\end{Rec}

\begin{Cor}\label{cor:VImod}
        Homological support induces a bijection
        \[
            \begin{tikzcd}[ampersand replacement=\&,column sep=small]
                {\begin{Bmatrix}
                    \text{principal ideals } \langle M\rangle_{\otimes}\colon \\
                    M \in \sfD(\VIMod)^c
                \end{Bmatrix}}
                    \& 
                {\begin{Bmatrix}
                    \text{cofinite subsets }\\
                    \text{of } \bbN, \text{ or } \emptyset
                \end{Bmatrix}}.
                    \arrow["\sim", from=1-1, to=1-2]
            \end{tikzcd}
        \]
\end{Cor}
\begin{proof}
    The universal support function \eqref{eq:ttclassification} provides a bijection between principal thick ideals of $\sfD(\VIMod)^c$ and the complements of quasi-compact open subsets of $\Spc(\sfD(\VIMod)^c)$. By \cref{thm-elementary}, the latter biject with the collection of cofinite subsets of $\pi_0\cat E_p \cup \{\mathfrakp_{\infty}\}$ containing $\mathfrakp_{\infty}$, together with the empty set. Upon identifying $\pi_0\cat E_p$ with $\bbN$ and unwinding the construction of the homological support function via \eqref{eq:groupprimemap}, we obtain the result. 
\end{proof}

Unwinding the construction of the support function, we recover \cref{thmx:VImod}.

\newpage
\part{Reflective filtrations}
In this part, we study filtrations of collections of finite groups by reflective subcategories. By combining this with a result of Gallauer~\cite{Gallauer} concerning the continuity of Balmer spectra, we show that if $\cat{U}$ admits a nice filtration, then the spectrum of $\D{\cat{U}}^c$ is homeomorphic to the inverse limit over the spectra of the constituent filtration pieces. We remind the reader that \cref{hyp} is in place throughout.

\section{Reflective subcategories of groups}
\label{sec-refl}

In this section, we prove various facts about inclusions of subcategories of groups where the inclusion admits a left adjoint; recall that such subcategories are called \emph{reflective}. For example,
abelianisation is left adjoint to the inclusion of abelian groups in
all groups. These results will be used to construct well behaved filtrations of collections $\cat{U}$, so that we may make arguments by reduction to the filtration pieces. To put this into practice, it will be helpful to broaden our scope and consider finitely generated groups.

\begin{Def}\label{def-fg}
 We write $\tCFG$ for the category of finitely generated groups and
 surjective homomorphisms, and $\CFG$ for the quotient category in
 which we identify conjugate homomorphisms.  Given any full subcategory
 $\cat{S}\subseteq\CFG$, we write $\widetilde{\cat{S}}$ for the corresponding subcategory of
 $\tCFG$.  In particular, $\cat{G}$ is the category of finite groups and conjugacy classes of epimorphisms and $\widetilde{\cat{G}}$ is the category of finite groups and
 surjective homomorphisms. 
\end{Def}

\begin{Rem}
Note that in this section we work in greater generality than \cref{hyp} and generally consider full subcategories of $\CFG$ rather than of $\cat{G}$. To avoid confusion, we will therefore write $\cat{S}$ and $\cat{T}$ for such full subcategories. 
\end{Rem}

\begin{Lem}\label{lem-Hom-finite}
    If $G\in\CFG$ and $H\in\cat{G}$ then the set of group homomorphisms
    \[
    \mathrm{Hom}(G,H)
    \]
    is finite. Furthermore, the sets 
    \[
    \Hom_{\tCFG}(G,H)\;\; and \;\;\Hom_{\CFG}(G,H)
    \]
    are also finite.
\end{Lem}
\begin{proof}
    Choose a finite subset $X\subseteq G$ that generates $G$, and let $\mathrm{Map}(X,H)$ be the set of all functions from $X$ to $H$, which is clearly finite.  As $X$ generates $G$ we see that the restriction $\mathrm{Hom}(G,H)\to\mathrm{Map}(X,H)$ is injective, so $\mathrm{Hom}(G,H)$ is finite.  Moreover, $\Hom_{\tCFG}(G,H)$ is a subset of $\mathrm{Hom}(G,H)$, and $\Hom_{\CFG}(G,H)$ is a quotient of $\Hom_{\tCFG}(G,H)$, so both of these are also finite.
\end{proof}

We begin by giving some equivalent conditions for a subcategory of groups to be reflective.
\begin{Prop}\label{prop-refl}
 Suppose we have full subcategories $\cat{S}\subseteq\cat{T}\subseteq\CFG$.  Then the
 following are equivalent:
 \begin{enumerate}
  \item $\widetilde{\cat{S}}$ is reflective in $\widetilde{\cat{T}}$, i.e., the inclusion
   $\widetilde{\cat{S}}\to\widetilde{\cat{T}}$ has a left adjoint;
  \item $\cat{S}$ is reflective in $\cat{T}$;
  \item for each $G\in\cat{T}$ there is a smallest normal subgroup
   $N(G)$ in $G$ such that $G/N(G)\in\cat{S}$: if there is $N\lhd G$ such that $G/N \in \cat{S}$, then $N(G)\leq N$. 
 \end{enumerate}
 Moreover, if these equivalent conditions hold then all of the following also hold.
 \begin{enumerate}
  \item[(d)] 
  Let $q(G) \coloneqq G/N(G)$. For any $\alpha\in\Hom_{\widetilde{\cat{T}}}(G,H)$ we have $\alpha(N(G))\leq N(H)$,
   so there is an induced homomorphism $q(\alpha)\colon q(G)\to q(H)$, and $q$ gives a functor $\widetilde{\cat{T}}\to\widetilde{\cat{S}}$.
  \item[(e)] If $\alpha$ and $\beta$ are conjugate, then so are $q(\alpha)$
   and $q(\beta)$, so $q$ also gives a functor $\cat{T}\to\cat{S}$.
  \item[(f)] The functors $q\colon \widetilde{\cat{T}}\to\widetilde{\cat{S}}$ and $q\colon \cat{T}\to\cat{S}$ are left
   adjoint to the corresponding inclusions.
  \item[(g)] $\cat{S}=\{G\in\cat{T}\mid N(G)=1\}$.
 \end{enumerate}
\end{Prop}
\begin{proof}
 First suppose we have a reflection $q\colon \widetilde{\cat{T}}\to\widetilde{\cat{S}}$, so for each
 $G\in\widetilde{\cat{T}}$ we have a unit map $\eta_G\colon G\to q(G)$.  This is a morphism
 in $\widetilde{\cat{T}}$, so it is a surjective group homomorphism.  We write $N(G)$
 for the kernel, so $\eta_G$ induces an isomorphism $G/N(G)\to q(G)$,
 so $G/N(G)\in\cat{S}$.  Now suppose we have another normal subgroup
 $M\lhd G$ with $G/M\in\cat{S}$.  The adjointness property says that the
 projection $G\to G/M$ must factor uniquely through $\eta_G$, so
 $N(G)\leq M$.  Thus $N(G)$ is the smallest normal subgroup with
 quotient in $\cat{S}$.  This shows that~(a) implies~(c).

 Suppose instead that we have a reflection $q\colon \cat{T}\to\cat{S}$.  The unit map
 $\eta_G\colon G\to q(G)$ is now just a conjugacy class of surjective
 homomorphisms, but conjugate homomorphisms have the same kernel, so
 we still have a well-defined kernel $N(G)\lhd G$.  By the same line
 of argument as in the previous paragraph, we deduce that~(b)
 implies~(c).  

 Now suppose that~(c) holds.  For $\alpha\in\Hom_{\widetilde{\cat{T}}}(G,H)$ we see that $\alpha$
 induces an isomorphism $G/\alpha^{-1}(N(H))\to H/N(H)\in\widetilde{\cat{S}}$, so we
 must have $N(G)\leq\alpha^{-1}(N(H))$, or equivalently
 $\alpha(N(G))\leq N(H)$.  This means that~(d) holds, so we can define
 $q(G)=G/N(G)$ and let $\eta_G\colon G\to q(G)$ be the projection, and
 define $q(\alpha)\colon q(G)\to q(H)$ to be the unique homomorphism with
 $q(\alpha)\circ\eta_G=\eta_H\circ\alpha$.  This clearly makes $q$ into a
 functor $\widetilde{\cat{T}}\to\widetilde{\cat{S}}$.  Now consider an inner automorphism
 $\gamma_h\in\Hom_{\widetilde{\cat{T}}}(H,H)$ given by $\gamma_h(x)=hxh^{-1}$.  The map
 $\gamma_{\eta_H(h)}$ then has the defining property of $q(\gamma_h)$.
 Claim~(e) follows from this.  

 Now suppose $G\in\widetilde{\cat{T}}$, $H\in\widetilde{\cat{S}}$, and $\alpha\in\Hom_{\widetilde{\cat{T}}}(G,H)$.  This
 means that $G/\ker(\alpha)\in\widetilde{\cat{S}}$, so $N(G)\leq\ker(\alpha)$, so $\alpha$
 factors uniquely through $\eta_G$.  This shows that $q\colon \widetilde{\cat{T}}\to\widetilde{\cat{S}}$
 is left adjoint to the inclusion $\widetilde{\cat{S}}\to\widetilde{\cat{T}}$, and essentially the
 same argument also shows that $q\colon \cat{T}\to\cat{S}$ is left adjoint to the
 inclusion $\cat{S}\to\cat{T}$, so~(a),~(b), and~(f) hold.
 
 Finally, it is clear from the definition of $N(G)$ that $N(G)=1$ if and only if
 $G\in\cat{S}$, which is claim~(g).
\end{proof}

\begin{Not}\label{not-reflections}
    Consider full subcategories $\cat{S}\subseteq\cat{T}\subseteq\CFG$. If $\cat{S}$ is reflective in $\cat{T}$, we will denote the reflection by $q\colon \cat{T} \to \cat{S}$ and note that for all $G \in \cat T$ we have $q(G) = G/N(G)$ for a normal subgroup $N(G)\lhd G$ as in \cref{prop-refl}. Throughout we will often suppress notation for the inclusion and view reflections as endofunctors. 
\end{Not}
\begin{Rem}\label{rem-nested-reflection}
    Suppose that $\cat{S}_0 \subseteq \cat{S}_1 \subseteq \cat{T}$. If $\cat{S}_0$ is reflective in $\cat{T}$ with reflection $q\colon \cat{T} \to \cat{S}_0$, then $\cat{S}_0$ is also reflective in $\cat{S}_1$ with reflection given by $q\arrowvert_{\cat{S}_1}\colon \cat{S}_1 \to \cat{S}_0$. By a slight abuse of notation, we will often write $q$ for both reflections.
\end{Rem}

\begin{Exa}\label{eg-refl-ker}\leavevmode
 \begin{enumerate}
  \item The category of finite abelian groups is reflective in $\cat{G}$, with $N(G)=[G,G]$.
  \item The category of abelian groups of exponent dividing $n$ 
   is reflective in the category of all finite abelian groups with 
   $N(G)=\{g^n\mid g\in G\}$.
  \item The category of elementary abelian $p$-groups is reflective in $\cat{G}$, with $N(G)=G^p[G,G]$ the $p$-Frattini subgroup.
  \item Let $\cat{S}$ be the category of finite groups that can be
   expressed as a product of simple groups.  Then one can verify that $\cat{S}$ is
   reflective in $\cat{G}$, with $N(G)$ being the intersection of all
   maximal proper normal subgroups of $G$. As this takes some work to verify, we have decided not to include the details. 
 \end{enumerate}
\end{Exa}
The next proposition gives a first nontrivial example of a reflective subcategory.
\begin{Prop}\label{prop-refl-subprod}
 Let $\cat{S}_0\subseteq\cat{G}$ be an essentially finite subcategory, and let $\cat{S}$ be the category of groups that
 can be embedded in a finite product of groups taken from $\cat{S}_0$.
 Then $\cat{S}$ is reflective in $\CFG$.
\end{Prop}
\begin{proof}
 Choose a group from every isomorphism class in $\cat{S}_0$, and let $P$
 be the product of those groups.  Then $P$ is finite and $G\in\cat{S}$
 if and only if $G$ can be embedded in $P^r$ for some $r$.  If $G\in\CFG$ then
 $\Hom(G,P)$ is finite by \cref{lem-Hom-finite}.  If
 $\Hom(G,P)=\{\alpha_i\mid i<r\}$ then we can combine the maps $\alpha_i$
 into a single homomorphism $\alpha\colon G\to P^r$ and define
 $N(G)=\ker(\alpha)$, so $G/N(G)\cong\img(\alpha)\leq P^r$ so
 $G/N(G)\in\cat{S}$.  Conversely, if $M\lhd G$ and $G/M\in\cat{S}$ then $M$
 can be realized as the kernel of some morphism $\beta\colon G\to P^s$, and
 each component of $\beta$ must be one of the maps $\alpha_i$, so
 $N(G)\leq M$.  Thus, $N(G)$ is the smallest normal subgroup of $G$
 with quotient in $\cat{S}$, so everything follows from
 \cref{prop-refl}.  
\end{proof}
We next identify some general types of families which give rise to reflective subcategories.

\begin{Def}\label{defn-submult}
 We say that a subcategory $\cat{U}\subseteq\cat{G}$ is \emph{submultiplicative}, or
 \emph{$\infty$-submultiplicative}, if $1\in\cat{U}$ and whenever $G\leq G_0\times G_1$
 with $G_0,G_1\in\cat{U}$, we also have $G\in\cat{U}$.
\end{Def}

\begin{Prop}\label{prop-submult-refl}
    If $\cat{U}\subseteq\cat{G}$ is submultiplicative, then it is reflective in $\cat{G}$.
\end{Prop}
\begin{proof}
  Given $G\in\cat{G}$ (so $G$ is finite) we let $N_0,\dotsc,N_{r-1}$ be the list of all normal subgroups such that $G/N_i\in\cat{U}$, and put $N(G)=\bigcap_iN_i$.  Then the natural map $G/N(G)\to\prod_iG/N_i$ is injective, and the image lies in $\cat{U}$ by the submultiplicativity assumption, so $G/N(G)\in\cat{U}$.  This means that $N(G)$ is one of the groups $N_i$, and it is clearly the smallest one.  Thus, condition~(c) in \cref{prop-refl} is satisfied, so $\cat{U}$ is reflective in $\cat{G}$.
\end{proof}

\begin{Exa}\label{ex-submult}
    Note that multiplicative global families are submultiplicative.  Thus, the following subcategories are all submultiplicative, and therefore reflective in $\cat{G}$:
    \begin{enumerate}
        \item the whole category $\cat{G}$;
        \item the subcategory of abelian groups;
        \item the subcategory of solvable groups;
        \item the subcategory of nilpotent groups of class (i.e., the length of the minimal central series of the group) at most $c$ (for a fixed $c$);
        \item the subcategory of $p$-groups (for a fixed prime number $p$);
        \item the subcategory of $G$ for which all simple composition factors of $G$ have order at most $m$ (for a fixed $m$).
    \end{enumerate}
    We also note that there are examples of submultiplicative families which are not multiplicative global families, for example the family $\cat W_3$ from \cite[Definition 3.3]{PolStrickland2022}.
\end{Exa}

\begin{Rem}\label{rem-refl-res}
 Suppose we have a reflective subcategory $\cat{S}\subseteq\cat{T}$.  Let
 $\cat{T}'\subseteq\cat{T}$ be a full subcategory that is closed downwards in $\cat{T}$, so
 whenever we have $\alpha\in\Hom_{\cat{T}}(G,H)$ with $G\in\cat{T}'$ we also have
 $H\in\cat{T}'$.  Put $\cat{S}'=\cat{S}\cap\cat{T}'$.  It is then easy to see that
 $\cat{S}'$ is reflective in $\cat{T}'$, with the reflection $q'\colon \cat{T}'\to\cat{S}'$
 just being the restriction of $q\colon \cat{T}\to\cat{S}$.  
\end{Rem}

We next give a particular example of \cref{rem-refl-res} which will be important for our main application regarding groups of bounded $p$-rank.
\begin{Def}\label{defn-r-generated}
 For any natural number $r$, we let $F_r$ denote the free group on $r$ generators. We say that a finite group $G$ is
 \emph{$r$-generated} if there exists a surjective homomorphism
 $F_r\to G$.
 We also say that every finite group is $\infty$-generated.  For any
 $r\in\bbN\cup\{\infty\}$ we write $\cat{G}\ip{r}$ for the category of
 $r$-generated finite groups, so $\cat{G}\ip{\infty}=\cat{G}$.  For an arbitrary full
 subcategory $\cat{S}\subseteq\cat{G}$ we put $\cat{S}\ip{r}=\cat{S}\cap\cat{G}\ip{r}$.
\end{Def}

\begin{Prop}\label{prop-fg-finite}
 Suppose that $\cat{S}\subseteq\cat{G}$, $\cat{S}$ is reflective in $\CFG$ and that $r<\infty$. Then
 $\cat{S}\ip{r}$ is essentially finite and is reflective in $\CFG\ip{r}$.
\end{Prop}
\begin{proof}
 As $\CFG\ip{r}$ is closed downwards in $\CFG$,
 \cref{rem-refl-res} shows that the subcategory
 $\cat{S}\ip{r}=\cat{S}\cap\CFG\ip{r}$ is reflective in $\CFG\ip{r}$.
 If $G\in\cat{S}\ip{r}$ then we have a morphism $F_r\to G$ in $\CFG$,
 which must factor through a morphism $q(F_r)\to G$ by adjunction.  Here
 $q(F_r)\in\cat{S}\subseteq\cat{G}$, so $q(F_r)$ is finite, so there are only
 finitely many possible quotients of $q(F_r)$.  It follows that there
 are only finitely many possibilities for the isomorphism class of
 $G$. 
\end{proof}

We finish this section by verifying that \cref{hyp} and \cref{hyp2} are inherited by reflective subcategories.

\begin{Lem}\label{lem-refl-wide}
 Let $\cat{U}, \cat{V}$ be subcategories of $\cat{G}$ with $\cat{U}$ reflective in $\cat{V}$. If $\cat{V}$ is widely closed then $\cat{U}$
 is also widely closed. Moreover, if $\unit \in \D{\cat{V}}$ is compact, then $\unit$ is also compact in $\D{\cat{U}}$. 
\end{Lem}
\begin{proof}
 Recall from \cref{prop-refl} that for all $G \in \cat V$ there exists a smallest normal subgroup $N(G)\lhd G$ such that $G/N(G) \in \cat{U}$. With this in mind, suppose we have morphisms $G\xleftarrow{\alpha}H\xrightarrow{\beta}K$ in $\cat{V}$, and we
 let $L$ be the image of the combined map $\delta\colon H\to G\times K$, so
 $L\in\cat{V}$.  Suppose that $G,K\in\cat{U}$, so $N(G)=N(K)=1$ by \cref{prop-refl}(g). We have
 surjective homomorphisms $G\xleftarrow{\pi_G}L\xrightarrow{\pi_K}K$, so
 $\pi_G(N(L))\leq N(G)=1$ and $\pi_K(N(L))\leq N(K)=1$ by \cref{prop-refl}(d).  From this it
 is clear that $N(L)=1$, so $L\in\cat{U}$ as required.

Write $q\colon \cat{V} \to \cat{U}$ for the reflection and $i$ for the inclusion. Since $q^* \simeq i_!$ is fully faithful (\cref{cons-restriction-derived}) and symmetric monoidal (\cref{prop-gen-fun-preserve-compacts}), we have 
 \[\Hom_{\D{\cat{U}}}(\unit,-) \simeq \Hom_{\D{\cat{V}}}(q^*\unit, q^*(-)) \simeq \Hom_{\D{\cat{V}}}(\unit, q^*(-)).\] Since $q^*$ moreover preserves coproducts, we deduce that if $\unit \in \D{\cat{V}}$ is compact, then so is $\unit \in \D{\cat{U}}$.
\end{proof}

\section{Reflective filtrations}
\label{sec-refl-filt}
In this section we introduce the notion of a reflective filtration and give various important examples. We begin by making the key definition for this section. 
\begin{Def}\label{defn-refl-filt}
 A \emph{reflective filtration} of $\cat{U}$ is a sequence of reflective
 subcategories $\cat{U}[n]\subseteq\cat{U}$ with $\cat{U}[n]\subseteq\cat{U}[n+1]$ and
 $\cat{U}=\bigcup_n\cat{U}[n]$.  Given such a filtration, we will always write $i_n\colon \cat{U}[n] \to \cat{U}$ for the inclusion, and
 $q_n$ for the reflection $\cat{U}\to\cat{U}[n]$ which will necessarily have the form
 $q_n(G)=G/N_n(G)$ for some normal subgroup $N_n(G)\lhd G$ by \cref{prop-refl}. We say that the filtration is
 \emph{essentially finite} if each $\cat{U}[n]$ is essentially finite.
\end{Def}

\begin{Rem}\label{rem-hyps-refl-fil}
    Suppose that $\cat U$ admits a reflective filtration $\{\cat U[n]\}_{n\geq 0}$, and that $\cat U$ satisfies \cref{hyp} (resp. \cref{hyp2}). Then for all $n\geq 0$, the subcategory $\cat U[n]$ satisfies \cref{hyp} (resp. \cref{hyp2}), see \cref{lem-refl-wide}.
\end{Rem}

\begin{Rem}\label{rem-res-filt}
 If $\{\cat{V}[n]\}_{n\geq 0}$ is a reflective filtration of $\cat{V}$, and
 $\cat{U}$ is a subcategory that is closed downwards in $\cat{V}$, then we
 find that the subcategories $\cat{U}[n]=\cat{U}\cap\cat{V}[n]$ give a reflective
 filtration of $\cat{U}$, see \cref{rem-refl-res}.  Similarly, if the original filtration is
 essentially finite, then so is the restricted filtration.
\end{Rem}

We now give a variety of examples of reflective filtrations.
\begin{Exa}\label{eg-ab-filt}
 Let $\cat{U} = \cat{A}$ be the category of finite abelian groups.  For $G\in\cat{U}$
 put $N_n(G)=\{g^{n!}\mid g\in G\}$ and $\cat{U}[n]=\{G\mid N_n(G)=1\}$ and 
 $q_n(G)=G/N_n(G)$.  This gives a
 reflective filtration, which becomes essentially finite when
 restricted to $\cat{U}\ip{r}$. 
\end{Exa}

\begin{Exa}\label{eg-p-filt}
 Let $\cat{U}=\cat{G}_p$ be the category of finite $p$-groups for some prime number $p$.  For $G\in\cat{U}$ let
 $\Phi(G)=G^p[G,G]$ be the Frattini subgroup.  Put $\cat{U}[n]=\{G\mid\Phi^n(G)=1\}$ and
 $q_n(G)=G/\Phi^nG$.  Standard facts about $p$-groups show that this
 gives a reflective filtration, which becomes essentially finite
 when restricted to $\cat{U}\ip{r}$.
\end{Exa}

We next modify \cref{defn-submult} slightly to introduce a key 
class of examples for this paper, the $r$-submultiplicative families.
We will construct a reflective filtration for each such family. Recall 
the definition of $r$-generated group from \cref{defn-r-generated}.

\begin{Def}\label{defn-r-submult}
    Let $r$ be a natural number. We say that $\cat{U}$ is \emph{$r$-submultiplicative} if it satisfies the following conditions:
    \begin{enumerate}
        \item all groups in $\cat{U}$ are $r$-generated;
        \item $1\in\cat{U}$;
        \item if $G \leq G_0 \times G_1$ with $G_0, G_1 \in \cat{U}$ and $G$ is $r$-generated, then $G \in \cat{U}$.
    \end{enumerate}
\end{Def}

\begin{Lem}\label{lem-submult-closure}
 If $\cat{U}$ is submultiplicative, then $\cat{U}\ip{r}$ is 
 $r$-submultiplicative for all $r$.
 Conversely, suppose that $\cat{U}$ is $r$-submultiplicative, and let $\cat{U}^*$ be
 the category of groups that can be embedded in a finite product $\prod_{i}G_i$ for some groups $G_i\in\cat{U}$. Then $\cat{U}^*$ is
 submultiplicative and $\cat{U}=\cat{U}^*\ip{r}$.
\end{Lem}
\begin{proof}
This is straightforward from the definitions.
\end{proof}

\begin{Exa}\label{ex-r-submultiplicative}
    We can produce $r$-submultiplicative families by taking any of the categories $\cat{U}$ in \cref{ex-submult} and intersecting it with $\cat{G}\ip{r}$. For example, the category of $r$-generated solvable groups is $r$-submultiplicative.
\end{Exa}

We can produce reflections for $r$-submultiplicative families by a method very similar to that used in \cref{prop-submult-refl}.

\begin{Def}\label{defn-std-filt}
    For any $\cat{U}\subseteq\cat{G}$ we put   
        \[
            \cat{U}_{\leq n}\coloneqq\{G\in\cat{U}\mid |G|\leq n\}.
        \]
    For any $F \in \CFG$, we define
        \begin{align*}
            \cat{K}_n(F) & \coloneqq \{N\lhd F\mid F/N\in\cat{U}_{\leq n}\}; \\
            K_n(F) & \coloneqq \bigcap_{N\in\cat{K}_n(F)} N = 
            \ker\left(F \to \prod_{N\in\cat{K}_n(F)} F/N\right); \\
            q_n(F) & \coloneqq F/K_n(F); \\
            \cat{U}[n] & \coloneqq  \{G\in\cat{U} \mid K_n(G)=1\}.
        \end{align*}
\end{Def}

\begin{Rem}\label{rem-U[n]-general}
    By construction $G \in \cat U[n]$ if and only if $G$ can be embedded into a finite product of groups in $\cat U_{\leq n}$.
\end{Rem}

\begin{Exa}
   Consider $\cat U=\cat{A}$ the collection of finite abelian groups. Then $\cat{A}[n]$ consists of the products of groups of the form $C_{p^m}$ with $p^m \leqslant n$. Note that in particular this filtration is not the same as the one given in \cref{eg-ab-filt}.
\end{Exa}

\begin{Prop}\label{prop-std-filt}\leavevmode
 \begin{enumerate}
  \item If $\cat{U}$ is submultiplicative, then
   \cref{defn-std-filt} gives a reflective filtration of
   $\cat{U}$, in which $\cat{U}[n]$ is reflective in $\CFG$.  Moreover, for any
   $r\in\bbN$ this restricts to give an essentially finite reflective
   filtration of $\cat{U}\ip{r}$, with $\cat{U}\ip{r}[n]$ being reflective in
   $\CFG\ip{r}$.
  \item Similarly, if we assume that $\cat{U}$ is
   $r$-submultiplicative, then \cref{defn-std-filt} gives an
   essentially finite reflective filtration of $\cat{U}$, with $\cat{U}[n]$
   being reflective in 
   $\CFG\ip{r}$. 
 \end{enumerate}
\end{Prop}
\begin{proof}
 For claim~(a), \cref{prop-refl-subprod} applied to $\cat{S}_0=\cat{U}_{\leq n}$ shows that $\cat U[n]$ is reflective in $\CFG$, and then it follows from \cref{rem-U[n]-general} and \cref{prop-refl} that the above $q_n$ gives a left adjoint for the inclusion $\cat{U}[n]\to\CFG$. It then follows from \cref{rem-nested-reflection} that $\cat U[n]$ is also reflective in $\cat{U}$ with reflection given by the restriction of $q_n$.
 It is clear that $\cat{U}_{\leq n}\subseteq \cat{U}[n]\subseteq\cat{U}[n+1]$, and that
 $\bigcup_n\cat{U}_{\leq n}=\cat{U}$, so $\bigcup_n\cat{U}[n]=\cat{U}$.
 \cref{prop-fg-finite} shows that the filtration becomes
 essentially finite when restricted to $\cat{U}\ip{r}$.  This proves~(a),
 and~(b) follows similarly using \cref{lem-submult-closure}.
\end{proof}

\section{Continuity and reflective filtrations}
In this section we consider the case when $\cat{U}$ admits a reflective filtration and examine its consequences for the structure of $\D{\cat{U}}^c$. In particular, we show that $\D{\cat{U}}^c$ is then a filtered colimit of the filtration pieces $\D{\cat{U}[n]}^c$. If the filtration is moreover essentially finite, we show that this colimit description can be used to completely classify the finitely generated thick ideals in $\D{\cat{U}}^c$ in terms of certain subsets of $\pi_0\cat{U}$.
We begin this section with the following observation.

\begin{Rem}
Suppose that $\cat{U}$ has a reflective filtration $\{\cat{U}[n]\}_{n \geq 0}$. For natural numbers $n \leq m$, we write $i_n\colon \cat{U}[n] \to \cat{U}$ and $i_n^m\colon \cat{U}[n] \to \cat{U}[m]$ for the inclusions. There is a natural isomorphism of functors
\begin{equation}\label{eq:filtrationcommutes}
    (i_n^m)_! \circ i_n^* \xrightarrow{\sim} i_m^*.
\end{equation}
Indeed, there is a natural map $(i_n^m)_! \circ i_n^* \to i_m^*$ which is adjoint to the natural isomorphism $i_n^* \simeq (i_n^m)^*i_m^*$. Since all the functors involved are colimit preserving, to check this is an isomorphism it suffices to check on the compact generators $e_G$, where it is true by \cref{rem-restriction-gen}.
\end{Rem}

\begin{Prop}\label{prop-counit}
    Suppose that $\cat{U}$ admits a reflective filtration $\{\cat{U}[n]\}_{n \geq 0}$. Given $X \in \D{\cat{U}}^c$, there exists a natural number $n_0$ such that for all $n \geq n_0$, the counit map $(i_n)_!i_n^*X \to X$ is an isomorphism, and $i_n^*X$ is compact.
\end{Prop}
\begin{proof}
    We first argue that if the claims are true for some $n$, then they also hold for all $m \geq n$. Firstly, if $(i_n)_!i_n^*X \to X$ is an isomorphism, then
    \[
        (i_m)_!i_m^*X \simeq (i_m)_!(i_n^m)_!i_n^*X \simeq (i_n)_!i_n^*X \simeq X,
    \]
    by using \eqref{eq:filtrationcommutes}. A straightforward diagram chase identifies this with the counit map. By using \eqref{eq:filtrationcommutes} again, one also verifies that if $i_n^*X$ is compact then $i_m^*X$ is compact for all $m \geq n$ since $(i_n^m)_!$ preserves compacts by \cref{prop-gen-fun-preserve-compacts}(b).
    
    Observe that the full subcategory \[ \{X \in \D{\cat{U}}^c \mid \text{there exists $n$ such that $(i_n)_!i_n^*X \to X$ is an isomorphism}\}\] of $\D{\cat{U}}^c$ is thick, so it suffices to check the claim on the generators $X = e_G$ for $G \in \cat{U}$. In this case, we may choose $n$ large enough so that $G \in \cat{U}[n]$, and then it is immediate from \Cref{rem-restriction-gen} that the counit map is an isomorphism. So it remains to show that $i_n^*X$ is compact. The full subcategory
    \[\{X \in \D{\cat{U}}^c \mid \text{there exists $n$ such that $i_n^*X$ is compact}\}\] is thick, and hence it suffices to check for $X = e_G$ which is clear.
\end{proof}

\begin{Thm}\label{thm-colimit}
    Let $\cat U$ be a subcategory satisfying \cref{hyp2} and admitting a reflective filtration $\{\cat{U}[n]\}_{n \geq 0}$. Then the canonical functor 
        \[
            \colim_n \D{\cat{U}[n]}^c \to \D{\cat{U}}^c
        \] 
    induced by $(i_n)_! \cong q_n^*\colon \D{\cat{U}[n]}^c \to \D{\cat{U}}^c$ is a tt-equivalence.
\end{Thm}
\begin{proof}
    Filtered colimits of tt-categories are created by the forgetful functor to categories. Since $(i_n)_! \cong (q_n)^*$ is symmetric monoidal, it follows that the induced map \[i_!\colon\colim_n \D{\cat{U}[n]}^c \to \D{\cat{U}}^c\] is a tt-functor. Since each $(i_n)_!$ is fully faithful, it follows that $i_!$ is fully faithful (because mapping spaces in a filtered colimit of categories are the filtered colimits of the mapping spaces). Moreover, we see that $i_!$ is essentially surjective by \cref{prop-counit} and hence is an equivalence as claimed.
\end{proof}

\begin{Cor}\label{cor-lim-spectra}
 Let $\cat U$ be a subcategory satisfying \cref{hyp2} and admitting a reflective filtration $\{\cat{U}[n]\}_{n \geq 0}$. The equivalence of \cref{thm-colimit} induces a homeomorphism
    \[
        \Spc(\D{\cat{U}}^c)\xrightarrow{\sim}\lim_n \Spc(\sD(\cat U[n])^c).
    \]
\end{Cor}
\begin{proof}
    This follows by combining \cref{thm-colimit} with the fact that the Balmer spectrum sends filtered colimits of tt-categories to filtered limits of topological spaces~\cite[Proposition 8.2]{Gallauer}.
\end{proof}

\begin{Cor}\label{cor-ess-finite-invlim}
    Let $\cat U$ be a subcategory satisfying \cref{hyp2} and admitting an essentially finite reflective filtration $\{\cat{U}[n]\}_{n \geq 0}$. Then there is a homeomorphism 
    \[
        \Spc(\D{\cat{U}}^c) \xrightarrow{\sim} \lim_n\pi_0\cat{U}[n],
    \]
    where the limit is taken over the reflections. The spectrum $\Spc(\D{\cat{U}}^c)$ is profinite and, in particular,  for $\mathfrakp, \mathfrak{q} \in \Spc(\D{\cat{U}}^c)$, we have $\mathfrakp \subseteq \mathfrak{q}$ if and only if $\mathfrakp = \mathfrak{q}$. 
\end{Cor}
\begin{proof}
    The first claim follows from \cref{thm:spcfinite} and \cref{cor-lim-spectra}. This means that $\Spc(\D{\cat{U}}^c)$ is a profinite space, thus Hausdorff, and we conclude from \cite[Proposition 2.9]{Balmer2005} that there are no containments between distinct primes.
\end{proof}

\begin{Rem}
    In contrast to the setting of the previous result, the Balmer spectrum of $\D{\cat{U}}^c$ is not always zero-dimensional. For example, in \cref{cor:infinitedimensional} we showed that the Balmer spectrum for the family of finite abelian $p$-groups is infinite dimensional.
\end{Rem}

\begin{Rem}\label{rem-ess-finite-refl-standard}
    In the context of \cref{cor-ess-finite-invlim}, the tt-category $\D{\cat{U}}^c$ is standard by \cref{prop:standardtt} and \cref{thm:spcfinite}, so that the computation of the Balmer spectrum yields a classification of all thick ideals.
\end{Rem}

\begin{Rem}\label{rem-invlimit-shortcomings}
    Although the previous result gives a description of $\Spc(\D{\cat{U}}^c)$ as a profinite space, this description is somewhat inexplicit, since the inverse limit may be difficult to understand without a more concrete description. In the next part we will tackle this shortcoming, by providing additional conditions on the filtration which ensure that the inverse limit may be described in simple terms.
\end{Rem}

Nonetheless we are able to classify all thick ideals which are finitely generated in terms of certain subsets of $\pi_0\cat{U}$.

\begin{Prop}\label{prop-thick-building}
    Let $\cat U$ be a subcategory satisfying \cref{hyp2} and admitting an essentially finite reflective filtration $\{\cat{U}[n]\}_{n \geq 0}$, and let $X, Y \in \D{\cat{U}}^c$. Then $X \in \thickt{Y}$ if and only if $\hsupp(X) \subseteq \hsupp(Y)$.
\end{Prop}
\begin{proof}
    The forward implication is contained in \cref{lem:hsupp}, so we prove the reverse direction. Given $X, Y \in \D{\cat{U}}^c$, by \cref{thm-colimit}, there is an $n \geq 0$ so that $X \simeq q_n^*X_n$ and $Y \simeq q_n^*Y_n$ for some $X_n, Y_n \in \D{\cat{U}[n]}^c$. Write $\varphi_n = \Spc(q_n^*)$ for the induced map on spectra. Since $\cat{U}[n]$ is essentially finite, by \cref{thm:spcfinite} the map $\mathfrakp_{-}^{\cat{U}[n]}\colon \pi_0\cat{U}[n] \to \Spc(\D{\cat{U}[n]}^c)$ is a homeomorphism and hence \begin{equation}\label{eq:suppispullback}
         \supp(X) = \varphi_n^{-1}\supp(X_n) = \varphi_n^{-1}(\mathfrakp_{-}^{\cat{U}[n]}\hsupp(X_n)).
     \end{equation}
     By \cref{lem:groupprimes_naturality}, there is a commutative diagram
        \[
            \xymatrix{\pi_0\cat{U} \ar[r]^-{\mathfrakp_{-}^\cat{U}} \ar[d]_{\pi_0q_n} & \Spc(\D{\cat{U}}^{c}) \ar[d]^{\varphi_n} \\
            \pi_0\cat{U}[n] \ar[r]_-{\mathfrakp_{-}^{\cat{U}[n]}}^-{\simeq} & \Spc(\D{\cat{U}[n]}^{c})}.
        \]
     So by applying $(\mathfrakp_{-}^\cat{U})^{-1}$ to \eqref{eq:suppispullback} and using injectivity of $\mathfrakp_{-}^{\cat{U}[n]}$, we obtain
        \begin{align}\label{align:prop-thick-building}
        \begin{split}
            \hsupp(X) &= (\mathfrakp_{-}^\cat{U})^{-1}\supp(X) \\ 
            &= (\mathfrakp_{-}^\cat{U})^{-1}\varphi_n^{-1}\mathfrakp_{-}^{\cat{U}[n]}\hsupp(X_n) \\ 
            & = (\pi_0q_n)^{-1}(\mathfrakp_{-}^{\cat{U}[n]})^{-1}\mathfrakp_{-}^{\cat{U}[n]}\hsupp(X_n) \\ 
            & = (\pi_0q_n)^{-1}\hsupp(X_n).
        \end{split}
        \end{align}
     Since $\pi_0q_n$ is surjective, it follows that $(\pi_0q_n)(\hsupp(X)) = \hsupp(X_n)$ and similarly for $Y$. 
     By the assumption that $\hsupp(X) \subseteq \hsupp(Y)$, we deduce that $\hsupp(X_n) \subseteq \hsupp(Y_n)$, and hence that $\supp(X) \subseteq \supp(Y)$ by \eqref{eq:suppispullback}. From general tt-geometry and \cref{rem-ess-finite-refl-standard}, it follows that $X \in \thickt{Y}$ as claimed.
\end{proof}

We can extend \cref{prop-thick-building} to a classification of the lattice of finitely generated thick ideals $\FIdl(\D{\cat{U}}^c)$ in $\D{\cat{U}}^c$ as follows.
\begin{Def}\label{defn-asym}
 Suppose that $\cat{U}$ is equipped with an essentially finite, reflective filtration $\{\cat{U}[n]\}_{n \geq 0}$ with reflections $q_n \colon \cat U \to \cat U[n]$. 
 We say that a subset $S\subseteq \pi_0\cat{U}$ is \emph{asymptotic} if
 there exists a natural number $n$ and a subset $S'\subseteq \pi_0\cat{U}[n]$ such that
 $S=(\pi_0q_n)^{-1}(S')$. We write $\asym(\cat{U})$ for the lattice of
 asymptotic subsets of $\pi_0\cat{U}$ with meet given by intersection and join given by union. 
\end{Def}

\begin{Rem}\label{rem-independent}
    In the above context it is clear that every essentially finite subcategory of $\cat{U}$ is contained in some $\cat{U}[n]$.  From that we see that any two essentially finite reflective filtrations give the same collection of asymptotic subsets. In particular, the notation $\asym(\cat{U})$ is justified since it is independent of the choice of filtration.
\end{Rem}

\begin{Rem}
    Notice that the asymptotic subsets arise from the filtration pieces:
    \[
        \asym(\cat{U}) = \textstyle\bigcup_{n \geq 0} (\pi_0q_n)^{-1}(\mathcal{P}(\pi_0\cat{U}[n])),
    \]
    where $\mathcal{P}(\pi_0\cat{U}[n])$ denotes the set of all subsets of $\pi_0\cat{U}[n]$.
\end{Rem}

\begin{Thm}\label{thm:fgttideals}
    Let $\cat U$ be a subcategory satisfying \cref{hyp2}. If $\cat{U}$ has an essentially finite, reflective filtration $\{\cat{U}[n]\}_{n \geq 0}$, then the homological support gives an isomorphism of lattices
    \[
        \hsupp\colon \FIdl(\D{\cat{U}}^c) \xrightarrow{\sim} \asym(\cat{U}).
    \]
\end{Thm}
\begin{proof}
Firstly, we note that the homological support of any compact object is an asymptotic subset by \eqref{align:prop-thick-building}. Since the Balmer support $\supp(-)$ is a lattice map and $\hsupp = (\mathfrakp_{-}^{\cat{U}})^{-1}\supp$, we deduce that $\hsupp(-)$ is also a lattice map. As such, we have a well-defined lattice map
\[\hsupp\colon \FIdl(\D{\cat{U}}^c) \to \asym(\cat{U}).\]
This is injective by \cref{prop-thick-building}, so it remains to show surjectivity. Suppose that $S$ is an asymptotic subset of $\pi_0\cat{U}$ so that $S = (\pi_0q_n)^{-1}(S')$ for some natural number $n$ and subset $S' \subseteq \pi_0\cat{U}[n]$. By \cref{thm:spcfinite}, we therefore know that $S' = \hsupp(X)$ for some $X \in \D{\cat{U}[n]}^c$. Now $q_n^*X \simeq (i_n)_!X$ is compact by \cref{prop-gen-fun-preserve-compacts}(b), and $\hsupp(q_n^*X) = (\pi_0q_n)^{-1}\hsupp(X) = S$ by \eqref{eq:hsupprestriction} which concludes the proof.
\end{proof}

\begin{Rem}\label{rem:spectralclosure}
    Formally, in the  situation of \cref{thm:fgttideals}, the bounded distributive lattice $\FIdl(\D{\cat{U}}^c)$ and thus also $\asym(\cat{U})$ determine the lattice $\Idl(\D{\cat{U}}^c)$ of all thick ideals in $\D{\cat{U}}^c$, see \cite[Theorem 3.1.9]{KockPitsch2017}. Indeed, writing $\mathrm{Id}(\cat L)$ for the lattice of ideals in a bounded distributive lattice $\cat L$, there are bijections
        \[
            \Idl(\D{\cat{U}}^c) \cong \mathrm{Id}(\FIdl(\D{\cat{U}}^c)) \cong \mathrm{Id}(\asym(\cat{U})).
        \]
    Via Stone duality, this means that $\pi_0\cat{U}$ along with the collection of asymptotic subsets determines the spectrum of $\D{\cat{U}}^c$. However, in general $\mathrm{Id}(\asym(\cat{U}))$ does not embed into the set of subsets of $\pi_0\cat{U}$ via homological support, in contrast to the description of finitely generated thick ideals in terms of asymptotic subsets. For example, for $\cat U = \cat{C}_p$ the family of cyclic $p$-groups, $\hsupp$ cannot distinguish between $\D{\cat{U}}^c$ and $\mathfrakp_{\infty}$ which are both supported on all of $\pi_0\cat{U}$, see \cref{exa:cyclicp}. 
    
    Reexpressed in the context of \cref{rem-bad-unions}, we are observing that the open subsets of the spectrum are not determined by restricting to $\pi_0\cat{U}$. Consequently, we need to enlarge $\pi_0\cat{U}$ to obtain a parametrization of all thick ideals of $\D{\cat{U}}^c$; this is the subject of the next part.
\end{Rem}

\newpage
\part{The Balmer spectrum for families with profinite filtrations}
We begin this part by introducing a new type of prime ideal in $\D{\cat U}^c$ which we call profinite group primes. These new prime ideals are classified by an extension of $\pi_0\cat{U}$ which we denote by $\pi_0\hCU$. We then define a profinite group prime map and show that it is injective if $\cat U$ admits a special type of reflective filtration which we call profinite, see \cref{prop:profinitemap-inj}. If $\cat U$ admits a suitable profinite reflective filtration, we can equip $\pi_0\hCU$ with a profinite topology which is constructed by realizing $\pi_0\hCU$ as an inverse limit of the filtration pieces (\cref{thm-Uhat-is-a-lim}). As our main result of this part, we show that the profinite group map is then a homeomorphism, and we apply this result to compute the Balmer spectrum for $r$-submultiplicative families. 

As a guide for the reader, the following diagram summarizes the main results of this part together with references. For a family $\cat{U}$ of $r$-generated groups which has an essentially finite, reflective filtration, we will construct maps
\[\pi_0\cat{U} \hookrightarrow \pi_0\hCU \xrightarrow{\sim} \Spc(\D{\cat{U}}^c).\] 
These maps induce the indicated bijections, demonstrating the various different perspectives on the classification of (finitely generated) thick ideals of $\D{\cat{U}}^c$ which we establish:
        \[
            \begin{tikzcd}[ampersand replacement=\&, column sep=large]
                {\begin{Bmatrix}
                    \text{thick ideals} \\
                    \text{of } \D{\cat{U}}^c
                \end{Bmatrix}} \arrow[r, "\hsupp", "\sim\, (\ref{Thm-profinite-rgenerated})"'] \&
                {\begin{Bmatrix}
                    \text{open subsets} \\
                    \text{of } \pi_0\hCU
                \end{Bmatrix}} 
                    \& 
                 \\
                {\begin{Bmatrix}
                    \text{finitely generated} \\
                    \text{thick ideals} \\
                    \text{of } \D{\cat{U}}^c
                \end{Bmatrix}} \arrow[r, "\hsupp", "\sim\, (\ref{Thm-profinite-rgenerated})"'] \arrow[u, hookrightarrow] \&
                {\begin{Bmatrix}
                    \text{clopen subsets} \\
                    \text{of } \pi_0\hCU
                \end{Bmatrix}} \arrow[r, "\pi_0\cat{U} \cap -", "\sim\, (\ref{cor-recover-asym})"'] \arrow[u, hookrightarrow]
                    \& 
                \asym(\cat{U}). 
            \end{tikzcd}
        \]
We remind the reader that \cref{hyp} is in place throughout.

\section{Profinite group primes}\label{sec-profinite-group-primes}

In this section we show that for a family $\cat{U}$, there are tt-primes associated to each finitely topologically generated profinite group built out of the groups in $\cat{U}$.
Let us start by setting up some notation and definitions. 

\begin{Def}\label{defn-hCG}
 We define $\hCG$ to be the class of finitely topologically generated profinite groups. Recall that a topological group $G$ is topologically generated by a set $S$ of elements if the closure of the subgroup of $G$ generated by $S$ is $G$. As usual, we regard $\hCG$ as a category, with the conjugacy classes of surjective homomorphisms as the morphisms.  We also define $\hCG\ip{r}$ to be the subcategory of groups that can be generated topologically by a set with at most $r$ elements. 
\end{Def}

\begin{Rem}\label{rem-hCG}
 It is known~\cite{nise:fgpgi, nise:fgpgii} that for $G\in\hCG$ the open subgroups are the same as the subgroups of finite index, and thus that all homomorphisms between groups in $\hCG$ are automatically continuous. (This is a hard theorem depending on the classification of finite simple groups, but we only use it for convenience; it could be avoided by including continuity in our definition of $\hCG$.)  Just as in \cref{lem-Hom-finite}, if $G\in\hCG$ and $H\in\cat{G}$ then $\mathrm{Hom}(G,H)$ and $\Hom_{\hCG}(G,H)$ are finite.
\end{Rem}

\begin{Def}\label{defn-hCU}
 Let $\cat{U}\subseteq\cat{G}$ be a subcategory.  
 For $G\in\hCG$, we let $\cat{N}(G)$ be the set of open normal subgroups of
 $G$, and put 
    \[
    \cat{N}(G;\cat{U})\coloneqq \{N\in\cat{N}(G)\mid G/N\in\cat{U}\}.
    \]  
    We then set
 \[ 
 \hCU \coloneqq \{G\in\hCG\mid \cat{N}(G;\cat{U}) \text{ is cofinal in } \cat{N}(G)\} = 
     \left\{G\mid G \cong \lim_{N\in\cat{N}(G;\cat{U})} G/N \right\}
 \]
 and refer to it as the \emph{profinite extension} of $\cat U$.
\end{Def}

We now provide some examples of profinite extensions for simple families $\cat{U}\subseteq\cat{G}$.

\begin{Exa}\label{Exa-Cyclicp-Uhat}
    Let $\cat{U}  =\cat{C}_p$ denote the collection of cyclic $p$-groups for a prime number $p$. Then $\widehat{\cat{C}}_p$ consists of the cyclic $p$-groups, plus the profinite group $\Z_p$ of $p$-adic integers. We can make $\widehat{\cat{C}}_p$ into a poset by declaring that $G \gg H$ if there exists a surjective group homomorphism $G \twoheadrightarrow H$. A representation of this poset can be seen in \cref{fig:prk1}.
    \begin{figure}[h!]
    \centering
\begin{tikzcd}[column sep={2.3cm, between origins}, row sep={1.3cm, between origins}]
0   & \Z/p \arrow[l, two heads]& \Z/p^2 \arrow[l, two heads]& \Z/p^3 \arrow[l, two heads]&\phantom{A}\cdots\phantom{A} \arrow[l, two heads] & \Z_p \arrow[l, two heads]
\end{tikzcd}
\caption{The poset associated to the profinite extension $\widehat{\cat{C}}_p$ for the family $\cat{C}_p$ of cyclic $p$-groups.}
\label{fig:prk1}
\end{figure}
\end{Exa}

\begin{Exa}\label{Exa-pRank2-Uhat}
    Let $\cat{U}  =\fabprk{2}$ denote the collection of abelian $p$-groups of $p$-rank at most $2$ for a prime number $p$. Then $\hfabprk{2}$ consists of the groups in $\fabprk{2}$, plus groups isomorphic to $\Z_p {\times} \Z/p^n$ for each $n \geqslant 0$, plus the group $\Z_p {\times} \Z_p$. The associated poset to $\hfabprk{2}$, defined as in \cref{Exa-Cyclicp-Uhat}, can be seen in \cref{fig:prk2}.
    \begin{figure}[h!]
    \centering
\begin{tikzcd}[column sep={2.3cm, between origins}, row sep={1.3cm, between origins}]
     &  & &                     & & \Z_p {\times} \Z_p \arrow[d, two heads] \arrow[dl, two heads]\\
     &  & &                     &\phantom{\frac{A}{A}}\iddots \phantom{\frac{A}{A}} \arrow[dl, two heads]& \vdots \arrow[d, two heads]\\
     &  & & \Z/p^3 {\times} \Z/p^3 \arrow[d, two heads] \arrow[dl, two heads] & \phantom{A}\ldots \phantom{A} \arrow[l, two heads]& \Z_p {\times} \Z/p^3 \arrow[d, two heads] \arrow[l, two heads]\\
     &  &\Z/p^2 {\times} \Z/p^2 \arrow[dl, two heads] \arrow[d, two heads]& \arrow[l, two heads]\Z/p^3 {\times} \Z/p^2\arrow[d, two heads] \arrow[dl, two heads] \arrow[l, two heads]& \phantom{A}\ldots\phantom{A} \arrow[l, two heads]& \Z_p {\times} \Z/p^2 \arrow[d, two heads] \arrow[l, two heads]\\
     & \Z/p {\times} \Z/p \arrow[dl, two heads] \arrow[d, two heads] & \arrow[l, two heads]\Z/p^2 {\times} \Z/p \arrow[d, two heads] \arrow[dl, two heads]& \arrow[l, two heads] \Z/p^3 {\times} \Z/p \arrow[d, two heads]\arrow[dl, two heads] & \phantom{A}\ldots \phantom{A} \arrow[l, two heads]& \Z_p {\times} \Z/p\arrow[d, two heads] \arrow[l, two heads]\\
0   & \Z/p \arrow[l, two heads]& \Z/p^2 \arrow[l, two heads]& \Z/p^3 \arrow[l, two heads]&\phantom{A}\ldots \phantom{A} \arrow[l, two heads]& \Z_p \arrow[l, two heads]
\end{tikzcd}
\caption{The poset associated to the profinite extension $\hfabprk{2}$ for the family $\fabprk{2}$ of abelian $p$-groups of $p$-rank at most $2$.}
\label{fig:prk2}
\end{figure}
\end{Exa}

\begin{Exa}
    If there are no infinite chains of non-injective epimorphisms for groups in $\cat{U}$, then $\cat{U}=\hCU$. In particular this applies to any essentially finite subcategory $\cat U$.
\end{Exa}

\begin{Exa}\label{exa:elabp_completion}
    Let $\cat{U}=\cat{E}_p$ be the family of elementary abelian $p$-groups. We claim that $\widehat{\cat E_p}=\cat E_p$. Indeed, any $G \in \hCG$  surjects continuously onto only finitely many groups in $\cat{E}_p$, and so the limit defining $G \in \widehat{\cat{E}_p}$ stabilizes at a finite stage. Therefore for the family $\cat{E}_p$ there is no difference between profinite group primes and group primes. 
\end{Exa}

\begin{Rem}
 If $G \in \cat U$, then $\cat{N}(G;\cat{U})$ admits 
 a minimal element, so $\cat U \subseteq \hCU$.   
\end{Rem}

\begin{Rem}\label{rem:Uhat_cofinal}
    Note that $\cat{N}(G;\cat{U})$ is cofinal in the filtered set $\cat{N}(G)$, so $\cat{N}(G;\cat{U})$ is itself filtered. 
\end{Rem}

Recall from \cref{hyp} that we always assume that $\cat{U}$ is widely closed. The following shows that the wide closure condition extends to the
profinite context.

\begin{Lem}\label{lem-profinite-wide}
 Let $H_0\xleftarrow{\alpha_0}G\xrightarrow{\alpha_1}H_1$ be morphisms with $G\in\hCU$ and $H_0,H_1\in\cat{U}$. Then the image of the combined morphism
 $G\to H_0\times H_1$ also lies in $\cat{U}$.
\end{Lem}
\begin{proof}
 The group $\ker(\alpha_0)\cap\ker(\alpha_1)$ lies in $\cat{N}(G)$, and
 $G\in\hCU$ so we can choose $N\in\cat{N}(G;\cat{U})$ with
 $N\leq\ker(\alpha_0)\cap\ker(\alpha_1)$.  We now have a diagram
 $H_0\xleftarrow{}G/N\xrightarrow{}H_1$ in $\cat{U}$ to which we can apply the wide
 closure axiom to obtain the desired conclusion.
\end{proof}

We now give the central definition of this section. 

\begin{Def}\label{defn-profinite-groupprimes}
  For $G\in\hCU$ and $X\in\A{\cat{U}}$ we define 
 \[ 
 X(G) \coloneqq \colim_{N\in\cat{N}(G;\cat{U})} X(G/N). 
 \]
 This is a filtered colimit by \cref{rem:Uhat_cofinal}, and so is an exact functor of $X$.  It is
 also clear that $(X\otimes Y)(G)=X(G)\otimes Y(G)$.  Thus, the subcategory 
 \[ 
 \pri_G \coloneqq \{X\in\D{\cat{U}}^c\mid H_*(X)(G) =0\}
 \]
 is a prime thick ideal in $\D{\cat{U}}^c$. In analogy with \cref{def:groupprimes}, we call these \emph{profinite group primes}.
\end{Def}

\begin{Rem}\label{rem-ev-const}
 In many cases of interest, when $X\in\D{\cat{U}}^c$ and $G\in\hCU$, the
 colimit system defining $H_*(X)(G)$ will be eventually
 constant.  See \cref{prop-ev-const}(b) for a result of this type.
\end{Rem}

We next extend the group prime map \eqref{eq:groupprimemap} to the profinite case.

\begin{Cons}\label{cons:hsuppextension}
Let $\cat U$ be a subcategory satisfying \cref{hyp2}.
The construction of profinite group primes provides a map 
\begin{equation}\label{eq-profinite-group-prime-map}
    \mathfrakp_{-}=\mathfrakp_{-}^{\hCU}\colon \pi_0\hCU \to \Spc(\D{\cat{U}}^c), \quad G \mapsto \mathfrakp_G.
\end{equation}
This map extends the group prime map of \eqref{eq:groupprimemap}, in that the diagram
\[
\begin{tikzcd}
    \pi_0\cat{U} \ar[dr, "\mathfrakp_{-}^{\cat{U}}"] \ar[d, hookrightarrow] & \\
    \pi_0\hCU \ar[r, "\mathfrakp_{-}^{\hCU}"'] & \Spc(\D{\cat{U}}^c)
\end{tikzcd}
\]
commutes. Moreover, the homological support (\cref{def:hsupp}) then extends to a well-defined function taking values in subsets of $\pi_0\hCU$, by setting 
    \[
        \hsupp_{\hCU}(X) \coloneqq \{G \in \pi_0\hCU \mid H_*(X)(G) \neq 0\} 
    \]
for all $X \in \D{\cat U}^c$. Since the intersection of  $\hsupp_{\hCU}(X)$ with $\pi_0\cat U$ coincides with the usual homological support $\hsupp(X)$, we will mostly drop the subscript from the notation for the \emph{extended homological support} function.
\end{Cons}

\begin{Rem}\label{rem:exthsupp}
    Note that by definition, for any $X \in \D{\cat U}^c$, we have
        \begin{align*}
            \hsupp_{\hCU}(X) & = \{G \in \pi_0\hCU \mid X \notin \pri_G\} \\ 
                & = \{G \in \pi_0\hCU \mid \pri_G \in \supp(X)\} \\
                & = (\mathfrakp_{-}^{\hCU})^{-1}(\supp(X)),
        \end{align*}
    where $\supp(X)$ denotes the universal support of $X$ (see \eqref{eq:ttclassification}). 
\end{Rem}

The following example shows that in general the profinite group prime map \eqref{eq-profinite-group-prime-map} is not bijective. However, in the following sections we provide conditions on $\cat{U}$ which ensure that it is bijective, and demonstrate that these conditions are satisfied for a wide range of natural examples, see \cref{sec-bounded}. 

\begin{Exa}
    Continuing \cref{exa:elabp_completion}, we saw in \cref{thm-elementary} that the Balmer spectrum of $\D{\cat{E}_p}^c$ contains a prime ideal which is not a group prime, and hence the profinite group prime map \eqref{eq-profinite-group-prime-map} is not surjective in this case.
\end{Exa}

\section{Profinite filtrations}\label{sec-profinite-filtrations}
In this section we introduce a special kind of reflective filtration, which we call \emph{profinite}, in which the reflections are compatible with the profinite extension $\hCU$ of $\cat{U}$. We will show that if $\cat U$ admits such a filtration then its profinite extension admits a concrete description in terms of the inverse limit presentation of \cref{cor-ess-finite-invlim}. In the case that the filtration is $r$-generated and essentially finite, we then equip $\hCU$ with a natural profinite topology. We will use these results in the next section to study the profinite group prime map of \eqref{eq-profinite-group-prime-map} and identify the Balmer spectrum for a large collection of families $\cat{U}$.

\begin{Def}\label{defn-filt-profinite}
 A reflective filtration $\{\cat{U}[n]\}_{n\geq 0}$ of $\cat{U}$ is
 \emph{profinite} if each subcategory $\cat{U}[n]$ is reflective in $\hCU$ (and not just in $\cat{U}$), that is, the inclusion $\cat U[n] \to \hCU$ admits a left adjoint $q_n\colon \hCU \to \cat{U}[n]$.  Note that the restriction of $q_n$ to $\cat{U} \to \cat{U}[n]$ (\cref{rem-nested-reflection}) agrees with the original reflection.
\end{Def}

The profinite condition lets us write any $G\in \hCU$ as a filtered limit of its reflections. 

\begin{Prop}\label{prop-ev-const}
     Let $\{\cat{U}[n]\}_{n\geq 0}$ be a profinite reflective filtration of $\cat{U}$ with reflections $q_n \colon \hCU \to \cat U[n]$ and inclusions $i_n \colon \cat U[n]\to \hCU$.
 \begin{enumerate}
     \item  For any $G\in\hCU$ the groups $\ker(G\to q_n(G))$ are cofinal in
 $\cat{N}(G;\cat{U})$, so $G\cong \lim_nq_n(G)$ and $X(G)\cong \colim_n X(q_n(G))$ 
 for all $X\in\A{\cat{U}}$.
 \item For any $X\in\D{\cat{U}}^c$, there exists $n\geq 0$ such
 that:
 \begin{itemize}
 \item[(i)] the counit map $q_n^*i_n^*X \to X$ is an isomorphism; 
 \item[(ii)] $i_n^*X$ is compact;
 \item[(iii)] $H_*(X)(G)\cong H_*(X)(q_n(G))$ for all $G\in\hCU$.
 \end{itemize}
 In particular, the colimit system defining $H_*(X)(G)$ is eventually constant.
 \end{enumerate}
\end{Prop}
\begin{proof}
    Suppose that $M\in\cat{N}(G;\cat{U})$.  Then $G/M\in\cat{U}$, so for sufficiently
 large $n$ we have $G/M\in\cat{U}[n]$, so the map $G\to G/M$ must factor
 through $q_n(G)$, so $\ker(G\to q_n(G))\leq M$. Part (a) follows
 from this. For part (b), recall that by \cref{prop-counit} there exists an $n$ such that the counit map $q_n^*i_n^*X \to X$ is an isomorphism and $i_n^*X$ is compact. Thus $X$ is in the image of $q_n^*\colon \D{\cat{U}[n]}^c\to\D{\cat{U}}^c$. 
 Using this and the fact that $q_n^*$ is exact and so commutes with passage to homology, 
 we have 
 \[
 H_*(X)(q_m(G)) \cong H_*(q_n^*i_n^*X)(q_m(G)) \cong H_*(X)(i_n q_nq_m(G))
 \] 
 which is the same as $H_*(X)(q_n(G))$ when $m\geq n$, so the colimit system defining $H_*(X)(G)$ is eventually constant with value $H_*(X)(q_n(G))$. 
\end{proof}

Most of the reflective filtrations that we have considered are profinite.  As a first step in proving this, we can do essentially the same construction as in \cref{defn-std-filt}, but with profinite groups:
\begin{Def}\label{defn-std-filt-profinite}
 For any $G\in\hCG$, we put 
 \begin{align*}
  \cat{K}_n(G) & \coloneqq \{N\lhd G\mid |G/N|\leq n\}; \\
  K_n(G) & \coloneqq \bigcap_{N\in\cat{K}_n(G)} N = 
    \ker\left(G \to \prod_{N\in\cat{K}_n(G)} G/N\right). \\
 \end{align*}
\end{Def}
With these definitions in place we are ready to prove the following result.

\begin{Prop}\label{prop-invlim}
 The subgroup $K_n(G)$ is open and normal in $G$, so the quotient
 $G/K_n(G)$ is finite.  Moreover, $G$ is the inverse limit of the groups
 $G/K_n(G)$. 
\end{Prop}
\begin{proof}
 By Cayley's theorem, any group of order at most $n$ embeds in the symmetric group $\Sigma_n$, so every group $N\in\cat{K}_n(G)$ is the kernel of some homomorphism $G\to\Sigma_n$.  As $G$ is finitely generated there are only finitely many such homomorphisms (\cref{rem-hCG}), so $\cat{K}_n(G)$ is finite.  It follows that $K_n(G)$ has finite index (and so is open) in $G$, and it is clearly normal. Now let $U\leq G$ be any open subgroup.  Let $N$ be the kernel of the resulting map from $G$ to the finite group of permutations of the set $G/U$, so $N\leq U$ and $N$ again has finite index.  If the index of $N$ is $n$ then we have $K_n(G)\leq N\leq U$.  This shows that the subgroups $K_n(G)$ are cofinal among open normal subgroups of $G$, so $G\cong\lim_n G/K_n(G)$.
\end{proof}

\begin{Rem}\label{rem-hCU-tower}
 If $G\in\hCU$ then we can take $N_1$ to be any finite index subgroup of $G$ with $G/N_1 \in \cat{U}$ and recursively choose open
 normal subgroups $N_{n}\leq K_{n}(G)\cap N_{n-1}$ with $G/N_{n}\in\cat{U}$ for $n \geq 2$.
 We then find that the sequence of subgroups $N_n$ is cofinal in $\cat{N}(G;\cat{U})$, so
 $G\cong\lim_nG/N_n$ and $X(G)\cong\colim_nX(G/N_n)$ for all $X\in\A{\cat{U}}$.
\end{Rem}

We now turn to verifying that several types of reflective filtration are automatically profinite.
\begin{Lem}\label{lem-profinite-refl}
 If $\cat{U}\subseteq\cat{G}$ and $\cat{U}$ is reflective in $\CFG$, then
 $\cat{U}$ is also reflective in $\hCG$.  Similarly, if $\cat{U}\subseteq\cat{G}\ip{r}$
 and $\cat{U}$ is reflective in $\CFG\ip{r}$, then $\cat{U}$ is also reflective
 in $\hCG\ip{r}$.
\end{Lem}

\begin{proof}
 We will prove the first claim; the proof of the second is essentially the same. Let $q\colon \CFG\to\cat{U}$ be the reflection. Now consider a finitely topologically generated group $G\in\hCG$, which means that for some $r$ we have a homomorphism $\alpha\colon F_r\to G$ from the free group on $r$ elements to $G$ with dense image. The group $q(F_r)\in\cat{U}$ is finite, with order $m$ say.  Consider a surjective homomorphism $\phi\colon G\to H$ with $H\in\cat{U}$.  As $\alpha$ has dense image, the composite $\phi\alpha\colon F_r\to H$ must be surjective, so it must factor through $q(F_r)$ by the adjunction property of $q$. Therefore $|H|\leq|q(F_r)|=m$, so $K_m(G)\leq\ker(\phi)$.  This means that $\phi$ factors through the finite group $G/K_m(G)$, and thus also through the group $q(G/K_m(G))$. We can therefore define a reflection $\widehat{q}\colon \hCG\to\cat{U}$ by $\widehat{q}(G)\coloneqq q(G/K_m(G))$.
\end{proof}

\begin{Prop}\label{prop-essfin-profinite}
    Let $\cat{V}$ be reflective in $\cat U$. If $\cat V$ is essentially finite, then $\cat{V}$ is also reflective in $\hCU$.
\end{Prop}
\begin{proof}
    Write $q \colon \cat U \to \cat V$ for the reflection. Let $n$ be the maximum of the orders of the groups in $\cat{V}$, and let $N$ be any open normal subgroup of $G\in \hCU$ contained in $K_n(G)$ such that $G/N\in\cat{U}$.  It is clear that any surjective homomorphism from $G$ to a group in $\cat{V}$ factors through $G/N$, and thus also through the group $q(G/N)$. We can therefore define a reflection $\widehat{q}\colon \hCU \to \cat V$ by $\widehat{q}(G)\coloneqq q(G/N)$.
\end{proof}

\begin{Cor}\label{cor-std-filt-profinite}
 Suppose that $\{\cat{U}[n]\}_{n\geq 0}$ is a reflective filtration of
 $\cat{U}$.  Suppose that any of the following conditions are satisfied:
 \begin{enumerate}
  \item $\cat{U}[n]$ is reflective in $\CFG$; or
  \item $\cat{U}[n]$ is essentially finite; or
  \item there exists $r$ such that $\cat{U}\subseteq\cat{G}\ip{r}$ and
   $\cat{U}[n]$ is reflective in $\CFG\ip{r}$. 
 \end{enumerate}
 Then the filtration is profinite. 
 In particular, the filtrations provided by \cref{prop-std-filt} are
 profinite. 
\end{Cor}
\begin{proof}
Part (b) follows from \cref{prop-essfin-profinite}, while (a) and (c) follow from \cref{lem-profinite-refl} by restricting the reflection to $\hCU$.
\end{proof}

Now that we have a good amount of examples of profinite reflective filtrations we can discuss some consequences. We start by noting that if $\cat U$ admits a profinite reflective filtration $\{\cat U[n]\}_{n\geq 0}$, then we can combine the reflections into a map  
\[
q\coloneqq (q_n)\colon \pi_0 \hCU \to \lim_n \pi_0 \cat U[n], \quad [G] \mapsto ([q_n G])_n
\] 
which we will show is always injective, see \cref{prop:Uhat-injects-in-a-lim}. We then investigate conditions on $\cat U$ which ensure that the above map is bijective. This culminates in \cref{thm-Uhat-is-a-lim}, which will be a key step in describing the 
Balmer spectrum for $r$-submultiplicative families in the next section. In order to prove injectivity of $q$, it will be important to understand in what sense an element of $\hCU$ is determined by the groups in $\cat U$ which it surjects onto. To this end, we make the following definition:

\begin{Def}\label{defn-quot-set}
 For $G\in\hCU$, we define $\cat{U}_{\ll G}$ to be the set of groups in $\cat{U}$
 that admit a surjective homomorphism from $G$.
\end{Def}

We prove that $\cat{U}_{\ll G}$ characterizes both the existence and behaviour of maps in $\hCU$. 

\begin{Lem}\label{lem-tow-maps}
 Suppose that $G,H\in\hCU$.
 \begin{enumerate}
  \item There is a surjective homomorphism $G\to H$ if and only if $\cat{U}_{\ll G}\supseteq\cat{U}_{\ll H}$.
  \item There is an isomorphism $G\to H$ if and only if $\cat{U}_{\ll G}=\cat{U}_{\ll H}$.
  \item Any endomorphism $G\to G$ in $\hCU$ is an isomorphism.
 \end{enumerate}
\end{Lem}
\begin{proof}
 If there is a surjective homomorphism $G\to H$ then it is clear that
 $\cat{U}_{\ll G}\supseteq\cat{U}_{\ll H}$, and if $G\cong H$ then it is clear that
 $\cat{U}_{\ll G}=\cat{U}_{\ll H}$.

 For the converse, suppose that $\cat{U}_{\ll G}\supseteq\cat{U}_{\ll H}$.  We can
 choose a tower $(H/N_k)_k$ as in \cref{rem-hCU-tower} so that $H \cong \lim_k H/N_k$.  Let
 $\mathrm{Epi}(G,H/N_k)$ denote the set of surjective homomorphisms
 $G\to H/N_k$.  This is non-empty by assumption, and finite because $G$
 is finitely generated and $|H/N_k|<\infty$.  We note that the inverse limit of any tower of finite non-empty sets is again non-empty.  Using this, we can construct a homomorphism
 $\alpha\colon G\to H$ such that each composite $G\to H\to H/N_k$ is
 surjective, or equivalently the image of $\alpha$ is dense. Since $\alpha$ is continuous (see \cref{rem-hCG}) and $G$ is quasi-compact, the image $\alpha(G)\leq H$ is quasi-compact, and so closed (as $H$ is Hausdorff). This shows that $\alpha$ is in fact surjective, proving~(a).

 We next discuss~(c).  Any endomorphism of $G$ is represented by a
 surjective homomorphism $\gamma\colon G\to G$.  From the definition of
 $K_n(G)$ (see \cref{defn-std-filt-profinite}) we see that $\gamma(K_n(G))\leq K_n(G)$ and so $\gamma$ induces a
 surjective endomorphism of the finite quotient group $G/K_n(G)$.
 This must be an isomorphism by counting, and $G\cong \lim_nG/K_n(G)$,
 so $\gamma$ is an isomorphism as claimed.

 Finally, we consider the reverse implication in (b).  If $\cat{U}_{\ll G}=\cat{U}_{\ll H}$ then we can use~(a) to
 obtain morphisms $G\xrightarrow{\alpha}H\xrightarrow{\beta}G$.  Claim~(c) shows that
 $\alpha\beta$ and $\beta\alpha$ must be isomorphisms, and it follows that $\alpha$
 and $\beta$ are isomorphisms, as required.
\end{proof}

\begin{Cor}\label{prop:Uhat-injects-in-a-lim}
    Let $\{\cat U[n]\}_{n\geq 0}$ be a profinite reflective filtration of $\cat U$.
    Then the map induced by the reflections 
    \[
    q\coloneqq (q_n)\colon \pi_0 \hCU \hookrightarrow \lim_n \pi_0 \cat U[n], \quad [G] \mapsto ([q_n G])_n
    \]
    is injective.
\end{Cor}
\begin{proof}
    Notice that $\cat{U}_{\ll q_n(G)}=\cat{U}_{\ll G}\cap\cat{U}[n]$ and hence
 $\cat{U}_{\ll G}=\bigcup_n\cat{U}_{\ll q_n(G)}$. Therefore by \cref{lem-tow-maps}, we see that given $G,H \in \hCU$, there is an isomorphism $G \to H$ if and only if $q_n(G) \cong q_n(H)$ for all $n$. 
\end{proof}

We now turn to proving surjectivity of the map $q$.
In order to do this, we need the following criterion. Recall the definition of wide subgroups from \cref{not-families}.
\begin{Prop}\label{prop-quot-set-props}
 Consider a subset $\cat{X}\subseteq\pi_0\cat{U}$.  There is an object
 $G\in\hCU$ with $\cat{X}=\pi_0\cat{U}_{\ll G}$ if and only if $\cat{X}$ has the following properties:
 \begin{enumerate}
  \item there is a number $r$ with $\cat{X}\subseteq\cat{U}\ip{r}$;
  \item $\cat{X}$ is non-empty, and closed under quotients in $\cat{U}$; and
  \item for any $H,K\in\cat{X}$ there is a wide subgroup 
   $L\leq H\times K$ with $L\in\cat{X}$.
 \end{enumerate}
\end{Prop}
\begin{proof}
 First suppose that $\cat{X}=\pi_0\cat{U}_{\ll G}$ for some $G\in\hCU$.  Then
 $G\in\hCU\ip{r}$ for some $r$, which implies~(a).  Property~(b)
 is clear.  If $H,K\in\cat{X}$ then we can choose surjective
 homomorphisms $H\xleftarrow{\alpha}G\xrightarrow{\beta}K$ and put
 $L=\img(\ip{\alpha,\beta}\colon G\to K\times L)$; then $L$ is a wide subgroup that
 is contained in $\cat{X}$; so property~(c) is also satisfied.

 Conversely, suppose we have a subset $\cat{X}\subseteq\pi_0\cat{U}$ satisfying~(a)
 to~(c).  As $\pi_0\cat{G}$ is countable, we
 can choose a sequence of groups $H_i\in\cat{X}$ containing a representative for each of the isomorphism classes of objects in $\cat{X}$. Take $G_0=H_0$ and then recursively
 choose a wide subgroup $G_{i+1}\leq H_{i+1}\times G_i$ for each $i$, so
 we have projections $G_i\leftarrow G_{i+1}\to H_{i+1}$.
 Let $G$ be the inverse limit of the groups $G_i$.  As
 $G_i\in\cat{X}\subseteq\cat{U}\ip{r}$ for all $i$, by applying \cref{lem-tow-maps}(a) to the profinite completion of $F_r$, the free group on $r$ generators, we deduce that $G\in\hCU\ip{r}$. If $T\in\cat{U}$ is a finite quotient of $G$,
 then the map $G\to T$ must factor through $G_i$ for some $i$. Since
 $G_i\in\cat{X}$ and $\cat{X}$ is closed under quotients in $\cat{U}$ we have $T\in\cat{X}$.
 Conversely, if $T\in\cat{X}$ then $T\cong H_i$ for some $i$ so we have a
 surjective homomorphism $G\to G_i\xrightarrow{q_i}H_i\cong T$ so
 $T\in\cat{U}_{\ll G}$.  We therefore have $\pi_0\cat{U}_{\ll G}=\cat{X}$ as required.
\end{proof}

Using the above results we can give another description of $\pi_0\hCU$ which will be instrumental in calculating the Balmer spectrum for $r$-submultiplicative families. 

\begin{Thm}\label{thm-Uhat-is-a-lim}
    Let $\cat{U} \subseteq \cat{G}$ be a subcategory in which all groups are $r$-generated for a fixed $r \geq 0$, and suppose that $\cat{U}$ has a profinite reflective filtration $\{\cat U[n]\}_{n\geq 0}$.
    Then the map induced by the reflections 
    \[
    q = (q_n)\colon \pi_0 \hCU \xrightarrow{\sim} \lim_n \pi_0 \cat U[n], \quad [G] \mapsto ([q_n G])_n
    \]
    is a bijection.
\end{Thm}
\begin{proof} 
Injectivity holds by \cref{prop:Uhat-injects-in-a-lim}.
To show surjectivity, consider an element $([G_n])_n \in \lim_n \pi_0 \cat U[n]$ so that $q_n(G_{n+1})\cong G_n$. Set $\cat X$ to be the downward closure of $\{G_n \mid n\geq 0 \}$. We now apply \cref{prop-quot-set-props} to this $\cat X$ to obtain $G \in \hCU$ such that $\pi_0\cat{U}_{\ll G}=\cat X$: in order to do this, we need to verify conditions (a)--(c) from \cref{prop-quot-set-props}. Conditions (a) and (b) are clear. For condition (c), given $H, K \in \cat{X}$ there is a natural number $i$ such that $G_i$ surjects onto both $H$ and $K$. The image of the combined homomorphism $G_i \to H \times K$ is then a wide subgroup which lies in $\cat{X}$ (see \cref{eg-implies-widely-closed}) as required. 

We now claim that $q_n(G)\cong G_n$ for all $n$, showing that the above map is surjective. To prove the claim note that $G \gg G_n$ and so by adjunction $q_n(G)\gg G_n$. On the other hand, $G\gg q_n(G)$ and by definition of $\cat X$ there is some $m$ such that $G_m\gg q_n(G)$. Applying $q_n$ to this homomorphism, we obtain that $G_n\cong q_n(G_m)\gg q_n(G)$, and so $q_n(G)\cong G_n$, as claimed.
\end{proof}

\begin{Exa}
    We observe that the previous theorem applies to $r$-submultiplicative families by \cref{cor-std-filt-profinite}.
\end{Exa}

Finally, we use \cref{thm-Uhat-is-a-lim} to endow $\hCU$ with a natural profinite topology under the additional assumption that the filtration on $\cat{U}$ is essentially finite. 
 
\begin{Def}\label{defn-Uhat-topology}
    Let $\cat{U} \subseteq \cat{G}$ be a subcategory in which all groups are $r$-generated for a fixed $r \geq 0$. Assume in addition that $\cat{U}$ has an essentially finite, reflective filtration $\{\cat U[n]\}_{n\geq 0}$; in particular, the filtration is profinite by \cref{cor-std-filt-profinite}, with reflections $q_n \colon \hCU \to \cat U[n]$. Considering each finite set $\pi_0\CU[n]$ as a discrete space, we give 
        \[
            \pi_0\hCU\xrightarrow{\sim} \lim_n \pi_0\cat{U}[n]
        \]
    the inverse limit topology afforded by \cref{thm-Uhat-is-a-lim}. Explicitly, for every $n$ and $S\subseteq\pi_0\cat{U}[n]$ we have an open set $(\pi_0q_n)^{-1}(S)\subseteq\pi_0\hCU$, and these sets form a basis for the topology on $\pi_0\hCU$. 
\end{Def}

\begin{Rem}\label{rem:profinitetopology}
     In the above context it is clear that every essentially finite subcategory of $\cat U$ is contained in some $\cat U[n]$. From this we see that the induced topology on $\pi_0\hCU$ is independent of the choice of filtration. Moreover, the space $\pi_0\hCU$ is profinite, i.e., quasi-compact, Hausdorff, and totally disconnected.
\end{Rem}

We record some basic features of the topology on $\pi_0\hCU$.
\begin{Lem}\label{lem-hatU-topology}
    Let $\cat{U}$ be as in \cref{defn-Uhat-topology}. 
    \begin{enumerate}
        \item The following are equivalent for a subset $U \subseteq \pi_0\hCU$:
        \begin{itemize}
            \item[(i)] $U$ is clopen (i.e., closed and open);
            \item[(ii)] $U$ is quasi-compact and open;
            \item[(iii)] there exists $n \in \bbN$ and $S \subseteq \pi_0\cat{U}[n]$ such that $U = (\pi_0q_n)^{-1}(S)$.
        \end{itemize}
        \item A subset of $\pi_0\hCU$ is open if and only if it is Thomason, i.e., a union of subsets with quasi-compact open complement.
        \item Every clopen subset of $\pi_0\hCU$ is the closure of its intersection with $\pi_0\cat{U}$. In particular, the subset $\pi_0\cat{U}$ is dense in $\pi_0\hCU$.
    \end{enumerate}
\end{Lem}
\begin{proof}
    Since $\pi_0\hCU$ is a quasi-compact Hausdorff space, a subset is closed if and only if it is quasi-compact, which gives the equivalence of $(i)$ and $(ii)$ of $(a)$. If $S\subseteq\pi_0\cat{U}[n]$ then $(\pi_0q_n)^{-1}(S)$ and $(\pi_0q_n)^{-1}(\pi_0\cat{U}[n] \setminus S)$ are open and complementary, so $(\pi_0q_n)^{-1}(S)$ is clopen.  Conversely, suppose that $U\subseteq\pi_0\hCU$ is quasi-compact and open.  Then $U$ can be written as a finite union of basic open subsets, say $U=\bigcup_{i<m}(\pi_0q_{n_i})^{-1}(S_i)$ for some subsets $S_i\subseteq\pi_0\cat{U}[n_i]$.  Let $n$ be the maximum of the indices $n_i$, let $W_i$ be the preimage of $S_i$ in $\pi_0\cat{U}[n]$, and put $W=\bigcup_{i<m}W_i\subseteq\pi_0\cat{U}[n]$. We then find that $U=(\pi_0q_n)^{-1}(W)$ which completes the proof of (a). Since the quasi-compact opens form a basis for the topology on $\pi_0\hCU$, (a) implies (b).
    
    For (c), consider a basic open set $U=(\pi_0q_n)^{-1}(S)$ and a point $x\in U$.  We must show that $x$ lies in the closure of $U\cap\pi_0\cat{U}$.  Consider another basic open set $U'=(\pi_0q_{n'})^{-1}(S')$ containing $x$; we must show that $U'\cap U\cap\pi_0\cat{U}$ is nonempty.  Put $m=\max(n,n')$ and $y=q_m(x)\in\pi_0\cat{U}[m]\subseteq\pi_0\cat{U}$.  We then find that $y\in U'\cap U\cap\pi_0\cat{U}$ as required. The density of $\pi_0\cat{U} \subseteq \pi_0\hCU$ follows by taking the clopen subset $\pi_0\hCU$.
\end{proof}

We can recover the asymptotic subsets (\cref{defn-asym}) of $\pi_0\cat{U}$ from the topology on $\pi_0\hCU$ as follows:
\begin{Cor}\label{cor-recover-asym}
    Let $\cat{U}$ be as in \cref{defn-Uhat-topology}. There is an isomorphism of lattices
    \[
            \begin{tikzcd}[ampersand replacement=\&,column sep=small]
                \pi_0\cat{U} \cap -\colon
                {\begin{Bmatrix}
                    \text{clopen subsets} \\
                    \text{of } \pi_0\hCU
                \end{Bmatrix}}
                    \& 
                \asym(\cat{U}).
                    \arrow["\sim", from=1-1, to=1-2]
            \end{tikzcd}
        \]
\end{Cor}
\begin{proof}
    Recall that the reflection $\cat U \to \cat U[n]$ is the restriction of $q_n \colon \hCU \to \cat U[n]$, see \cref{defn-filt-profinite}. It then follows that an asymptotic subset of $\pi_0\cat{U}$ is of the form $(\pi_0q_n)^{-1}(S)\cap\pi_0\cat{U}$ for some $n$ and $S\subseteq\pi_0\cat{U}[n]$. \cref{lem-hatU-topology}(a) then shows that the above map is surjective, while injectivity follows from \cref{lem-hatU-topology}(c).
\end{proof}

As we will show in the next section, the reason to consider the extension from asymptotic subsets of $\pi_0\cat{U}$ to clopen subsets of $\pi_0\hCU$ is that $\pi_0\hCU$ is large enough to accommodate a classification of all thick ideals in terms of their homological support; compare with \cref{rem:spectralclosure}.

\section{Balmer spectra for \texorpdfstring{$r$}{r}-submultiplicative families}\label{sec-bounded}

Our goal in this section is to describe the Balmer spectrum for families satisfying:

\begin{Hyp}\label{hyp3}
     Let $\cat{U}$ be a subcategory satisfying \cref{hyp2} and in which all groups are $r$-generated for a fixed $r \geq 0$. Assume in addition that $\cat{U}$ has an essentially finite, reflective filtration $\{\cat U[n]\}_{n\geq 0}$; this is moreover profinite by \cref{cor-std-filt-profinite}, with reflections $q_n\colon \hCU \to \cat{U}[n]$.
\end{Hyp}

\begin{Exa}
    By \cref{cor-std-filt-profinite} $r$-submultiplicative families satisfy \cref{hyp3}.
\end{Exa}

We begin with the compatibility of the profinite group prime map with the maps induced by the reflections, which holds in slightly greater generality.

\begin{Lem}\label{lem:groupprimesmaps-compactible}
    Let $\cat U$ be a subcategory satisfying \cref{hyp2} and admitting a profinite reflective filtration $\{\cat{U}[n]\}_{n\geq 0}$. Then the maps \eqref{eq:groupprimemap} and \eqref{eq-profinite-group-prime-map} are compatible in the sense that there is a commutative diagram
      \[
    \begin{tikzcd}[column sep=2cm]
        \pi_0 \widehat{\cat U} \arrow[r,"\mathfrakp_{-}^{\hCU}"] \arrow[d, "q=(q_n)"'] & \Spc(\D{\cat U}^c) \arrow[d,"(\Spc(q_n^*))"]\\
        \lim_n \pi_0 \cat U[n] \arrow[r, "\lim_n\mathfrakp_{-}^{\cat U[n]}"'] & \lim_n \Spc(\D{\cat U[n]}^c).
    \end{tikzcd}
    \]
\end{Lem}
\begin{proof}
    Fix $n\geq 0$ and $G \in\pi_0 \widehat{\cat U}$. Unravelling the definitions and using \cref{prop-ev-const}, we see that
    \begin{align*}
        \Spc(q_n^*)(\mathfrakp_G) & = \{X \mid (q_n^*X)(G)\simeq 0\}\\
                                  & = \{X \mid \colim_m (q_n^*X)(q_m G)\simeq 0\}\\
                                  & = \{X \mid \colim_m X(q_nq_m G)\simeq 0\}\\
                                  & = \{X \mid X(q_n G)\simeq 0\} =\mathfrakp_{q_nG},  
    \end{align*}
    where in the penultimate equality we used that $q_nq_m G\cong q_nG$ for $m\geq n$.  
\end{proof}

We now show that, for $\cat{U}$ with a profinite reflective filtration, the set of profinite group primes injects into the Balmer spectrum, extending \cref{prop:groupprimes}.

\begin{Prop}\label{prop:profinitemap-inj}
    Let $\cat U$ be a subcategory satisfying \cref{hyp2} and admitting a profinite reflective filtration $\{\cat U[n]\}_{n\geq 0}$. Then the map of \eqref{eq-profinite-group-prime-map}
    \[
    \mathfrakp_{-}^{\hCU}\colon \pi_0\hCU \hookrightarrow \Spc(\D{\cat{U}}^c)
    \]
    is injective.
\end{Prop}
\begin{proof}
    Consider the commutative diagram of \cref{lem:groupprimesmaps-compactible}. The left vertical arrow is injective by \cref{prop:Uhat-injects-in-a-lim}, and the right vertical arrow is bijective by \cref{cor-lim-spectra}. The bottom horizontal map is a limit of injective maps by \cref{prop:groupprimes} and so is injective. The claim then follows by the commutativity of the diagram.
\end{proof} 

We now come to the main result of this section:
\begin{Thm}\label{Thm-profinite-rgenerated}
    Let $\cat{U}$ be as in \cref{hyp3}. Then the profinite group map \eqref{eq-profinite-group-prime-map} is a homeomorphism
        \[
            \mathfrakp_{-}\colon \pi_0 \hCU \xrightarrow{\sim} \Spc(\D{\cat{U}}^c) 
        \]
    when we endow $\pi_0 \hCU$ with the inverse limit topology as in \cref{defn-Uhat-topology}. Moreover, $\D{\cat U}^c$ is standard in the sense of \cref{def:standardtt} and the extended homological support (\cref{cons:hsuppextension}) induces bijections
    \begin{align*}
        \{\text{finitely generated thick ideals of } \D{\cat{U}}^c\} & \xrightarrow{\sim} 
        \{\text{clopen subsets of } \pi_0\hCU\} \\
        \{\text{thick ideals of } \D{\cat{U}}^c\} & \xrightarrow{\sim} 
        \{\text{open subsets of } \pi_0\hCU\} \\
        \{\text{prime ideals of } \D{\cat{U}}^c\} & \xrightarrow{\sim} 
        \{\text{complements of points in } \pi_0\hCU\} 
    \end{align*}
    extending the bijection of \cref{thm:fgttideals}. 
\end{Thm}

\begin{proof}
  By \cref{lem:groupprimesmaps-compactible}, there is a commutative diagram 
    \[
    \begin{tikzcd}[column sep=2cm]
        \pi_0 \widehat{\cat U} \arrow[r, "\mathfrakp_{-}^{\hCU}"] \arrow[d, "\sim","q=(q_n)"'] & \Spc(\D{\cat U}^c) \arrow[d,"\sim"',"(\Spc(q_n^*))"]\\
        \lim_n \pi_0 \cat U[n] \arrow[r, "\sim"', "\lim_n \mathfrakp^{\cat U[n]}_{-}"] & \lim_n \Spc(\D{\cat U[n]}^c).
    \end{tikzcd}
    \]
    The right vertical arrow is a homeomorphism by \cref{cor-lim-spectra}, the left vertical map is a homeomorphism by \cref{thm-Uhat-is-a-lim}, and the bottom horizontal map is a homeomorphism by \cref{thm:spcfinite}. Therefore, $\mathfrakp_{-}^{\hCU}$ is a homeomorphism as well. 
    
    In light of \cref{rem:exthsupp}, this exhibits the pair $(\pi_0\hCU,\hsupp_{\hCU})$ as the universal support datum in the sense of Balmer \cite[Theorem 3.2]{Balmer2005}. Since $\D{\cat{U}}^c$ is standard by \cref{rem-ess-finite-refl-standard}, the general classification theorem \eqref{eq:ttclassification} then establishes bijections
    \begin{align*}
        \{\text{thick ideals of } \D{\cat{U}}^c\} & \xrightarrow{\sim} 
        \{\text{Thomason subsets of } \pi_0\hCU\} \\
        \{\text{finitely generated thick ideals of } \D{\cat{U}}^c\} & \xrightarrow{\sim}
        \{\text{closed Thomason subsets of } \pi_0\hCU\}
    \end{align*}    
    induced by $\hsupp_{\hCU}$. It remains to observe that a subset of $\pi_0\hCU$ is (closed and) Thomason if and only if it is (closed and) open, see \cref{lem-hatU-topology}. Finally, as already observed in \cref{rem:isolation_criterion}, the support of a prime ideal $\mathfrakp$ is in general given by the complement of $\{\mathfrak{q}\mid \mathfrak{q} \supseteq \mathfrakp\}$. Since $\pi_0\hCU$ is Hausdorff, there are no non-trivial containments between prime ideals, so $\{\mathfrak{q}\mid \mathfrak{q} \supseteq \mathfrakp\}$ is a singleton. The classification of prime ideals in terms of the extended homological support follows. 
\end{proof}

\begin{Rem}\label{rem-bad-unions}
    Recall from \cref{thm:fgttideals} that the finitely generated thick ideals in $\D{\cat{U}}^c$ are classified by the asympotic subsets of $\pi_0\cat{U}$ via the homological support, and that these are in bijection with the clopen subsets of $\pi_0\hCU$ by \cref{cor-recover-asym}. Moreover, the open subsets of $\pi_0\hCU$, i.e., the unions of clopen subsets, classify the thick ideals of $\D{\cat{U}}^c$ by the previous theorem. Therefore, one might wonder whether the unions of asymptotic subsets of $\pi_0\cat{U}$ also biject with all thick ideals. This is not the case, as already observed via an explicit counterexample in \cref{rem:spectralclosure}.
\end{Rem}

\begin{Rem}
    One can also give a lattice theoretic approach to the classification of thick ideals in $\D{\cat{U}}^c$ as briefly discussed in \cref{rem:spectralclosure}. From this perspective, it suffices to compute the frame of ideals in the lattice $\asym(\cat{U})$; equivalently by \cref{cor-recover-asym}, the frame of ideals in the lattice of clopen subsets of $\pi_0\hCU$. Let $\mathsf{L}$ be the lattice of clopen sets of $\pi_0\hCU$, let $\mathsf{F}$ be the frame of ideals in $\mathsf{L}$, and let $\mathsf{F}'$ be the frame of open sets of $\pi_0\hCU$: we will show that $\mathsf{F}$ and $\mathsf{F}'$ are isomorphic.  For any ideal $I\subseteq \mathsf{L}$, we note that $I$ is a collection of clopen subsets of $\pi_0\hCU$, and we let $\phi(I)$ be the union of all those sets.  This defines a map $\phi\colon \mathsf{F}\to \mathsf{F}'$. In the opposite direction, given an open subset $U\subseteq\pi_0\hCU$, we define $\psi(U)$ to be the set of clopen subsets of $U$, which is an ideal in $\mathsf{L}$.  As the clopen subsets form a basis for the topology, we have $\phi(\psi(U))=U$.  Conversely, suppose that we start with an ideal $I\subseteq \mathsf{L}$.  It is immediate that $I\leq\psi(\phi(I))$.  In the opposite direction, suppose we have $V\in\psi(\phi(I))$, so $V$ is a quasi-compact open subset and is contained in the union of all the sets in $I$.  By quasi-compactness, $V$ is contained in the union of some finite family of sets in $I$.  As $I$ is an ideal it is closed downwards and under finite unions, so $V\in I$.  This shows that $I=\psi(\phi(I))$, and we conclude that $\phi$ and $\psi$ are mutually inverse isomorphisms of posets and hence also of frames. 
\end{Rem}

We now apply \cref{Thm-profinite-rgenerated} to compute the Balmer spectrum of $\D{\cat{U}}^c$ for various prominent choices of $\cat{U}$.
\begin{Exa}\label{ex-cyclic-groups-prime-order}
    Let $\Cprime$ denote the collection of cyclic groups of prime order together with the trivial group. Every group in this family is $1$-generated, and the family is downwards closed. Thus combining \cref{rem-res-filt} with the filtration from \cref{prop-std-filt}(b) for $\cat U$ the family of cyclic groups, we obtain an essentially finite, reflective filtration for $\Cprime$. Concretely, this reflection for $n\geqslant 1$ satisfies
    \begin{equation}\label{eq-reflection-Cpr}
    \Cprime[n] = \{C \in \Cprime \mid |C| \leq n\} \quad \text{and} \quad
    q_n(C_p)=
    \begin{cases}
    C_p  & \mathrm{if}\; p \leqslant n \\ 
    1  & \mathrm{if} \; p >n.\\
    \end{cases}
    \end{equation}
    As sets, $\widehat{\Cprime}=\Cprime$ because any sequence of groups and non-injective epimorphisms in $\Cprime$ 
    has length less than or equal to two. Therefore, by \cref{Thm-profinite-rgenerated},
    the map 
    \[
    \mathfrakp_{-}  \colon \pi_0\widehat{\Cprime}=\pi_0\Cprime \to \Spc(\D{\Cprime}^c)
    \]
    is a homeomorphism when $\pi_0\Cprime = \pi_0\widehat{\Cprime}$ is given the profinite topology, so it remains to identify this profinite topology.

    A basis of opens for the topology is given by the subsets of the form $(\pi_0q_n)^{-1}(S)\subseteq\pi_0\Cprime$, and these are precisely the clopen subsets in $\pi_0\Cprime$ by \cref{lem-hatU-topology}. Using \eqref{eq-reflection-Cpr} we see that if a subset $A \subseteq \pi_0\Cprime$ is clopen and $1 \in A$ then the complement of $A$ is finite, and if $1 \notin A$ then $A$ is finite. The arbitrary unions of these basis elements are precisely the subsets $U \subseteq \pi_0\Cprime$ that satisfy that if $1 \in U$ then the complement of $U$ is finite. This shows that $\pi_0\Cprime$ with the profinite topology is homeomorphic to the one-point compactification of the infinite discrete set $\pi_0\Cprime\setminus \{1\}$, with $\mathfrakp_1$ as the accumulation point. See \cref{exa:cyclicprimes} for a representation of this space.
\end{Exa}

\begin{Exa}\label{ex:cyclic-p-groups}
    Let $\cat{C}_p$ denote the collection of cyclic $p$-groups for a fixed prime number $p$; this is a $1$-submultiplicative family. For convenience we will use additive notation. As we saw in \cref{Exa-Cyclicp-Uhat}, $\widehat{\cat{C}}_p$ contains $\cat{C}_p$ plus the profinite group $\Z_p$. We apply \cref{Thm-profinite-rgenerated} and obtain a homeomorphism $\pi_0 \widehat{\cat{C}}_p\cong\Spc(\D{\cat{C}_p}^c)$. In order to identify the profinite topology, we note that the reflections and the filtration of \cref{prop-std-filt}(b) satisfy for $n\geqslant 1$ that
    \[
    \cat{C}_p[n]= \{\Z/p^m \mid p^m \leqslant n \} \quad
    q_{p^n}(\Z/p^m)=  \Z/p^{\min(m,n)} \quad \mathrm{and}\quad q_{p^n}(\Z_p)=\Z/p^n
    \]
    for all $m\geq 0$. Thus if a subset $A \subseteq \pi_0\widehat{\cat{C}}_p$ is clopen and $\Z_p \in A$ then the complement of $A$ is finite, and if $\Z_p \notin A$ then $A$ is finite. As in the previous example, this means that the profinite topology on $\pi_0\widehat{\cat{C}}_p$ is the one-point compactification of the discrete space $\pi_0\cat{C}_p$, with $\Z_p$ as the accumulation point. In contrast to the previous example, here the accumulation point is not a group prime, but a profinite group prime instead. See \cref{exa:cyclicp} for a representation of this space.
\end{Exa}

For the next example, let $\bbN^+$ denote the one-point compactification of $\bbN$ with compactification point denoted by $\infty \in \bbN^+$. We extend the linear order of $\bbN$ to $\bbN^+$ by declaring $\infty > n$ for all $n \in \bbN$.

\begin{Thm}\label{Thm-abelian-p-groups-rank-r}
    Fix a prime number $p$ and $r\in \bbN$, and let $\fabprk{r}$ be the family of abelian $p$-groups of $p$-rank at most $r$, or equivalently, the family of $r$-generated finite abelian $p$-groups. The spectrum $\Spc(\sfD(\fabprk{r})^c)$ is homeomorphic to the subspace
\[
\hatS_{\leq r}\coloneqq\{v=(v_1,\ldots, v_r)\in (\bbN^+)^r \mid \infty \geq v_1 \geq \ldots \geq v_r \geq 0\} \subseteq (\bbN^+)^r,
\]
which is profinite.
\end{Thm}
\begin{proof}
The family $\fabprk{r}$ is an $r$-submultiplicative family and the filtration coming from \cref{cor-std-filt-profinite} is given by filtering by $p$-exponent, $\fabprk{r}[p^l]= \fabprk{r}^{\leq l}$, where $\fabprk{r}^{\leq l}$ consists of all those finite abelian $p$-groups of $p$-rank at most $r$ and $p$-exponent less than or equal to $p^l$.
By \cref{Thm-profinite-rgenerated}, the Balmer spectrum of $\D{\fabprk{r}}^c$ is profinite and homeomorphic to $\pi_0 \hfabprk{r} \cong \lim_l \pi_0 \fabprk{r}^{\leq l}$, so we just need to identify the set $\pi_0 \hfabprk{r}$ and its profinite topology.

Set $v = (v_1,\ldots, v_r)$. We consider the discrete subsets of $\hatS_{\leq r}$
\[
\CS_{\leq r}=\{v \in \hatS_{\leq r}\mid v_1 <\infty \} \quad \mathrm{and}\quad
\CS_{\leq r}^{\leq l}=\{v \in \CS_{\leq r} \mid v_1 \leq l\}
\]
and define functions $m_l \colon \hatS_{\leq r}\to \CS_{\leq r}^{\leq l}$ by $m_l(v)_i=\min(l,v_i)$.
We define a map $\phi\colon \CS_{\leq r} \xrightarrow{} \pi_0\fabprk{r}$ by
\[
\phi(v_1,\ldots,v_r) \coloneqq \left[\bigoplus_i \bbZ/{p^{v_i}}\right].
\] 
The restriction $\phi_l$ of $\phi$ to $\CS_{\leq r}^{\leq l}$ induces a bijection $\CS_{\leq r}^{\leq l} \xrightarrow{\cong} \pi_0 \fabprk{r}^{\leq l}$ for each $l \geqslant 0$. For each $s \geqslant l$, the following diagram commutes

\[
    \begin{tikzcd}
        \CS_{\leq r}^{\leq s} \arrow[r, "\cong"', "\phi_s"] \arrow[d,"m_l"'] & \pi_0 \fabprk{r}^{\leq s} \arrow[d,"q_{p^l}"]\\
        \CS_{\leq r}^{\leq l} \arrow[r, "\cong"', "\phi_l"] & \pi_0 \fabprk{r}^{\leq l}.
    \end{tikzcd}
    \]
Therefore by taking limits, the maps $\phi_l$ induce homeomorphisms 
    \[
        \lim_l \CS_{\leq r}^{\leq l} \cong \lim_l \pi_0 \fabprk{r}^{\leq l} \cong \pi_0 \hfabprk{r}
    \]
of profinite spaces. Lastly, the maps $m_l$ induce a map $m \colon \hatS_{\leq r} \to  \lim_l \CS_{\leq r}^{\leq l}$, and we will now show that this map is a homeomorphism to conclude the proof of the theorem.

A point of $\lim_l \CS_{\leq r}^{\leq l}$ is represented by a sequence $b$ of points $b^l = (b^l_1,\ldots, b^l_r) \in \CS_{\leq r}^{\leq l}$ compatible with the maps $m_l$. By looking at each of the $r$ coordinates of the points $b_l$ and of points in $\hatS_{\leq r}$ separately we see that the map $m$ is bijective.

A basis of open sets for the topology of $\bbN^+$ is given by the sets of the form $\{a\}$, and intervals $[a, \infty]$ for each $a \in \bbN$. Thus a basis for the topology of $\hatS_{\leq r}$ is given by sets of the form $A=\hatS_{\leq r} \cap (A_1 \times \dots \times A_r)$, where each $A_i$ is either $\{a_i\}$ or $[a_i, \infty]$. Let $A$ be an element of this basis and let $n$ be the maximum of the $a_i$. We note that $A=m_{n+1}^{-1}(S)$ for some subset $S \subseteq \CS_{\leq r}^{\leq n+1}$, and thus the map $m$ sends $A$ to an open subset of $\lim_l \CS_{\leq r}^{\leq l}$. This shows that the map $m$ is open and continuous, and therefore a homeomorphism.
\end{proof}
 
The space $\hatS_{\leq r}$ can also be obtained by iterating disjoint unions and one-point compactifications.

\begin{Cor}\label{cor:abelian-p-groups-rank-r}
$\hatS_{\leq 0} = \{*\}$, and for each $r \geqslant 1$, $\hatS_{\leq r} \cong (\coprod_{n \in \bbN} \hatS_{\leq r-1})^+$.
\end{Cor}
\begin{proof}
    For each $r \geqslant 1$, let $\rho \colon \hatS_{\leq r} \to \bbN^+$ denote the projection to the $r$th coordinate of $(\bbN^+)^r$. Then $\rho^{-1}(\infty)$ consists of a single point, $(\infty, \dots, \infty)$. Each other preimage $\rho^{-1}(n)$ for $n \in \bbN$ is both open and closed by continuity. For each $n \in \bbN$, the map 
    \begin{align*}
        \hatS_{\leq r-1} \quad \to & \quad\rho^{-1}(n) \\
        (v_1,\ldots, v_{r-1}) \; \mapsto & \;(v_1+n,\ldots, v_{r-1}+n, n)
    \end{align*}
    is a homeomorphism. This shows that $\hatS_{\leq r} \setminus \{(\infty, \dots, \infty)\}$ is homeomorphic to $\coprod_{n \in \bbN} \hatS_{\leq r - 1}$ and is open in $\hatS_{\leq r}$. Since $\hatS_{\leq r}$ is compact Hausdorff, it is therefore homeomorphic to the one-point compactification of $\hatS_{\leq r} \setminus \{(\infty, \dots, \infty)\}$.
\end{proof}

\begin{Exa}\label{ex-abelian-p-groups-rank-2}
    Let $ \fabprk{2}$ denote the family of abelian $p$-groups of $p$-rank at most $2$. Then the poset associated to the profinite extension $\hfabprk{2}$ can be depicted as in \cref{fig:prk2}. Therefore the Balmer spectrum $\Spc(\D{\fabprk{2}}^c)$, which is homeomorphic to $\pi_0\hfabprk{2}$ by \cref{Thm-profinite-rgenerated}, can be depicted as in \cref{fig:phenomena2}. Under the homeomorphism to $\hatS_{\leq 2}$ of \cref{Thm-abelian-p-groups-rank-r}, the profinite group prime $\mathfrakp_{\Z_p \times \Z/{p^n}}$ corresponds to the point $(\infty, n) \in \hatS_{\leq 2}$, and $\mathfrakp_{\Z_p \times \Z_p}$ corresponds to $(\infty, \infty) \in \hatS_{\leq 2}$.
\end{Exa} 

We can also obtain a description of the Balmer spectrum for the family $\fabrk{r}$ of $r$-generated finite abelian groups, as a product of the spectra that we just computed for a single prime in \cref{Thm-abelian-p-groups-rank-r}.

\begin{Cor}\label{cor-abelian-groups-rank-r}
    For $r\in \bbN$, the Balmer spectrum for the family $\fabrk{r}$ of $r$-generated finite abelian groups is 
        \[
            \Spc(\sfD(\fabrk{r})^c)  \cong \prod_p \Spc(\sfD(\fabprk{r})^c) \cong \prod_p \hatS_{\leq r}.
        \]
\end{Cor}
\begin{proof}
    The family $\fabrk{r}$ is $r$-submultiplicative. Thus, we can apply \cref{Thm-profinite-rgenerated}. For any prime number $p$, we consider the family $\fabprk{r}$ of finite abelian $p$-groups of $p$-rank at most $r$. Note that $\fabrk{r}$ is not equivalent to the product of the categories $\fabprk{r}$ for varying $p$. It is however equivalent to the \emph{weak product} of the $\fabprk{r}$'s, that is, the full subcategory of $\prod_p \fabprk{r}$ of objects whose projection to $\fabprk{r}$ is the trivial group for all but finitely many $p$.

    The reflections $q_n$ constructed in \cref{defn-std-filt} for the families $\fabrk{r}$ and $\fabprk{r}$ commute with the projections $\fabrk{r} \to \fabprk{r}$. The subcategory $\fabprk{r}[n]$ consists of just the trivial group as long as $p > n$. This means that for each $n$, all but finitely many of the terms $\fabprk{r}[n]$ are trivial, and so $\fabrk{r}[n]$ is equivalent to the product $\prod_p \fabprk{r}[n]$, and not just the weak product. Therefore we obtain homeomorphisms
    \begin{align*}
    \label{eq-cyclicSpc}
        \Spc(\D{\fabrk{r}}^c)  & \cong \lim_n \pi_0 \fabrk{r}[n]
                             \cong \lim_n \prod_p \pi_0 \fabprk{r}[n] &\\
                             & \cong \prod_p \lim_n \pi_0 \fabprk{r}[n] 
                             \cong \prod_p \Spc(\D{\fabprk{r}}^c)
                             \cong \prod_p \hatS_{\leq r},
    \end{align*} 
    from \cref{Thm-profinite-rgenerated} and \cref{Thm-abelian-p-groups-rank-r} as desired.
\end{proof}

\begin{Exa}\label{ex-abelian-p-groups-rank-1}
\Cref{cor-abelian-groups-rank-r} applies in particular to the family $\cat{C} = \fabrk{1}$ of finite cyclic groups, yielding $\Spc(\D{\cat{C}}^c) \cong \prod_p \mathbb{N}^+$.
\end{Exa}

\begin{Rem}
    Due to the product topology, in $\Spc(\D{\fabrk{r}}^c)$ the group prime $\mathfrakp_1$ is an accumulation point of the group primes $\mathfrakp_{C_p}$ for varying $p$. This is the same phenomenon that occurred for the subfamily $\Cprime$ as described in \cref{ex-cyclic-groups-prime-order}.
\end{Rem}

\begin{Rem}
    Note that the proof of the last corollary uses the fact that we are restricting to $r$-generated finite abelian groups. In fact, we do not currently know whether the spectrum $\Spc(\sfD(\fab)^c)$ for the family of all finite abelian groups is the product of the spectra $\Spc(\sfD(\fabp)^c)$ for finite abelian $p$-groups.
\end{Rem}

\subsection{Consequences for Cantor--Bendixson rank}\label{ssec:cbrank}

We end this section by showing that the Balmer spectrum for the family of finite abelian $p$-groups has infinite Cantor--Bendixson rank. For convenience, we briefly recall the definition of Cantor--Bendixson rank; for further details, we refer the reader to~\cite[Section 4.3]{DST2019}.

\begin{Rec}\label{rec:CBrank}
    Let $X$ be a topological space. The \emph{Cantor--Bendixson derivative} $\delta X$ of $X$ is defined as the subspace of non-isolated points in $X$. For an ordinal $\alpha$, one then uses transfinite recursion to define $\delta^{\alpha}X$ via
        \[
            \delta^{\alpha}X \coloneqq
                \begin{cases}
                    X & \text{if } \alpha = 0; \\
                    \delta\delta^{\alpha -1}X & \text{if } \alpha \text{ is a successor ordinal}; \\
                    \bigcap_{\beta < \alpha} \delta^{\beta}X & \text{if } \alpha \text{ is a limit ordinal}.
                \end{cases}
        \]
    Finally, the \emph{Cantor--Bendixson rank} of a non-empty space $X$ is defined as
        \[
            \CB(X) \coloneqq 
                \begin{cases}
                    \sup\SET{\alpha}{\delta^{\alpha}X \neq \emptyset} & \text{if the supremum exists}; \\
                    \infty & \text{otherwise},
                \end{cases}
        \]
    while $\CB(\emptyset) \coloneqq -1$ by definition. Here, we use the convention that $\infty$ is larger than any ordinal.
\end{Rec}

As a consequence of \cref{cor:abelian-p-groups-rank-r}, we obtain:

\begin{Cor}\label{cor:CBrank_boundedrank}
    For any $r \geq 0$, the spectrum $\Spc(\D{\fabprk{r}}^c)$ has Cantor--Bendixson rank precisely $r$.
\end{Cor}
\begin{proof}
    We argue by induction on $r$. When $r=0$, the spectrum is a point, thus providing the base of the induction. For the induction step, we use the description of \cref{cor:abelian-p-groups-rank-r} to see that for $r>0$:
        \[
            \CB(\hatS_{\leq r}) = \CB\left (\left (\coprod_{n \in \bbN}\hatS_{\leq r-1}\right )^+\right) = \CB(\hatS_{\leq r-1})+1 = r,
        \]
    by induction hypothesis. This implies the claim in light of \cref{Thm-abelian-p-groups-rank-r}.
\end{proof}

\begin{Cor}\label{coro-CB-rank}
    For any family $\cat U \subset \cat G$ containing the family of all finite abelian $p$-groups $\cat{A}(p)$, the Cantor--Bendixson derivative $\delta^{\omega}\Spc(\D{\cat U}^c) \neq \emptyset$. In particular, $\CB(\Spc(\D{\cat U}^c))\geqslant \omega$.
\end{Cor}
\begin{proof}
     For each $r$, the inclusion $\fabprk{r} \subset \cat{A}(p) \subseteq \cat{U}$ induces a map on Balmer spectra, which is an embedding by \cref{cor:verdierquotient}, as $\fabprk{r}$ is downwards closed. Recall that embeddings of topological spaces give an inequality on Cantor--Bendixson ranks, see~\cite[4.3.2(2)]{DST2019}. \Cref{cor:CBrank_boundedrank} thus implies that $\delta^r\Spc(\D{\cat U}^c) \neq \emptyset$ for all $r \in \bbN$. Since the spectrum is quasi-compact, we deduce that their intersection is non-empty as well, i.e., $\delta^{\omega}\Spc(\D{\cat U}^c) \neq \emptyset$. Therefore, the Cantor--Bendixson rank of $\Spc(\D{\cat{U}}^c)$ is $\geqslant \omega$. 
\end{proof}

\newpage
\addcontentsline{toc}{part}{Appendices}
\appendix

\appendix

\section{Rational global homotopy theory}\label{appendix}

In this appendix we introduce the $\infty$-category of global spectra following the approach of \cite{LNP} and record some properties that it enjoys. We then introduce the category of rational global spectra as a Verdier quotient and show that this construction agrees with the original definition of Schwede given in \cite{Schwedebook}.

\subsection{Partially lax limits}
The definition of the category of global spectra will make use of the notion of partially lax limit. We will recall some of the key features of this construction and refer the interested reader to \cite{Berman}, \cite{GHN}, \cite{LNP} and \cite{GlobalTMF} for a more detailed account of this material. We will freely use the language of $\infty$-categories as developed in \cite{HTT,HA}.

\begin{Rec}
    Given functors $F,G\colon \cat I \to \Cat_\infty$, there is a notion of lax natural transformation $\eta \colon F \Rightarrow G$. Informally, this is a coherent collection of squares
    \[
    \begin{tikzcd}
        {F(i)} &  {F(j)} \\
        {G(i)} & {G(j)}
        \arrow["\eta_j",from=1-2, to=2-2]
        \arrow["F(f)", from=1-1, to=1-2]
        \arrow["G(f)"', from=2-1, to=2-2]
        \arrow["\eta_i"', from=1-1, to=2-1]
        \arrow["\epsilon", shorten <=7pt, shorten >=7pt, Rightarrow, from=2-1, to=1-2]
    \end{tikzcd}
    \]
    together with 2-morphisms $\epsilon\colon G(f) \circ \eta_i \Rightarrow \eta_j\circ  F(f) $. If the 2-morphisms $\epsilon$ are all equivalences, then $\eta$ defines a natural transformation.
\end{Rec}

\begin{Def}\label{def-laxlim}
    We define the \emph{lax limit} of a functor $F \colon \cat I \to \Cat_\infty$ to be the $\infty$-category of lax cones
    \[
    \laxlim_{\cat I} F \coloneqq \mathrm{Nat}^{\mathrm{lax}}(\Delta(\ast), F),
    \]
   that is, lax natural transformations from the constant functor on the terminal $\infty$-category to $F$. If $\cat I$ is marked by a subcategory $\cat W$, we define the \emph{partially lax limit} of $F$
\[
\laxlimdag_{\cat I,\cat W} F \subset \laxlim_{\cat I} F 
\]
to be the full subcategory spanned by those lax cones whose restriction to $\cat W\subset \cat I$ is strictly natural.
\end{Def}

\begin{Rem}\label{rem-laxlim-sections}
It follows from \cite[Remark 2.7]{GlobalTMF} that there is an equivalence 
\[
\laxlim_{\cat I} F \simeq \Fun_{\cat I}(\cat I,\mathrm{Un}^{\mathrm{co}}{F}).
\]
In other words, the lax limit of $F$ is equivalent to the $\infty$-category of sections of the cocartesian unstraightening of $F$. Furthermore, if $\cat I$ is marked by $\cat W$, then a lax cone is in $\laxlimdag F$ if and only if the associated section sends maps in $\cat W$ to cocartesian edges in $\mathrm{Un}^{\mathrm{co}}{F}$.
\end{Rem}

\begin{Exa}\label{ex-constant-laxlim}
    Consider the constant functor $\Delta(\cat C)\colon \cat I \to \Cat_\infty$ on an $\infty$-category $\cat C$. Its cocartesian unstraightening is given by the projection $\cat C \times \cat I \to \cat C$ and so it follows from \cref{rem-laxlim-sections} that 
    \[
    \laxlim_{\cat I} \Delta(\cat C)=\Fun_{\cat I}(\cat I, \cat C \times \cat I)\simeq\Fun(\cat I, \cat C).
    \]
\end{Exa}
We will also need to discuss how partially lax limits behave with respect to symmetric monoidal structures. 

\begin{Rec}\label{rec-laxlim-sym-mon}
     Given a functor $F \colon \cat I \to \CAlg(\PrL)$ and a subcategory $\cat W\subseteq \cat I$, we can form the partially lax limit marked at $\cat{W}$ 
     \[
     \laxlimdag_{\cat I, \cat W} F\in \CAlg(\PrL),
     \]
     see \cite[Section 5]{LNP}.
     This construction defines a functor 
     \[
      \laxlimdag_{\cat I, \cat W} \colon \Fun(\cat I, \CAlg(\PrL)) \to \CAlg(\PrL).
     \]
     By the discussion in \cite[Remarks 5.1 and 5.2]{LNP}, the symmetric monoidal structure on the partially lax limit has the following universal property: for any $\cat D \in \CAlg(\PrL)$, postcomposition with the projection functors induces an equivalence
     \[
     \Fun^{\mathrm{L},\otimes}(\cat D, \laxlimdag_{\cat I, \cat W}F)\xrightarrow{\sim} \laxlimdag_{i\in \cat I, \cat W} \Fun^{\mathrm{L}, \otimes}(\cat D, F(i)),
     \]
     where $\Fun^{\mathrm{L}, \otimes}$ denotes colimit preserving and symmetric monoidal functors. In particular, the inclusion 
     \begin{equation}\label{eq-inclusions}
     \laxlimdag_{\cat I, \cat W} F \subseteq \laxlim_{\cat I} F
     \end{equation}
     is symmetric monoidal. 
\end{Rec}

\subsection{Global spectra}
We are finally ready to introduce the $\infty$-category of global spectra. Unlike the rest of the paper where we restrict to families of finite groups, several of the definitions and constructions in this appendix work more generally for global families $\cat F$ of compact Lie groups, i.e., collections of compact Lie groups closed under subgroups and quotients, so we will write this appendix in this generality. Throughout all subgroups will be implicitly assumed to be closed.
\begin{Def}\label{def-global-orbits}
Let $\cat F$ be a global family of compact Lie groups.
\begin{enumerate}
\item We let $\Glo{\cat F}$ denote the \emph{global orbit} $\infty$-\emph{category} of \cite[Definition 6.1]{LNP} with isotropy in $\cat F$, whose objects are the groups $G\in\cat F$. In this category, the space of morphisms $ H\to  G$ is the homotopy orbit space $\Hom(H,G)_{hG}$, where $\Hom(H,G)$ is the space of continuous homomorphisms from $H$ to $G$, and $G$ acts on this by conjugation.  
\item We denote by $\Orb{\cat F}\subseteq \Glo{\cat F}$ the wide subcategory spanned by those morphisms which are represented by an injective group homomorphism. 
\end{enumerate}
\end{Def}

\begin{Rem}\label{rem-homotopy-category}
    Given $G,H\in \cat F$, there is a decomposition 
    \[
    \Map_{\Glo{\cat F}}(G,H)\simeq \coprod_{[\alpha]\in\Map_{\cat F}(G,H)} B C(\alpha),
    \]
    where $C(\alpha)$ denotes the centralizer of the image of $\alpha$, see \cite[Proposition 2.5]{Korschgen}. It follows that the morphisms in $\Ho(\Glo{\cat F})$ from $H$ to $G$ are given by conjugacy classes of group homomorphisms $\alpha\colon H\to G$. 
\end{Rem}

\begin{Rec}
    For any compact Lie group $G$, we write $\Sp_G$ for the $\infty$-category of genuine $G$-spectra: for an explicit model we can take the underlying $\infty$-category associated to the category of orthogonal $G$-spectra (indexed on a complete $G$-universe) with the stable model structure, see \cite[Chapter III.4]{MM}. The $\infty$-category $\Sp_G$ is stable and presentably symmetric monoidal with symmetric monoidal structure induced by the smash product of orthogonal $G$-spectra \cite[Chapter II.3]{MM}. We refer the reader to \cite{LMSM} and \cite[Chapter 3]{Schwedebook} for more details on equivariant stable homotopy theory.
\end{Rec}

\begin{Cons}\label{cons-functor-Spbullet}
    By \cite[Proposition 10.4 and Corollary 10.6]{LNP}, there is a functor 
    \[
    \Sp_\bullet \colon \Glo{\cat F}^{\op} \to \CAlg(\PrL)
    \]
    which sends $G$ to the $\infty$-category $\Sp_G$ of genuine $G$-spectra, and a morphism $ \alpha\colon  H \to G$ to the restriction-inflation functor $\alpha^* \colon \Sp_G \to \Sp_H$.
\end{Cons}

\begin{Def}\label{def-global-spectra}
    Let $\cat F$ be a global family of compact Lie groups. We define the $\infty$-category of $\cat F$-\emph{global spectra} $\Spgl{\cat F}$ as the partially lax limit of the functor $\Sp_\bullet$ marked at $\Orb{\cat F}$, and we will write
    \[
    \Spgl{\cat F}\coloneqq \laxlimdag_{\Glo{\cat F}^{\op}, \Orb{\cat F}^{\op}} \Sp_\bullet \in \CAlg(\PrL).
    \]
    Informally, we can think of a global spectrum $X \in\Spgl{\cat F}$ as consisting of the following data: 
    \begin{itemize}
        \item a genuine $G$-spectrum $\res_G X$ for all $G \in \cat F$;
        \item compatible maps $f_\alpha \colon \alpha^* \res_GX \to \res_H X$ in $\Sp_H$ for all continuous group homomorphism $\alpha \colon H \to G$;
    \end{itemize}
    subject to the condition that $f_\alpha$ is an equivalence whenever $\alpha$ is injective. 
\end{Def}

\begin{Rem}\label{rem-agrees-schwede}
    By the main result of \cite{LNP}, this model of global spectra agrees with Schwede's model constructed via the global model structure on orthogonal spectra; also see the remark at the beginning of part II in \cite{LNP}. Therefore we are free to use and cite all the results from \cite{Schwedebook}. We also note that the category of global spectra, with respect to a global family of finite groups, admits a model in terms of spectral Mackey functors, see \cite[Theorem A]{Lenz}.
\end{Rem}

\begin{Rem}\label{rem-symm-mon}
    One of the advantages of describing global spectra as a partially lax limit is that we obtain a symmetric monoidal structure without needing to restrict attention to multiplicative global families. The discussion in \cref{rec-laxlim-sym-mon} shows that the symmetric monoidal structure on $\Spgl{\cat F}$ is such that the projection functors $\Spgl{\cat F}\to \Sp_G$ are symmetric monoidal and colimit preserving for all $G\in \cat F$. 
\end{Rem}

We next relate global spectra to the category of genuine $G$-spectra and then discuss fixed points functors in this setting. We will use this material later to show that the category of global spectra is compactly generated.

\begin{Cons}\label{cons-restrictions}
     We just noted that for any $G \in \cat F$, there is a symmetric monoidal colimit preserving functor
    \[
    \res_G \colon \Spgl{\cat F}\to \Sp_G 
    \]
    with right adjoint denoted by $R_G$.
    As limits in a partially lax limit are calculated pointwise, we deduce that $\res_G$ also preserves all limits, and that it admits a left adjoint $L_G$. We therefore have an adjoint triple
    \[
    \begin{tikzcd}[column sep=large, row sep=large]
        \Spgl{\cat F} \arrow[r,"\res_G" description] & \Sp_G, \arrow[l, yshift=2mm, "L_G"'] \arrow[l, yshift=-2mm, "R_G"]
    \end{tikzcd}
    \]
    where $L_G$ and $R_G$ are canonically oplax and lax symmetric monoidal, respectively. We warn the reader that the functors $L_G$ and $R_G$ are in general not fully faithful. However by \cite[Theorem 4.5.1]{Schwedebook} they are fully faithful if $G$ is the trivial group $1$. 
\end{Cons}

\begin{Cons}\label{cons-geometric-fixed-points}
    Recall that for any compact Lie group $G$, there is a $G$-fixed points functor $(-)^G \colon \Sp_G \to \Sp$ and a $G$-geometric fixed points functor $\Phi^G \colon \Sp_G \to \Sp$, see for instance \cite{MM} and \cite{LMSM}. The latter functor is symmetric monoidal and colimit preserving. We can construct an explicit model as follows.  We let $\mathcal{P}_G$ be the family of proper subgroups of $G$, and recall that there is an essentially unique $G$-space $E\mathcal{P}_G$ such that $(E\mathcal{P}_G)^G=\emptyset$ and $(E\mathcal{P}_G)^H$ is contractible for all $H<G$.  If we write $\widetilde{E}\mathcal{P}_G$ for the unreduced suspension of $E\mathcal{P}_G$, then we have
    \[  \Phi^G X\simeq (X \otimes \widetilde{E}\mathcal{P}_G)^G. \]
    If the group $G$ belongs to our global family $\cat F$, we can consider the composites
    \[
    \Spgl{\cat F} \xrightarrow{\res_G}\Sp_G \xrightarrow{(-)^G} \Sp \;\;\mathrm{and} \;\;\Spgl{\cat F} \xrightarrow{\res_G}\Sp_G \xrightarrow{\Phi^G} \Sp
    \] 
    which by a slight abuse of notation we still denote by $(-)^G$ and $\Phi^G$, respectively. 
\end{Cons}

\begin{Lem}\label{lem-jointly-conservative}
 The following collections of functors are jointly conservative:
  \begin{enumerate}
      \item $\{\res_G \colon \Spgl{\cat F}\to \Sp_G\}_{G \in \cat F}$;
      \item $\{(-)^G\colon \Spgl{\cat F}\to \Sp\}_{G\in \cat F}$;
      \item $\{\Phi^G\colon \Spgl{\cat F}\to \Sp\}_{G \in\cat F}$. 
  \end{enumerate}
\end{Lem}

\begin{proof}
    Part (a) is clear by the definition of partially lax limits. For part (b) assume that $(\res_G X)^G\simeq 0$ for all $G\in\cat F$; we want to show that $X \simeq 0$. By part (a) it suffices to show that $\res_G X \simeq 0$ for all $G\in \cat F$. In other words, we need to prove that $(\res_G X)^H\simeq 0$ for any $G \in \cat F$ and subgroup $i_H \colon H \leq G$.  
    To this end recall from \cref{def-global-spectra} that the global spectrum $X$ has a structure map $f_{i_H}\colon i^*_H \res_G X \xrightarrow{\sim} \res_H X$. Using this and our assumption, we see that
    \[
    (\res_G X)^H \simeq (i^*_H \res_G X)^H \xrightarrow[\sim]{(f_{i_H})^H} (\res_H X)^H \simeq 0,
    \]
    as required. The argument for (c) is similar to that of (b); the key ingredient is that for a fixed $G \in \cat F$, the family of functors $\{\Phi^H \res^G_H \colon \Sp_G \to \Sp_H \to \Sp\}_{H \leq G}$ is jointly conservative, see for instance \cite[Proposition 6.13]{MNN}.
\end{proof}

\begin{Prop}\label{prop-global-compact-generation}
    Let $\cat F$ be a global family of compact Lie groups. Then $\Spgl{\cat F}$ is compactly generated with a set of compact generators given by $\{L_G \unit \mid G \in\cat F\}$.
\end{Prop}

\begin{proof}
    The objects $L_G \unit$ are compact since $\res_G$ preserves all colimits. To show that they generate assume that $0 \simeq \Map_{\Spgl{\cat F}}(L_G \unit, X) \simeq (\res_GX)^G$ for all $G \in \cat F$; we want to show that $X \simeq 0$. This now follows from \cref{lem-jointly-conservative}(b). 
\end{proof}

\subsection{Rational global spectra}
Now we are finally ready to introduce the main category of interest for this section. 
\begin{Def}\label{def-rational-gl-sp}
    For a prime number $p$, write $\unit{\sslash}p\coloneqq\mathrm{cone}(\unit \xrightarrow{\cdot p} \unit ) \in \Spgl{\cat F}$. We define the $\infty$-category of \emph{rational} $\cat F$-\emph{global spectra} $\Spgl{\cat F}^\bbQ$ as the Verdier quotient of $\Spgl{\cat F}$ by the localizing ideal generated by 
    \[
        \{\unit{\sslash}p \mid p\; \mathrm{prime}\}\subseteq \Spgl{\cat F}.
    \]
\end{Def}

\begin{Rem}\label{rem-quotient-monoidal}
    As the $\infty$-category of rational global spectra is defined as a Verdier quotient by a localizing ideal, it follows from \cite[Proposition 2.2.1.9]{HA} that it admits a symmetric monoidal structure making the functor $L_\bbQ \colon \Spgl{\cat F}\to \Spgl{\cat F}^\bbQ$ into a symmetric monoidal functor of presentably symmetric monoidal $\infty$-categories; in particular $\Spgl{\cat F}^\bbQ \in \CAlg(\PrL)$.
\end{Rem}

Schwede in \cite[Theorem 4.5.28]{Schwedebook} defines the category of rational global spectra as the Bousfield localization of $\Spgl{\cat F}$ at the class of \emph{rational} $\cat F$-\emph{global equivalences}, namely those maps $f \colon X \to Y$ such that
\[
\pi_*^G(\res_G(f))\otimes \bbQ \colon \pi_*^G(\res_G(X))\otimes \bbQ \to \pi_*^G(\res_G(Y)) \otimes \bbQ
\]
is an equivalence for all $G \in \cat F$. The next result reconciles the two definitions.

\begin{Prop}\label{prop-Qglsp-bousfield}
    The $\infty$-category $\Spgl{\cat F}^\bbQ$ is the Bousfield localization of $\Spgl{\cat F}$ at the class of rational $\cat F$-global equivalences. In particular $X \in \Spgl{\cat F}$ belongs to $\Spgl{\cat F}^\bbQ$ if and only if $\pi_*^G(\res_G X)$ is rational for all $G \in \cat F$.
\end{Prop}

\begin{proof}
    Set $\cat{L} = \loct{\unit{\sslash}p \mid p \; \mathrm{prime}}$ so that $\Spgl{\cat F}^{\bbQ}\simeq \Spgl{\cat F}/ \cat{L}$, and write $L_\bbQ\colon \Spgl{\cat F}\to \Spgl{\cat F}^\bbQ$ for the quotient functor. Then by~\cite[Proposition 5.6]{BlumbergGepnerTabuada} the $\infty$-category $\Spgl{\cat F}^\bbQ$ is the Bousfield localization of $\Spgl{\cat F}$ at the class $S$ consisting of those maps whose cofibre lies in $\cat{L}$. 
    
    By definition, an object $Z$ is $S$-local if and only if $\Map_{\Spgl{\cat F}}(L,Z) \simeq 0$ for all $L \in \cat{L}$, if and only if $\iHom(\unit{\sslash}p, Z) \simeq 0$ for all $p$ prime, where $\iHom$ denotes the internal hom in $\Spgl{\cat F}$.  By applying $\iHom(-,Z)$ to the cofibre sequence $\unit\xrightarrow{p}\unit\to\unit{\sslash}p$ we see that $\iHom(\unit{\sslash}p, Z)=\text{fibre}(p \cdot 1_Z)$, and it follows that $Z$ is $S$-local if and only if every prime number $p$ acts invertibly on $\res_G Z$ for all $G \in \cat F$. This shows that $Z$ is $S$-local if and only if $\pi_*^G(\res_G Z)$ is rational for all $G\in \cat F$. Finally note that a map $f \colon X \to Y$ in $\Spgl{\cat F}$ is an $S$-local equivalence if and only if it is sent to an equivalence by $L_\bbQ$, which in turn is equivalent to $\pi_*^G(\mathrm{cone}(f))$ being torsion for all $G \in \cat F$. 
\end{proof}

\begin{Exa}\label{ex-Q-glo-sphere}
    Consider the following colimit
    \[
    \unit_\bbQ\coloneqq \colim (\unit \xrightarrow{2}\unit \xrightarrow{3}\unit \xrightarrow{4} \unit \xrightarrow{5} \unit \to \ldots) \in \Spgl{\cat F}.
    \]
     Since the functor $\pi_*^G(\res_G(-))$ commutes with filtered colimits, we see that 
     \[
     \pi_*^G(\res_G \unit_\bbQ)\cong\pi_*^G(\res_G \unit)\otimes \bbQ.
     \]
     It then follows that $\unit_\bbQ=L_\bbQ(\unit)$, the rationalization of the global sphere. Finally, we observe that $\Phi^G \unit_\bbQ \in \Sp$ is the rational sphere for all $G\in \cat F$.
\end{Exa}

\begin{Cor}\label{cor-smashing+compactly-gen}
    The rationalization functor $L_\bbQ \colon \Spgl{\cat F} \to \Spgl{\cat F}$ is a smashing localization, that is $L_\bbQ(-) \simeq \unit_\bbQ \otimes -$. In particular, the $\infty$-category $\Spgl{\cat F}^\bbQ$ is compactly generated by 
    \[
    \{L_G^\bbQ \unit \coloneqq L_\bbQ L_G \unit \mid G \in \cat F\}.
    \]
\end{Cor}
\begin{proof}
    A similar calculation as in \cref{ex-Q-glo-sphere} shows that $X \otimes \unit_\bbQ$ has rational $G$-homotopy groups for all $G \in \cat F$, and that the canonical map $X \to\unit_\bbQ \otimes X$ is a rational $\cat F$-global equivalence for all $X \in \Spgl{\cat F}$. By the universal property of Bousfield localization, we deduce that $L_\bbQ \simeq \unit_\bbQ \otimes -$. From this we deduce that the inclusion $\Spgl{\cat F}^\bbQ\subseteq \Spgl{\cat F}$ preserves all colimits, and so the left adjoint $L_\bbQ \colon \Spgl{\cat F} \to \Spgl{\cat F}^\bbQ$ preserves compact generators, giving the final claim (see \cref{prop-global-compact-generation}).
\end{proof}

\begin{Lem}\label{lem-geom-fix-rational-eq}
    Consider a map $f \colon X \to Y$ in $\Spgl{\cat F}$. Then $f$ is a rational $\cat F$-global equivalence if and only if $\Phi^G f$ is a rational equivalence for all $G \in \cat F$.
\end{Lem}

\begin{proof}
      Write $C$ for the cone of $f$, so that $f$ is a rational $\cat F$-global equivalence if and only if
      $L_\bbQ C\simeq \unit_\bbQ \otimes C \simeq 0$. Using \cref{lem-jointly-conservative}(b), this is equivalent to $\Phi^G \unit_\bbQ \otimes \Phi^G C \simeq 0$ for all $G \in \cat F$. By \cref{ex-Q-glo-sphere} that $\Phi^G \unit_\bbQ$ is the rational sphere, so $\Phi^G \unit_\bbQ \otimes \Phi^G C \simeq 0$ is equivalent to $\Phi^G C$ being rationally zero for all $G \in \cat F$.
\end{proof}

\bibliographystyle{alpha}
\bibliography{reference}

\end{document}